\documentclass[11pt,twoside]{article}
\usepackage{amssymb}
\usepackage{amsmath}
\usepackage{amsthm}
\usepackage{amsfonts}
\usepackage{color}
\usepackage{enumerate}
\usepackage{mathrsfs}
\usepackage{bm}
\usepackage{epsfig}
\usepackage{graphics}
\usepackage{graphicx}
\usepackage[abs]{overpic}
\usepackage{cite}
\usepackage{verbatim}

\usepackage{indentfirst}

\usepackage[T1]{fontenc}
\usepackage[utf8]{inputenc}
\usepackage{authblk}
\usepackage{booktabs}
\usepackage[driverfallback = dvipdfm,
           bookmarksnumbered=true,
           bookmarksopen=true,
              colorlinks,
            linkcolor=red,
            anchorcolor=blue,
            citecolor=green]{hyperref}

\DeclareMathOperator{\esssup}{ess\ sup}
\DeclareMathOperator{\essinf}{ess\ inf}

\newtheorem{theorem}{Theorem}
\theoremstyle{definition}
\newtheorem{definition}[theorem]{Definition}
\newtheorem{remark}{Remark}
\newtheorem{lemma}{Lemma}
\newtheorem{proposition}[theorem]{Proposition}

\newtheorem{corollary}[theorem]{Corollary}

\providecommand{\keywords}[1]{\textbf{\textit{Keywords---}} #1}

\allowdisplaybreaks

\begin{document}
\title{Well-posedness for $p(x)$-Laplacian parabolic equations under a multiple-regime setting on an annulus}

\author[2,4,5]{Yitian Wang}
\author[1,4,5]{Runzhang Xu}
\author[1,3]{Chao Yang\thanks{Corresponding author: cyang@ims.cuhk.edu.hk, chaoyang@agh.edu.pl}}

\affil[1]{The Institute of Mathematical Sciences, The Chinese University of Hong Kong, Shatin, NT, Hong Kong, China}
\affil[2]{School of Mathematical Sciences, Shanghai Jiao Tong University, Shanghai 200240, China}
\affil[3]{Faculty of Applied Mathematics, AGH University of Krakow, Krakow {\rm 30-059}, Poland}
\affil[4]{College of Mathematical Sciences, Harbin Engineering University, 150001 Harbin, China}
\affil[5]{Sanya Nanhai Innovation and Development Base of Harbin Engineering University, Sanya, China}

\date{}
\maketitle

\begin{abstract}
    This paper proposes the energy-variational framework to deal with the well-posedness of the solution to the initial boundary value problem of the parabolic equation with variable exponents on an annulus, under the multiple-regime setting consisting of the subcritical, critical and supercritical regimes of the variable exponents.
    The multiple-regime setting leads to the appearance of the critical and supercritical regions of positive measure, which cause the failure of the classical Sobolev embedding and the lack of compactness.
    To deal with the potential destructiveness of the multiple-regime phenomenon towards the well-posedness of the solution under the energy framework, we recover the compact Sobolev embedding for a class of radially symmetric functions by the Strauss inequality and obtain the local existence of the solution via the combination of the subdifferential technique and the Principle of Symmetric Criticality for any given radially symmetric initial data.
    Additionally, the uniqueness result is also given with the recovered Sobolev embedding.
    Moreover, examining the impact on well-posedness of solutions caused by the variable exponents crossing three sensitive regions, we establish a variational structure of the potential well theory for the evolution equation under the multiple-regime setting on the annulus.
    For subcritical and critical initial energy levels, we discuss for the non-positive and positive initial energy cases separately. We obtain the finite-time blow-up of the solution with non-positive initial energy and the threshold condition for global existence and finite-time blow-up of the solution with positive initial energy.
    For supercritical initial energy level, several sufficient criteria for the global existence and finite-time blow-up of the solution are given.
    Finally, the above results (including the global existence and finite-time blow-up of the solution) are generalized to the case of the more general locally symmetric domain containing an annulus.
\end{abstract}


\tableofcontents

\section{Introduction}
In this paper, we study the well-posedness of the solution to the following initial boundary value problem  (IBVP) of the quasilinear $p(x)$-Laplacian parabolic equation
    \begin{align}
          &u_t-\Delta_{p(x)} u=|u|^{q(x)-2} u, && \text { in } \mathcal{A} \times(0,+\infty), \label{1}\\
          &u=0, && \text { on } \partial \mathcal{A} \times(0,+\infty), \label{2}\\
          &u(x, 0)=u_0(x), && \text { in } \mathcal{A}, \label{3}
    \end{align}
where $\mathcal{A} \subset \mathbb{R}^n$ ($n\ge 2$) is an annulus centered at the origin defined by
    \begin{align*}
      \mathcal{A}:=\left\{x \in \mathbb{R}^n: r<|x|<R\right\}
    \end{align*}
and the variable exponents $p, q: \bar{\mathcal{A}} \rightarrow \mathbb{R}$ are given continuous functions satisfying some specific constraints/hypotheses given in Subsection \ref{sec2.1}.
Furthermore, we also assume throughout this paper that the exponent $p(x)$ satisfies the so-called log-H\"{o}lder condition, i.e.,
    \begin{align}\label{4}
      \left|p(x)-p\left(y\right)\right|
      \leq \frac{C}{ \log \left( e + \frac{1}{|x-y|}\right)} \quad \text { for any } x, y \in \mathcal{A}
    \end{align}
with some constant $C>0$, which is essential for some properties of the Sobolev space with variable exponents, including Poincar\'{e} and Sobolev inequalities.

Equation \eqref{1} can be commonly found in mathematical models that describe a variety of physical phenomena and applications in reality, including the field of image restoration \cite{pyb7,pyb8,pyb9}, as well as the modeling of electrorheological fluids \cite{pyb4,pyb5,pyb6}.
As the feature of \eqref{1}, the variable exponents in \eqref{1} are functions of the space variable $x$, and
this kind of model is relatively complex in both the physical and mathematical points of view.
Then, a natural strategy is to consider the simplified version,
namely the corresponding constant exponent case as a compromise.
Before the variable exponent cases were widely considered, many works focused on the constant exponent case. For relevant studies, we refer the readers to \cite{r9,ap1,ap4,ap3} for parabolic equations, and to \cite{aw1,aw2} for wave equations.
Particularly, when $p(x)$ and $q(x)$ are constants in \eqref{1}, i.e., $ p(x) \equiv p$ and $ q(x) \equiv q$, extensive research has been conducted on the following equation, whose nonlinear features are characterized by the exponents $p,q$:
	\begin{align}\label{44}
		 u_t - \Delta_p u = u^{q-1}
	\end{align}
on a bounded domain $\Omega \subset \mathbb{R}^n$ with appropriately smooth boundary $\partial \Omega$ or on the whole space $\mathbb{R}^n$, including investigations into the topic of well-posedness \cite{r13,r18,r19,r22,r24,r25,r26,r23}.
Under the following assumptions on the exponents,
    \begin{align}
      & 2< p < q, \label{30}\\
      & q<{{p}^{*}} :=\frac{np}{n-p},\,p<n, \label{26}
    \end{align}
Tsutsumi \cite{r22} considered the well-posedness of the solution to the IBVP of \eqref{44} by the potential well theory and obtained a sharp condition that the solution exists globally or blows up in finite time with the subcritical initial energy.
Actually, the assumption \eqref{30} means that the diffusion of \eqref{44} is of the slow type and the source dominates the diffusion, which is consistent with the hypotheses $(\mathrm{H_{p,q}})$ and $(\mathrm{H^2_{p,q}})$ in this paper; and the assumption \eqref{26} means the subcritical regime for the index $q$ of the nonlinear source term in the sense of Sobolev embedding with the constant exponents, i.e.,
    \begin{align}\label{73}
      W_{0}^{1,p}(\Omega ) \hookrightarrow {L}^{q}(\Omega ),
    \end{align}
where $p$ stands for the growth of the diffusion term.
Throughout this paper, the notions of subcritical, critical and supercritical regime always appear in the same sense.

It is well known that the well-posedness of the solution to PDEs can be affected by the structure of the equation, exponents of the nonlinear structural terms, and the initial data.
The constraint \eqref{26} on $q$ of subcritical regime is a standard condition frequently adopted within the energy-variational framework for the study of well-posedness problem of PDEs, which ensures that the embedding \eqref{73} is compact and the energy function of \eqref{44} is well-defined on ${{W}_{0}}^{1,p}(\Omega )$. However, for the critical regime $q={{p}^{*}}$, we lose the compactness of \eqref{73}. For the supercritical regime $q>{{p}^{*}}$, the integrability of the nonlinear source term $u^{q-1}$ in \eqref{44} is even out of the control of the diffusion part, namely $W^{1,p}_{0}(\Omega )$ no longer belongs to ${L}^{q}(\Omega )$.
Furthermore, the critical and supercritical exponent conditions, both of which are unfavorable, may obstruct the well-posedness of solutions.
Specifically, some works considered the effects of the critical and supercritical regime of \eqref{73}, i.e., $q={{p}^{*}}$ and $q>{{p}^{*}}$ on the corresponding well-posedness problem for the class of equations represented by \eqref{44} for $p=2$ (namely $\Delta_p$ degenerates to the classical Laplacian) with $H_0^1(\Omega)$-data on general bounded domains, whose nonlinearity possesses the polynomial growth characterized by the exponent $q$:
    \begin{itemize}
      \item
      The well-posedness of the solution was proved to be still achievable for the subcritical and critical regime $q \le 2^*$ in \cite{r32+}, where $2^* := \frac{2n}{n-2}$ is the critical exponent of the Sobolev embedding $H_0^1(\Omega) \hookrightarrow L^q(\Omega)$.

      \item
      Particularly, for the critical regime $q = 2^*$, due to the failure of the $(PS)$ condition caused by the lack of the compactness of the Sobolev embedding, Tan \cite{p8} selected $\frac{1}{n} S^{\frac{n}{2}}$ as the potential well depth based on the important observation of Brezis and Nirenberg \cite{e1} rather than the mountain pass level (which is not achievable in the critical regime case) to establish the corresponding potential well theory and obtain the threshold condition of the global existence and the finite blowup of the solution to the IBVP of \eqref{44} with subcritical initial energy, where $S$ is the best constant of the embedding $H_0^1(\Omega) \hookrightarrow L^{2^*}(\Omega)$.

    \item
    But in the critical case $q = 2^*$, the solution is not uniformly well-posed \cite{r32}, since the maximal existence time of the solution is only uniform on a neighborhood of each initial data rather than some bounded sets of the initial data for the subcritical regime case, and the infinite-time $L^{\infty}$-blow-up of the solution occurs \cite{r35}.

    \item
    For the supercritical regime $q > 2^*$, the authors of \cite{r32} showed that the solution is even no longer well-posed by constructing sequences of smooth nonnegative initial data $u_{0}^k$ that converge to a nonnegative initial data $u_0$ such that the maximal existence time of their solutions approach zero as $k \rightarrow \infty$.
    \end{itemize}

As the theory of variable exponent function spaces (such as Musielak--Orlicz spaces) has matured, numerous authors have extended classical results from the constant exponent setting to the variable exponent framework, thereby enabling a more accurate mathematical description of physical phenomena in non-homogeneous media; in particular, when the nonlinear source term in the \(p\)-Laplace equation involves variable exponents, relevant contributions can be found in \cite{r28,r30,r31} and the references therein, while for parabolic equations involving the \(p(x)\)-Laplacian, we refer to \cite{r38,r39,r40,r1}, and for wave equations with variable exponents, related works are presented in \cite{aw3,aw4}.
Almost all existing studies mentioned above on variable exponent models within the energy-variational framework are confined to the subcritical regime $q(x)<p^*(x)$ under the variable exponent Sobolev embedding
    \begin{align}\label{77}
        {{W}_{0}}^{1,p(\cdot )}(\Omega ) \hookrightarrow {{L}^{q(\cdot)}}(\Omega ),
    \end{align}
whose critical exponent ${p}^{*}(x)$ is defined by
    \begin{align}\label{165}
    \begin{aligned}
      p^{*}(x):=
        \begin{cases}\frac{np(x)}{n-p(x)} & \text { if } 1<p(x)<n, \\
            \infty & \text { if } p(x) \geq n.
        \end{cases}
    \end{aligned}
    \end{align}
Restricting attention to the subcritical regime $q(x)<p^*(x)$ appears to be a reasonable extension of the well-posedness theory established for the constant exponent subcritical case, provided that the corresponding variational framework can be carried over and the resulting IBVPs remain well-posed in the variable exponent setting (as noted in the last paragraph, it is precisely the critical and supercritical constant exponents that give rise to the functional-analytic difficulties and the tendency toward ill-posedness). However, since the variable exponents vary from point to point with $x$, it is more natural---and in fact unavoidable---to consider a ``mixed regime'' in the variable exponent setting, in which the subcritical, critical, and supercritical cases may all appear in different parts of the same domain $\Omega$. This is in sharp contrast with the constant exponent case, where only one of the three regimes can occur over the whole domain. Such a contrast highlights the distinctive nature of the variable exponent framework and suggests that a more careful treatment is needed beyond the usual subcritical assumption.

Specifically, by taking the critical exponent ${p}^{*}(x)$ as the reference for $q(x)$, the index of the nonlinear source term can exhibit three different states at different points $x \in \Omega$ about the relationship between $q(x)$ and ${p}^{*}(x)$ with $1<p(x)<n$, including the subcritical regime
    \begin{align*}
      q(x)<p^{*}(x),  \quad\text{for } x \in \Omega_1,
    \end{align*}
the critical regime
    \begin{align*}
      q(x)=p^{*}(x),  \quad\text{for } x \in \Omega_2,
    \end{align*}
and the supercritical regime
    \begin{align*}
      q(x)>p^{*}(x),  \quad\text{for } x \in \Omega_3.
    \end{align*}
This means we can divide the domain $\Omega$ into three disjoint parts, i.e., $\Omega_1, \Omega_2$ and $\Omega_3$. Each one of $\Omega_i, i=1,2,3$, possesses a single regime and is not necessarily connected.
We call this phenomenon ``multiple-regime'' or ``triple-regime'', which was proposed in the pioneering work of Alves and Boudjeriou \cite{r1} concerning the parabolic problem \eqref{1}. They proved the local existence and global existence of the solution with initial energy strictly less than the mountain-pass level and also obtained the finite-time blow-up of the solution with negative initial energy.
To provide a better understanding of this phenomenon for the readers, we draw the following schematic diagram Figure \ref{f1} to show one possible configuration in which the multiple-regime phenomenon occurs for the case that the domain is the annulus $\mathcal{A}$ with the spatial dimension $n =2$ (the consideration for $n=1$ is of non-sense according to \eqref{165}) and $p(x),q(x)$ radially symmetric.

    \begin{figure}[h]
    \centering
            \begin{overpic}[width=9cm]{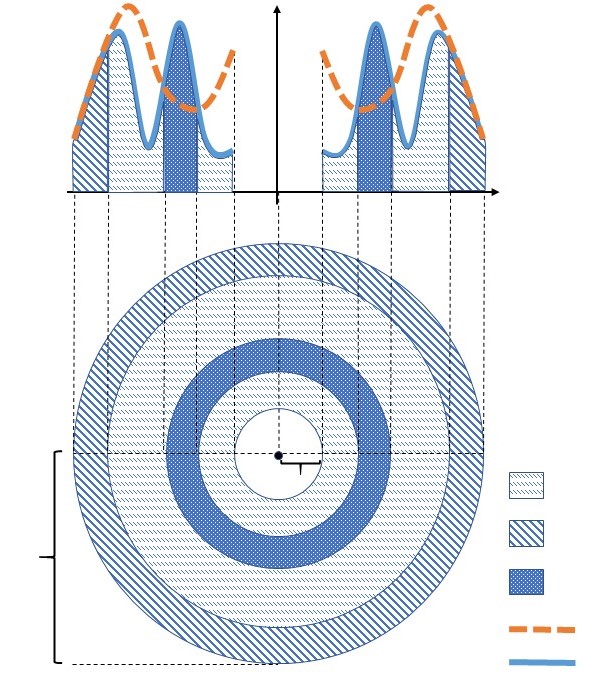}
            \put(210, 200){{$x$}}
            \put(0, 50){{$R$}}
            \put(121, 80){{$r$}}
            \put(108, 83){{$0$}}
            \put(233, 80){{$\Omega_1$}}
            \put(233, 60){{$\Omega_2$}}
            \put(233, 40){{$\Omega_3$}}
            \put(250, 20){{$p^*(x)$}}
            \put(250, 7){{$q(x)$}}
            \end{overpic}
    \caption{Schematic diagram of the multiple-regime phenomenon for the variable exponents in the annulus $\mathcal{A}$}\label{f1}
    \end{figure}

In this paper, we aim to establish the variational structure for the parabolic problem \eqref{1}-\eqref{3} under the multiple-regime setting based on the potential well theory, and in particular, we expect to investigate the rich and complex phenomena of well-posedness with different initial data, especially in relation to the embedding in variable exponent Sobolev space \eqref{77}.
However, this problem is complex and challenging: the appearance of the critical and supercritical regions of positive measure leads to the failure of the compact Sobolev embedding with variable exponent, i.e., \eqref{77}, including the loss of the integrability corresponding to the source term and also the lack of compactness, which makes it impossible to build the corresponding variational structure for problem \eqref{1}-\eqref{3} in the standard Sobolev space $W_0^{1,p(\cdot)}(\Omega)$.

Therefore, the recovery of the compact Sobolev embedding under appropriate restrictions is the cornerstone to reconstructing the variational structure for \eqref{1}-\eqref{3} under the multiple-regime setting. From the constant exponent elliptic point of view, the compact Sobolev embedding with constant exponents can be recovered with the help of the symmetry of the domain.
As a consequence, the existence of the solution to the critical/supercritical elliptic variational problems on bounded domain with certain symmetry constraints can be achieved \cite{a1,a2,a3,a4,a5,a6,a7,a8}.
From the constant exponent parabolic point of view, for \eqref{44} with $p=2$:
      Ni and Sacks et al. \cite{p9} proved that for any positive radially symmetric decreasing $L^{\infty}$ data $u_0$, there exists a scalar $\tau>0$ depending on $u_0$ such that the weak solution to the corresponding IBVP with $q \ge 2^*$ (i.e., the critical and supercritical regime cases) in an $n$-dimensional ball produced by $\tau u_0$ exists globally in time but is not uniformly bounded.
      Then, by the potential well theory, comparison principle, and upper and lower solution method, Ishiwata \cite{p2} completed this result for the critical regime case $q=2^*$ with a class of nonnegative radially symmetric initial data taken from $C^1(\Omega) \cap H_0^1(\Omega)$ and $\Omega = \mathbb{R}^n$ or a $n$-dimensional ball by showing that the solution exists globally with small initial data and does not exist globally with large data in the sense of scaling, which can be regarded as a generalization of Theorem 11 \cite{r7} (for $q<2^*$ on a general bounded domain), and he also distinguished the three possibilities of the decay, non-decay or finite-time blow-up of the solution in $L^{\infty}(\Omega)$ as the scalar starts from small to large.
      Moreover, considering the complete/incomplete blow-up (cf. \cite{p10}), Mizoguchi \cite{r47} obtained a similar result on $\mathbb{R}^n$ for $q > 2^*$ with nonnegative radially symmetric $C^1$ data.
Note that all of these results \cite{p9,p2,r47} imply that the symmetry of both the domain and the initial data leaves us with a glimmer of hope to investigate the well-posedness of the solution to the parabolic problem with constant exponents of the critical/supercritical regime. In this situation, the ``bifurcation phenomenon'' of the global well-posedness of the solution (including global existence and finite-time blow-up) with initial data of different sizes still exists. Thus, it is natural to consider the distinguishment of the initial data that drive the solution to exist globally or blow up in finite time in the corresponding phase space.
Additionally, we refer the readers to \cite{p1,p4,p5} for the $L^{\infty}$-blow-up and to \cite{p6,p7} for the asymptotic symmetry of the solution in the supercritical regime case.

The above discussion of the critical/supercritical elliptic/parabolic problems with symmetry constraints motivates us to establish the potential well theory to deal with the well-posedness for the parabolic problem \eqref{1}-\eqref{3} with variable exponents under the multiple-regime setting on a bounded radial symmetric domain in the expectation that the compact Sobolev embedding \eqref{77} can be recovered under some symmetry constraints of the domain geometry and the restriction on the variable exponent $q(x)$ can be relaxed, without considering exactly in what form the multiple-regime appears (the variable exponent $q(x)$ can even be fully critical or supercritical), see Hypotheses ($\mathrm{H_D}$) and ($\mathrm{H_{p,q}}$) introduced in Section \ref{sec2}, which are the contributions and interesting points of the present paper.
Additionally, as an important tool for recovering the Sobolev embedding with the variable exponents, since the Strauss inequality (cf. Lemma 1, W. A. Strauss 1977 \cite{r10}) fails at the origin, we choose the radially symmetric bounded domain to be the annulus $\mathcal{A}$ rather than the ball,
and we also extend the existence and finite-time blow-up results obtained on the annulus to the case of locally symmetric domains.

It should be pointed out that, to the best of our knowledge, this is the first work that establishes the potential well theory for evolutional PDEs under the multiple-regime setting of the variable exponents on an annulus, which is employed to classify the initial data in the phase space concerning the effects of the initial data on the solution's well-posedness, including the global existence, asymptotic behavior and finite-time blow-up.
Furthermore, the approach proposed in the present paper is based on the energy-variational framework, without requiring higher regularity of the initial data to deal with strong solutions as \cite{p9,p2,r47}.
To provide the readers an intuitive understanding of the role the radial symmetry of the annulus, the variable exponents and the initial data play in the arguments of this paper under the multiple-regime assumption, we summarize the relationship of the related lemmas and theorems in the following Figure \ref{f3}.
    \begin{figure}[h]
    \centering
            \begin{overpic}[width=13cm]{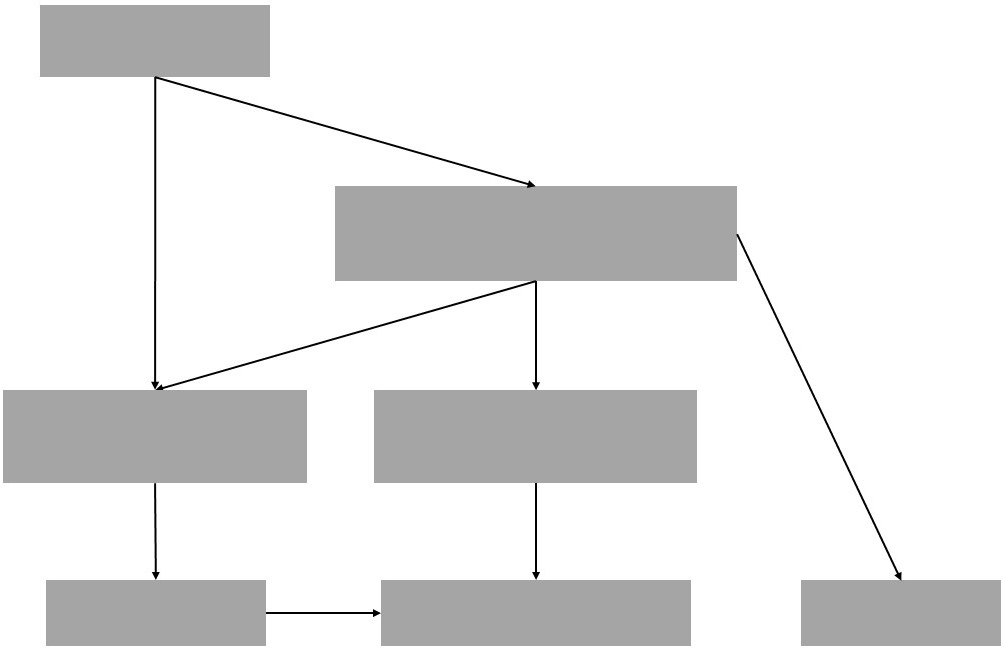}
            \put(16, 227){\small Radial symmetry}
            \put(16, 217){\small of $\mathcal{A}, p(x),q(x), u_0$}
            \put(125, 161){\small Compact Sobolev embedding}
            \put(125, 151){\small for radially symmetric functions}
            \put(125, 141){\small (Lemma \ref{l5} \& Corollary \ref{cor1})}
            \put(3, 85){\small Subdifferential technique}
            \put(3, 75){\small with $G$-symmetry}
            \put(3, 65){\small (Lemma \ref{l19})}
            \put(20, 15){\small Local existence}
            \put(20, 5){\small (Theorem \ref{l1})}
            \put(140, 85){\small Potential well structures}
            \put(140, 75){\small and related properties}
            \put(140, 65){\small (Lemmas \ref{l3}-\ref{l11})}
            \put(145, 15){\small Global well-posedness}
            \put(145, 5){\small (Theorems \ref{t10}-\ref{t2})}
            \put(300, 15){\small Uniqueness}
            \put(300, 5){\small (Theorem \ref{t1})}
            \end{overpic}
    \caption{Relationships of the lemmas and the theorems of the present paper
    in terms of the effects of the multiple-regime phenomenon of $q(x),p^*(x)$ and the radial symmetry of $\mathcal{A}$, $p(x),q(x)$ and $u_0$}\label{f3}
    \end{figure}

The outline of this paper is given as follows.
In Section \ref{sec2}, we give some basic notations, hypotheses, definitions and also some functional spaces with the variable exponents. Additionally, we introduce the subdifferential technique with $G$-symmetry and the corresponding variational structures of the potential well theory for problem \eqref{1}-\eqref{3} on the annulus.
In Section \ref{sec3}, we prove the local existence of the solution via the combination of the subdifferential technique and the Principle of Symmetric Criticality, and the uniqueness result is also given.
In Section \ref{sec33}, we give the finite-time blow-up of the solution with negative initial energy first.
Then, we prove the global existence and finite-time blow-up of the solution with subcritical and critical initial energy and also give the asymptotic behavior of the global ones and the estimation of the blow-up time for the blow-up solution especially for the subcritical initial energy level. Moreover, we give some sufficient criteria of the global existence and finite-time blow-up of the solution with supercritical initial energy.
In Section \ref{sec6}, the existence and finite-time blow-up results obtained on the annulus in Sections \ref{sec3}-\ref{sec33} are extended to the case of locally symmetric domains.
Finally, we give several remarks to discuss further possibilities of the generalization for the present paper.

\section{Preliminaries}\label{sec2}

    \subsection{Basic notations, hypotheses and definitions }\label{sec2.1}

    $\bullet$ {\bf Basic notations.}
    Firstly, we introduce some notations for the classical Lebesgue spaces $L^{p}(\Omega)$. The norm of $L^{p}(\Omega)$ is denoted by
    $\| \cdot \|_{p,\Omega} := \| \cdot \|_{L^p(\Omega)}$ with $p \in (1,\infty)$, and the inner product in $L^2(\Omega)$ is given by $(\cdot,\cdot)_{\Omega}$. Moreover, the symbol $C$ is used to represent a generic constant, whose exact value is not crucial for the argument and its value  may change from line to line even inside the same formula in some estimates.
    In addition, the notations of other specific constants will be given and defined in the text.

    $\bullet$ {\bf Hypotheses.}
    Next, we set some hypotheses on the variable exponents $p(x)$ and $q(x)$ to consider the multiple-regime situation.
    Throughout this paper, we set the common basic hypothesis as follows, which is required for all of our main results

    {\flushleft\textbf{($\mathrm{H_{p,q}}$)}}
   \quad
   The variable exponents $p(x)$ and $q(x)$ satisfy
        \begin{align*}
        \begin{cases}
          &2 \le p^{-} \le p(x) \le p^{+} < n,\\
          &2 \le q^{-} \le q(x) \le q^{+} < \infty
        \end{cases}
        \end{align*}
    for all $x\in \mathcal{A}$, where the constants $p^{\pm}$ and $q^{\pm}$ are given by
        \begin{align*}
            & p^{-} := \mathop{\essinf}\limits_{x\in \mathcal{A}} p(x),\quad
            p^{+} := \mathop{\esssup}\limits_{x\in \mathcal{A}} p(x), \\
            & q^{-} := \mathop{\essinf}\limits_{x\in \mathcal{A}} q(x),\quad
            q^{+} := \mathop{\esssup}\limits_{x\in \mathcal{A}} q(x).
        \end{align*}

    Next, we illustrate the hypothesis on the geometry of the domain in terms of the multiple-regime setting.

    {\flushleft\textbf{($\mathrm{H_D}$)}}
    \quad
    Let us define
        \begin{align*}
        \begin{aligned}
          &\Omega_{\rm sub}:= \{x\in \mathcal{A}: q(x)<p^{*}(x)\},\\
          &\Omega_{\rm critical}:= \{x\in \mathcal{A}: q(x)=p^{*}(x)\},\\
          &\Omega_{\rm sup}:= \{x\in \mathcal{A}: q(x)>p^{*}(x)\}.
        \end{aligned}
        \end{align*}
    Then, without loss of generality, we suppose that all of $\Omega_{\rm sub}, \Omega_{\rm critical}$ and $\Omega_{\rm sup}$ are sets of nonzero measures, where $p^{*}(x)$ is the critical Sobolev exponent given by \eqref{165}.\\

   Next, we introduce two additional hypotheses on the variable exponents $p(x)$ and $q(x)$.

   {\flushleft\textbf{($\mathrm{H^1_{p,q}}$)}}
   \quad
        $
          p(x)=p(|x|), \quad q(x)=q(|x|), \quad \forall x \in \mathcal{A}.
        $

    {\flushleft\textbf{($\mathrm{H^{2}_{p,q}}$)}}
   \quad
        $
        p^{+} <  q(x), \quad \forall x \in \mathcal{A}.
        $\\

   One can observe that the hypothesis $(\mathrm{H^1_{p,q}})$ requires the variable exponents $p(x),q(x)$ to be radially symmetric within $\mathcal{A}$.
   The hypothesis $(\mathrm{H^2_{p,q}})$ implies that the diffusion term is dominated by the nonlinear source term on $\mathcal{A}$, which leads to the mountain-pass geometry of the energy functional.

    \begin{remark}[On hypothesis $(\mathrm{H^2_{p,q}})$]\label{rmk3}
      The hypothesis $(\mathrm{H^2_{p,q}})$ requires that $q(x) > p^+$ holds on the whole domain $\mathcal{A}$, i.e., $q^- > p^+$.
      Additionally, we don't consider the case $q^{+} < p^{-}$ in this paper, in which case the nonlinear effect of the source term can be dominated by virtue of the chain of embeddings
     $ W_{0}^{1,p^-}(\mathcal{A}) \hookrightarrow
       L^{p^-}(\mathcal{A}) \hookrightarrow
       L^{q^+}(\mathcal{A})
     $,
     for which the global existence of the solution can be easily obtained.
    \end{remark}

    $\bullet$ {\bf Definitions.}
    Next, we introduce some definitions in the following content, including the weak solution, maximal existence time and some sets about the well-posedness of the solution in terms of the initial data.
    Furthermore, for the consideration of symmetry and compactness to deal with the multiple-regime case, our initial data are selected from the space $W_{0,\mathrm{rad}}^{1, p(\cdot)}(\mathcal{A})$, see Definition \ref{d1} (ii) below,
    which is a subspace of $W_0^{1, p(\cdot)}(\mathcal{A})$ containing a class of radially symmetric functions, and will be introduced at the end of Subsection \ref{sec2.2}.

    \begin{definition}[Weak solution]\label{d1}
      A function $u=u(x, t)$ is called a weak solution of problem (\ref{1})-(\ref{3}) on $\Omega \times(0, T)$, provided $u \in L^{\infty} (0, T ; W_{0,\mathrm{rad}}^{1, p(\cdot)}(\mathcal{A}))$ with $u_{t} \in L^{2}\left(0, T; L^{2}(\mathcal{A})\right)$ satisfying
        \begin{enumerate}[(i)]
          \item  for any test function $\psi \in W_0^{1,p(\cdot)}(\mathcal{A})$, and a.e. $t \in(0, T)$
                    \begin{align}\label{23}
                       \left(u_{t}, \psi\right)
                        + \left(|\nabla u|^{p(x)-2} \nabla u, \nabla \psi \right)
                        =\left(|u|^{q(x)-2} u, \psi\right);
                    \end{align}
          \item  $u(x, 0)=u_{0}(x)$ in $W_{0,\mathrm{rad}}^{1, p(\cdot)}(\mathcal{A})$.
        \end{enumerate}
    \end{definition}

    \begin{remark}[On Definition \ref{d1}]
      It needs to be emphasized here that, for any radially symmetric initial data $u_0 \in W_{0,\mathrm{rad}}^{1, p(\cdot)}(\mathcal{A})$, the corresponding solution is also radially symmetric, while the test function is not necessary to be radially symmetric, which is due to the so-called Principle of Symmetric Criticality \cite{r56}.
    \end{remark}

    \begin{definition}[Maximal existence time]\label{d2}
      Let $u(x,t)$ be a weak solution of problem (\ref{1})-(\ref{3}) with $u_0 \in W_{0,\mathrm{rad}}^{1, p(\cdot)}(\mathcal{A})$, then the maximal existence time $T_{\max}$ is defined as follows
         \begin{enumerate}[(i)]
	                \item if $u(x,t)$ exists in the sense of
                        Definition \ref{d1}
                        for any fixed $t \in [0,\infty)$, then $T_{\max} = \infty$;
	                \item if there exists a $t_0>0$ such that
                        $u$ exists for any fixed $t \in [0,t_0)$ but cease to exist at $t = t_0$, then $T_{\max} = t_0 <\infty$.
            \end{enumerate}
    \end{definition}

    For the convenience of the description of our main results obtained in this paper, for any solution $u(t)$ to problem \eqref{1}-\eqref{3} with $u_{0} \in W_{0,\mathrm{rad}}^{1, p(\cdot)}(\mathcal{A})$ defined on the interval $\left[0, T_{\max }\right)$, we use the notations employed in \cite{r7,r9} and introduce the following sets
        $$
        \begin{gathered}
        \mathcal{G}:=\left\{u_{0} \in W_{0,\mathrm{rad}}^{1, p(\cdot)}(\mathcal{A}): u(t) \text { exists globally }\right\}, \\
        \mathcal{G}_{0}:=\left\{u_{0} \in \mathcal{G}: u(t) \rightarrow 0 \text { in } W_{0,\mathrm{rad}}^{1, p(\cdot)}(\mathcal{A}) \text { as } t \rightarrow \infty\right\},
        \end{gathered}
        $$
    and
        $$
        \mathcal{B}:=\left\{u_{0} \in W_{0,\mathrm{rad}}^{1, p(\cdot)}(\mathcal{A}): u(t) \text { does not exist globally }\right\}.
        $$
    According to Definition \ref{d2}, we can state briefly that $T_{\max} = \infty$ by denoting $u_0 \in \mathcal{G}$ and $T_{\max} < \infty$ by denoting $u_0 \in \mathcal{B}$.

    \subsection{Spaces $L^{p(\cdot)}(\Omega), W_0^{1, p(\cdot)}(\Omega)$,
    and $W_{0,\mathrm{rad}}^{1, p(\cdot)}(\mathcal{A})$}\label{sec2.2}
    In this subsection, we recall some basic properties of the Lebesgue and Sobolev spaces with variable exponents $L^{p(\cdot)}(\Omega)$,
    $W_0^{1, p(\cdot)}(\Omega)$ for some general bounded domain $\Omega$ and $W_{0,\mathrm{rad}}^{1, p(\cdot)}\left(\mathcal{A}\right)$ for the annulus $\mathcal{A}$, respectively.
    For more details about functional spaces with variable exponents, we refer the readers to \cite{r4,r11} and the references therein.
    Then, we denote by $\mathcal{P}(\Omega)$ the set of all measurable functions $p(x)$ satisfying
    $$1< p^- \le p^+ <n,$$
    and denote by $\mathcal{P}_{\log }(\Omega)$ the set of all functions $p(x) \in \mathcal{P}(\Omega)$ satisfying the log-H\"{o}lder condition \eqref{4}.

   $\bullet$ {\bf The Lebesgue space $L^{p(\cdot)}(\Omega)$ with variable exponents.}
   For the following given modular
        $$
        \rho(u)
        := \int_{\Omega} |u|^{p(x)} \mathrm{d} x,
        $$
    the variable exponent Lebesgue space $L^{p(\cdot)}(\Omega)$ is defined by
        $$
        L^{p(\cdot)}(\Omega):=\left\{u: u \text { is measurable and } \rho(u) <\infty\right\}
        $$
    equipped with the so-called Luxemburg norm
        $$
        \|u\|_{p(\cdot),\Omega}:=\inf \left\{\lambda> 0 :
        \rho \left(\frac{u}{\lambda} \right)  \leq 1\right\},
        $$
    which is a separable and uniformly convex Banach space.
    Then we introduce some related properties including the relations between the norm $\|u\|_{p(\cdot),\Omega}$ and the modular $\rho(u)$, H\"{o}lder's inequality and embedding theorems for the variable exponent Lebesgue spaces.

    \begin{proposition}[Relationship between the modular and the norm\cite{r4}]\label{p1}
      For all $u \in L^{p(\cdot)}(\Omega)$, there holds
        $$
        \min \left\{\|u\|_{p(\cdot),\Omega}^{p^{-}},\|u\|_{p(\cdot),\Omega}^{p^{+}}\right\} \leq \rho(u) \leq \max \left\{\|u\|_{p(\cdot),\Omega}^{p^{-}},\|u\|_{p(\cdot),\Omega}^{p^{+}}\right\}.
        $$
    \end{proposition}

    \begin{proposition}[Embedding between Lebesgue spaces \cite{r4}]\label{p4}
      If $p(x) \leq q(x)$ for a.e. $x \in \Omega$, then $L^{q(\cdot)}(\Omega) \hookrightarrow L^{p(\cdot)}(\Omega)$ continuously.
    \end{proposition}

    $\bullet$ {\bf The spaces $W^{1, p(\cdot)}(\Omega)$ and $ W_0^{1, p(\cdot)}(\Omega)$.}
    The variable exponent Sobolev space is defined by
        $$
        W^{1, p(\cdot)}(\Omega) :=\left\{u \in L^{p(\cdot)}(\Omega):| \nabla u | \in L^{p(\cdot)}(\Omega)\right\}
        $$
    with the norm
        \begin{align}\label{92}
          \|u\|_{W^{1,p(\cdot)}(\Omega)} :=\left(\|u\|_{p(\cdot),\Omega}^{2}
          +\|\nabla u\|_{p(\cdot),\Omega}^{2}\right)^{1 / 2},
        \end{align}
    which is a reflexive and separable Banach space for $1<p^{-} \leq p^{+}<\infty$, and the space $W_{0}^{1, p(\cdot)}(\Omega)$ is the closure of $C_{0}^{\infty}(\Omega)$ in $W^{1, p(\cdot)}(\Omega)$.
    Then we introduce the Poincar\'{e} inequality and Sobolev embeddings for the variable exponent Sobolev spaces.

    \begin{proposition}[Poincar\'{e} inequality \cite{r4}]\label{p5}
      Let $\Omega$ be a bounded domain in $\mathbb{R}^{n}$ with smooth boundary $\partial \Omega$ and $p(x) \in \mathcal{P}_{\log }(\Omega)$. Then for all $u \in W_{0}^{1, p(\cdot)}(\Omega)$, there exists a constant
      $C_{p(\cdot)}^* > 0$ such that
        $$
        \|u\|_{p(\cdot),\Omega} \leq C_{p(\cdot)}^*\|\nabla u\|_{p(\cdot),\Omega}.
        $$
    \end{proposition}

\begin{proposition}[Sobolev embedding \cite{r4}]\label{p6}
      Let $\Omega$ be a bounded domain in $\mathbb{R}^{n}$ with smooth boundary $\partial \Omega$,
      $p(x) \in \mathcal{P}_{\log }(\Omega)$ and $q(x) \in \mathcal{P}(\Omega)$, the following statements hold:

    (i) if $1 \leq q(x) \leq p^{*}(x)$, then $W^{1, p(\cdot)}(\Omega) \hookrightarrow L^{q(\cdot)}(\Omega)$ continuously;

    (ii) if $1 \leq q(x)$ and $\inf \limits_{x \in \Omega}\left(p^{*}(x)-q(x)\right)>0$, then $W^{1, p(\cdot)}(\Omega) \hookrightarrow L^{q(\cdot)}(\Omega)$ compactly.
    \end{proposition}

    By virtue of the Poincar\'{e} inequality, see Proposition \ref{p5}, we have the following equivalent norm in $W_0^{1, p(\cdot)}(\Omega)$
        \begin{align}\label{52}
         \|u\|_{W_0^{1,p(\cdot)}(\Omega)}
        \simeq \|\nabla u\|_{p(\cdot),\Omega},
        \end{align}
    rather than the original one given by \eqref{92}.

    $\bullet$ {\bf The space $W_{0,\mathrm{rad}}^{1, p(\cdot)}(\mathcal{A})$.}
    Corresponding to the case that $\Omega$ is replaced by the annulus $\mathcal{A}$, we define the space $W_{0,\mathrm{rad}}^{1, p(\cdot)}(\mathcal{A})$ as follows
        \begin{align}\label{9}
          W_{0,\mathrm{rad}}^{1, p(\cdot)}(\mathcal{A}):=
          \left\{u \in W_{0}^{1, p(\cdot)}(\mathcal{A}): u(x)=u(|x|) \text { a.e. } x \in \mathcal{A}\right\},
        \end{align}
    which is a subspace of $W_0^{1,p(\cdot)}(\mathcal{A})$ invariant under the action of $\mathbf{O}(n)$ (the group of orthogonal transformations of dimension $n$) and is endowed with the same norm of  $W_0^{1,p(\cdot)}(\mathcal{A})$.
    By the symmetry of $\mathcal{A}$ and $W_{0,\mathrm{rad}}^{1, p(\cdot)}(\mathcal{A})$, we shall introduce the compact Sobolev embedding from
    $W_{0,\mathrm{rad}}^{1, p(\cdot)}(\mathcal{A})$ into
    $L^{q}(\mathcal{A})$ for constant $q$ satisfying $1<q<\infty$ under the hypotheses $(\mathrm{H_{p,q}})$, $(\mathrm{H^1_{p,q}})$ and $(\mathrm{H_{D}})$, which is made extensive use throughout the paper and plays a crucial role as the cornerstone of the establishment of the potential well structure of problem \eqref{1}-\eqref{3}. Before that, we recall an important result obtained by Strauss in 1977, i.e., Lemma 1 \cite{r10} shown as follows.
        \begin{lemma}[Strauss inequality \cite{r10}]\label{l20}
          Let $n \geq 2$. Every radial function $u \in H^1(\mathbb{R}^n)$ is almost everywhere equal to a function $U(x)$, continuous for $x \neq 0$, such that
            \begin{align}\label{181}
              |U(x)| \leq c|x|^{\frac{1-n}{2}}\|u\|_{H^1(\mathbb{R}^n)} \quad \text { for a.e. } x \in \mathbb{R}^n,
            \end{align}
          where the positive constant $c$ only depends on $n$.
        \end{lemma}
        \begin{remark}[On Lemma \ref{l20}]\label{rmk6}
           By checking the proof of Lemma 1 \cite{r10}, it is easy to find that for the case $n = 2$, no singularity appears at the origin, namely there holds
            \begin{align*}
              |U(x)| \leq c\|u\|_{H^1(\mathbb{R}^n)} \quad \text { for a.e. } x \in \mathbb{R}^n,\ n=2.
            \end{align*}
        \end{remark}

    It is well known that $H_{\mathbf{O}(n)}^1(\mathbb{R}^n) \hookrightarrow L^p(\mathbb{R}^n)$ compactly
    (cf. Corollary 1.26 \cite{r58}) for $2<p<2^*:= \frac{2n}{n-2}$, due to Lemma \ref{l20}, where
        \begin{align*}
          H_{\mathbf{O}(n)}^1(\mathbb{R}^n) := \{u\in H^1(\mathbb{R}^n): u(x) = u(|x|)\}.
        \end{align*}
    The recovered compactness with symmetry from Lemma \ref{l20} is usually applied to treat the compactness issue of the Sobolev embeddings on the whole space $\mathbb{R}^n$.
    In the present paper, in view of the lack of compactness caused by the multiple-regime phenomenon of the variable exponents, i.e., ($\mathrm{H_D}$) (especially the appearance of the regions $\Omega_{\rm critical}$ and $\Omega_{\rm sup}$), we shall apply Lemma \ref{l20} in another way to recover the compactness of the Sobolev embedding $W_{0,\mathrm{rad}}^{1, p(\cdot)}(\mathcal{A}) \hookrightarrow L^{q}(\mathcal{A})$ ($1<q<\infty$) on $\mathcal{A}$.

    Moreover, it is necessary to emphasize the rationality of consideration of the domain $\mathcal{A}$ in terms of compactness and symmetry, since $B_R (0)$ (a ball centered at the origin with radius $R>0$) is also a bounded radially symmetric domain. One can observe that the inequality \eqref{181} in Lemma \ref{l20} requires a ``hole'' containing the origin at the center part of the domain such that the estimate is uniform (in $x$), i.e., $|x|\ge r$ for some $r>0$. Specifically, in \eqref{181}, the exponent $\frac{1-n}{2}$ is negative for $n> 2$ and produces a singularity at the origin, where \eqref{181} is of non-sense. Therefore, an annulus centered at the origin seems to be the fitting domain to enjoy \eqref{181} rather than a ball.

    To end this part, we give the expected compact Sobolev embedding result in the following Lemma \ref{l5}.
    Although it is not difficult to prove, we could not find an exact reference for it. Therefore, we give a detailed proof of Lemma \ref{l5} for the sake of completeness.

        \begin{lemma}\label{l5}
         Let $(\mathrm{H_{p,q}})$, $(\mathrm{H^1_{p,q}})$ and $(\mathrm{H_{D}})$ be satisfied, then
         \begin{enumerate}[(i)]
         \item $W_{0,\mathrm{rad}}^{1, p(\cdot)}(\mathcal{A}) \hookrightarrow L^{\infty}(\mathcal{A})$ continuously;
         \item $W_{0,\mathrm{rad}}^{1, p(\cdot)}(\mathcal{A}) \hookrightarrow L^{q}(\mathcal{A})$ compactly for any constant $q \in(1,\infty)$.
         \end{enumerate}
        \end{lemma}
        \begin{proof}
          (i) Firstly, it follows from Lemma \ref{l20} that
            \begin{align}\label{8}
              \mathop{\esssup}\limits_{x \in \mathbb{R}^n\backslash B_r(0)} |u(x)|
              \le C r^{\frac{1-n}{2}} \|u\|_{H^1(\mathbb{R}^n)}
            \end{align}
           for any fixed $r>0$ and radially symmetric $u \in H^1(\mathbb{R}^n)$, i.e., $u(x)=u(|x|)$ for $x\in \mathbb{R}^n$, where $C$ is some positive constant only depends on the spatial dimension $n$.
           Then for $u \in W_{0,\mathrm{rad}}^{1, p(\cdot)}(\mathcal{A})$, we can extend it to a member of $W_{0,\mathrm{rad}}^{1, p(\cdot)}(\mathbb{R}^n)$ by defining
                \begin{align*}
                  \tilde{u}(x):=
                  \begin{cases}
                    u, & x\in \mathcal{A}, \\
                    0, & x\in \mathbb{R}^n\backslash \mathcal{A}.
                  \end{cases}
                \end{align*}
           Note that $W_{0,\mathrm{rad}}^{1, p(\cdot)}(\mathcal{A}) \hookrightarrow W_{0,\mathrm{rad}}^{1, p^-}(\mathcal{A}) \hookrightarrow H_0^1(\mathcal{A})$ for $p(x) \ge p^- >2$ (the former embedding follows from Proposition \ref{p4} in view of $p^- \le p(x)$ and $|\mathcal{A}|<\infty$, and the latter from $p^- > 2$), by substituting $\tilde{u}$ into \eqref{8}, we can obtain
                \begin{align}\label{105}
                  \|u\|_{\infty,\mathcal{A}}
                  = \mathop{\esssup}\limits_{x \in \mathcal{A}} |u(x)|
                  \le C \|\nabla u\|_{p(\cdot),\mathcal{A}},
                \end{align}
           where $C$ is some positive constant only depends on $n,p^-,p^+$ and $r,R$. Therefore, \eqref{105} indicates the continuous embedding $W_{0,\mathrm{rad}}^{1, p(\cdot)}(\mathcal{A}) \hookrightarrow L^{\infty}(\mathcal{A})$.

           (ii) Next, we prove the compact embedding $W_{0,\mathrm{rad}}^{1, p(\cdot)}(\mathcal{A}) \hookrightarrow L^{q}(\mathcal{A})$ for $1<q<\infty$. Suppose that $\{u_k\}_{k=1}^{\infty} \subset W_{0,\mathrm{rad}}^{1, p(\cdot)}(\mathcal{A})$ is a sequence of functions such that
           $\sup \limits_{k} \|\nabla u\|_{p(\cdot),\mathcal{A}} \le C$ with some constant $C>0$. Then, as a preparation of the modification, let us extend the domain of $u_k$ by defining
                \begin{align*}
                  \tilde{u}_k(x):=
                  \begin{cases}
                    u_k, & x\in \mathcal{A}, \\
                    0, & x\in \mathcal{A}_{\delta} \backslash \mathcal{A},
                  \end{cases}
                \end{align*}
           where
                $$\mathcal{A}_{\delta}:=\left\{x \in \mathbb{R}^n:
                r - \delta <|x|<R + \delta \right\}$$
           is the extended annulus with some fixed number $\delta$ satisfying $0<\delta<\frac{r}{2}$. Then, there holds
           $\|\tilde{u}_k\|_{q,\mathcal{A}_{\delta}}
           \ge \|u_k\|_{q,\mathcal{A}}$. Let $u_k^{\varepsilon}:= \eta_{\varepsilon} * \tilde{u}_k$ for $\varepsilon>0$ and $k=1,2,\cdots$, where $\eta_{\varepsilon}$ is the usual mollifier. Obviously, $\{u_k^{\varepsilon}\} \subset C_0^{\infty}(\mathcal{A}_{\delta})$ for sufficient small
           $\varepsilon <\delta$ and
               \begin{align}\label{106}
                 u_k^{\varepsilon} \rightarrow \tilde{u}_k
                 \quad\text{in } L^1(\mathcal{A}_{\delta}),
                 \quad\text{as } \varepsilon \rightarrow 0,
                 \text{ uniformly in } k,
               \end{align}
           according to the properties of the mollifier $\eta_{\varepsilon}$. Then, it follows from the $L^p$ interpolation inequality for $L^{\infty}(\mathcal{A}_{\delta}) \hookrightarrow L^{q}(\mathcal{A}_{\delta}) \hookrightarrow L^{1}(\mathcal{A}_{\delta})$ and \eqref{8} that
                \begin{align}\label{107}
                \begin{aligned}
                  \|u_k^{\varepsilon} - \tilde{u}_k\|_{q,\mathcal{A}_{\delta}}
                  \le& \|u_k^{\varepsilon} - \tilde{u}_k\|_{\infty, \mathcal{A}_{\delta}} \|u_k^{\varepsilon} - \tilde{u}_k\|_{1,\mathcal{A}_{\delta}}\\
                  \le& \|\nabla (u_k^{\varepsilon} - \tilde{u}_k)\|_{p(\cdot),\mathcal{A}_{\delta}}
                        \|u_k^{\varepsilon} - \tilde{u}_k\|_{1,\mathcal{A}_{\delta}},
                \end{aligned}
                \end{align}
           which together with \eqref{106} implies that $u_k^{\varepsilon} \rightarrow \tilde{u}_k$ in $L^q(\mathcal{A}_{\delta})$ as $\varepsilon \rightarrow 0$, uniformly in $k$. Moreover, for each fixed $\varepsilon \in (0,\delta)$, the family $\{u_k^{\varepsilon}\}_{k=1}^{\infty}$ is uniformly bounded and equicontinuous in $C(\overline{\mathcal{A}_{\delta}})$. Indeed, $\|u_k^{\varepsilon}\|_{\infty} \le \|\tilde{u}_k\|_{\infty} \le C$ uniformly in $k$ (see \eqref{8} and \eqref{105}), and the mean value theorem yields $|u_k^{\varepsilon}(x)-u_k^{\varepsilon}(y)| \le \|\nabla \eta_{\varepsilon}\|_{\infty}\|\tilde{u}_k\|_{1,\mathcal{A}_{\delta}}|x-y| \le C_{\varepsilon}|x-y|$ for all $x,y \in \overline{\mathcal{A}_{\delta}}$, because $\|\tilde{u}_k\|_{1,\mathcal{A}_{\delta}} \le |\mathcal{A}_{\delta}|\,\|\tilde{u}_k\|_{\infty} \le C$ uniformly in $k$. Hence, by the standard
           Arzela-Ascoli compactness argument and the diagonal argument (similarly as the proof of the Rellich-Kondrachov Compactness Theorem), we can obtain a Cauchy subsequence $\{u_{k_l}\}_{l=1}^{\infty} \subset \{u_k\}_{k=1}^{\infty}$ such that
                \begin{align*}
                  \limsup _{l_1, l_2 \rightarrow \infty}
                  \|u_{k_{l_1}}-u_{k_{l_2}}\|_{q,\mathcal{A}}
                  =  \limsup _{l_1, l_2 \rightarrow \infty}
                    \|\tilde{u}_{k_{l_1}}-\tilde{u}_{k_{l_2}}\|_{q,\mathcal{A}_{\delta}}=0,
                \end{align*}
           which asserts the compactness of the embedding $W_{0,\mathrm{rad}}^{1, p(\cdot)}(\mathcal{A}) \hookrightarrow L^{q}(\mathcal{A})$ for $1<q<\infty$.
        \end{proof}
        \begin{remark}[On Lemma \ref{l5}]
          We point out that, it can be further deduced from Lemma 1 \cite{r10} that $W_{0,\mathrm{rad}}^{1, p(\cdot)}(\mathcal{A}) \hookrightarrow C(\bar{\mathcal{A}})$ after eventual modification on a set of measure zero, but the assertion (i) of lemma \ref{l5} is adequate to prove assertion (ii) and other results of this paper. Besides, the constraint on $q$ of Lemma \ref{l5} (ii) (i.e., $q \in(1,\infty)$) is strict for the application of the $L^p$ interpolation inequality, cf. \eqref{107}.
        \end{remark}

        Moreover, from the embedding $L^{q^+}(\mathcal{A}) \hookrightarrow L^{q(\cdot)}(\mathcal{A})$, we have the following corollary immediately.

        \begin{corollary}\label{cor1}
        Let $(\mathrm{H_{p,q}})$, $(\mathrm{H^1_{p,q}})$ and $(\mathrm{H_{D}})$ be satisfied, then
        $W_{0,\mathrm{rad}}^{1, p(\cdot)}(\mathcal{A}) \hookrightarrow L^{q(\cdot)}(\mathcal{A})$ compactly.
        \end{corollary}

    \subsection{Subdifferential technique with $G$-symmetry}\label{sec2.4}

    Concerning symmetry, the so-called Principle of Symmetric Criticality was first proposed by R. Palais \cite{r57} in 1979, which was generalized for non-smooth functionals by Kobayashi and \^{O}tani \cite{r56}. Then, Akagi, Kobayashi and \^{O}tani combined the result of \cite{r56} and the subdifferential techniques in the reflexive Banach space setting \cite{r53} to consider the $G$-symmetric solutions of an abstract evolutionary inclusion governed by subdifferential operators, which can grasp the form of problem \eqref{1}-\eqref{3} studied in the present paper.

    As the generalization of Fr\'{e}chet derivative for non-smooth convex functionals, the theory of perturbation for subdifferential operators is often employed to verify the existence of solutions of evolutionary PDEs like \eqref{44}. For related studies about the subdifferential technique, we mention here \cite{r23} in the Hilbert space setting and \cite{r24,r53} in the reflexive Banach space setting.
    Next, we introduce the subdifferential techniques with $G$-symmetry proposed by \cite{r55} and recall some notations and definitions therein.

    $\bullet$ {\bf Group action and Banach $G$-space.}
    Let $X$ be a real Banach space and let $G$ be a group. Then, considering the action of $G$ on $X$,
    we give the definition of Banach $G$-space first.

    \begin{definition}[Banach $G$-space \cite{r55}]
      We call $X$ a Banach $G$-space if $G$ acts on $X$ as a transformation group of bounded linear operators in $X$, that is, a bounded linear operator $\pi_{X}(g)$ in $X$ is assigned to each $g \in G$ and satisfies
        $$
        \begin{aligned}
        & \pi_{X}(e) u=u, \quad\forall u \in X \\
        & \pi_{X}\left(g_{1} g_{2}\right) u=\pi_{X}\left(g_{1}\right)\left(\pi_{X}\left(g_{2}\right) u\right), \quad \forall g_{1}, g_{2} \in G, \ \forall u \in X,
        \end{aligned}
        $$
      where $e$ is the identity element of $G$, $\pi_{X}$ is called a representation of $G$ over $X$.
    \end{definition}

    Then, we give the definitions of the $G$-invariance for functionals on $X$, subsets of $X$ and linear subspaces of $G$-symmetric points of $X$, separately.
        \begin{definition}\label{d4}
        Let $X$ be a Banach $G$-space. A functional $f$ on $X$ is called $G$-invariant if
            $$
            f(g u)=f(u), \quad\forall u \in X,\ \forall g \in G
            $$
        and a subset $M$ of $X$ is called $G$-invariant if
            $$
            g M=\{g u: u \in M\} \subset M, \quad \forall g \in G.
            $$
        The linear subspaces of $G$-symmetric points of $X$ is defined as the common fixed points of $G$ :
            $$
            \Sigma_{X}:=\left\{u \in X: g u=u, \ \forall g \in G\right\}.
            $$
        \end{definition}
    It is clear that $\Sigma_{X}$ forms a closed linear subspace of $X$, and $\Sigma_{X}$ can be regarded as a Banach space with its induced topology \cite{r55}.

    $\bullet$ {\bf Subdifferential and evolutionary inclusions in reflexive Banach spaces with $G$-symmetry.}
    In this part, we concern the evolution equations governed by subdifferentials in reflexive Banach $G$-spaces. To this end, we introduce some basic notations and definitions for functionals defined on the real Banach space $X$, i.e., $\varphi: X \rightarrow (-\infty,+\infty]$ and the corresponding subdifferentials $\partial \varphi$ first.
    Let
        \begin{align*}
          \Phi(X):= \{\varphi: \varphi
          \text{ are proper lower semicontinuous convex functionals}\},
        \end{align*}
    where ``proper'' means the effective domain of $\varphi$
        $$D(\varphi):=\{u \in X : \varphi(u)<+\infty\}$$
    is not empty. For $u \in D(\varphi)$, the subdifferential $\partial_{X} \varphi(u)$ of $\varphi$ at $u$ is defined by
        \begin{align}\label{190}
         \partial_{X} \varphi(u):=\left\{v^{*} \in X^{*}: \varphi(w)-\varphi(u) \geq\left\langle v^{*}, w-u\right\rangle,\ \forall w \in X\right\},
        \end{align}
    where $X^{*}$ is the dual space of $X$.
    Then, as is well known, $\partial_{X} \varphi: X \rightarrow X^{*}$ is a maximal monotone operator. For simplicity of notation, we usually write $\partial \varphi$ instead of $\partial_{X} \varphi$. Additionally, we define
        \begin{align*}
          \Phi_{G}(X):= \{ \varphi\in \Phi(X): \varphi \ \text{is}\ G\text{-invariant in the sense of Definition \ref{d4}}\}.
        \end{align*}

    Let $V$ be a real reflexive Banach space such that $V$ and its dual $V^{*}$ are both strictly convex. We assume that there exists a real Hilbert space $H$ identified with its dual $H^{*}$ such that
        \begin{equation*}
        V \subset H \equiv H^{*} \subset V^{*},
        \end{equation*}
    where each injection is dense and continuous.
    Additionally, we denote $\|\cdot\|_V$ the norm of $V$ and $|\cdot|_H$ the norm of $H$.
    Then it holds that
        \begin{equation*}
        \langle u, v\rangle_{V}=(u, v)_{H}, \quad \forall u \in H,\ \forall v \in V,
        \end{equation*}
    where $\langle\cdot, \cdot\rangle_{V}$ denotes the duality pairing between $V$ and $V^{*}$ and
    $(\cdot, \cdot)_{H}$ denotes the inner product of $H$ respectively.

    Let $\varphi^{1}, \varphi^{2} \in \Phi_{G}(V)$ be such that $D\left(\varphi^{1}\right) \cap D\left(\varphi^{2}\right) \neq \emptyset$.
    Then, we consider the $G$-invariant strong solution of the following Cauchy problem
        \begin{equation*}
        \begin{cases}
          \frac{\mathrm{d} u}{\mathrm{d} t}(t)+\partial \varphi^{1}(u(t))-\partial \varphi^{2}(u(t)) \ni 0 \text { in } V^{*},\\
          u(0)=u_{0}.
        \end{cases} \tag{CP}
        \end{equation*}

     The strong solution of (CP) is defined as follows.

        \begin{definition}[Strong solution of (CP) \cite{r55}]
          A function $u \in C\left([0, T] ; V^{*}\right)$ is said to be a strong solution of (CP) on $[0, T]$ if the following conditions are satisfied:
            \begin{enumerate}[(i)]

            \item $u(\cdot)$ is a $V^{*}$-valued absolutely continuous function on $(0, T]$,

            \item $u(0)=u_{0}$,

            \item $u(t) \in D\left(\partial \varphi^{1}\right) \cap D\left(\partial \varphi^{2}\right)$ for a.e. $t \in(0, T)$ and there exist sections $\xi^{1}(t) \in$ $\partial \varphi^{1}(u(t))$ and $\xi^{2}(t) \in \partial \varphi^{2}(u(t))$ satisfying
                \begin{equation}\label{176}
                    \frac{\mathrm{d} u}{\mathrm{d} t}(t)+\xi^{1}(t)-\xi^{2}(t)=0 \text { in } V^{*} \text { a.e. on }(0, T).
                \end{equation}
            \end{enumerate}
        \end{definition}

        Moreover, we give the local existence of the solution of (CP) by setting $f \equiv 0$ in Theorem 2 \cite{r55}.

        \begin{lemma}[Existence of the strong solution of (CP) \cite{r55}]\label{l19}
        Assume that the following conditions are satisfied:
            \begin{enumerate}[(i)]
                \item The action of $G$ on $V$ and $H$ are both isometric and satisfy
                    \begin{equation}\label{177}
                    \pi_{H}(g) u=\pi_{V}(g) u, \quad \forall g \in G, \ \forall u \in V.
                    \end{equation}

                \item There exists some positive constant $\gamma$ such that
                    \begin{align}\label{120}
                      \|u\|_{V}^{\gamma}-C_{1}|u|_{H}^{2}-C_{2} \leq C_{3} \varphi^{1}(u), \quad \forall u \in \Sigma_{V}.
                    \end{align}

                \item $D\left(\varphi^{1}\right) \cap \Sigma_{V} \subset D\left(\partial \varphi^{2}\right)$ and the following holds:

                Let $\{u_{m}\}$ be a $G$-symmetric sequence of $L^{\infty}(0, T ; V) \cap W^{1,2}(0, T ; H)$ such that
                    \begin{align*}
                      \sup _{t \in[0, T], m \in \mathbb{N}}\left\{\varphi^{1}\left(u_{m}(t)\right)+|u_{m}(t)|_{H}\right\}<+\infty, \
                     \sup _{m \in \mathbb{N}} \int_{0}^{T}
                     |u_{mt}(\tau)|_{H}^{2} d t<+\infty
                    \end{align*}
                and let $\xi_{m}(\cdot) \in \partial \varphi^{2}\left(\left(u_{m}(\cdot)\right)\right.$. Then $\{\xi_{m}\}$ has a (strongly) convergent subsequence in $C\left([0, T] ; V^{*}\right)$.

                \item There exists $\tilde{\varphi}^{2} \in \Phi(H)$ such that $\tilde{\varphi}^{2}(u)=\varphi^{2}(u)$ for all
                    $u \in V$,
                    $$J_{\lambda}^{2}\left(D\left(\varphi^{1}\right) \cap \Sigma_{V}\right) \subset D\left(\varphi^{1}\right) \cap \Sigma_{V},$$ and
                        $$
                        \varphi^{1}\left(J_{\lambda}^{2} u\right) \leq l_{1}\left(\varphi^{1}(u)+l_{2}\left(|u|_{H}\right)\right), \quad \forall u \in D\left(\varphi^{1}\right) \cap \Sigma_{V},\ \forall \lambda>0,
                        $$
                    where $C_{1}, C_2, C_3$ are positive constants, $l_{1}, l_{2}$ are positive monotone increasing functions on $[0,+\infty), \gamma \in(1,+\infty)$, and $J_{\lambda}^{2}$ denotes the resolvent of $\partial_{H} \tilde{\varphi}^{2}$, i.e., $J_{\lambda}^{2}=\left(I+\lambda \partial_{H} \tilde{\varphi}^{2}\right)^{-1}$.
            \end{enumerate}
                Then for all $u_{0} \in D\left(\varphi^{1}\right) \cap \Sigma_{V}$, there exists a $T_0 >0$ such that $(\mathrm{CP})$ admits a local strong solution $u$ on $[0, T_0)$ satisfying
            \begin{align}\label{174}
            \begin{cases}
               u \in C_{w}([0, T_0] ; V) \cap W^{1,2}(0, T_0 ; H), \\
               u(t) \in D(\partial \varphi^{1}) \cap D(\partial \varphi^{2}) \cap \Sigma_{V}  \quad \text { for a.e. } t \in (0, T_0), \\
               \xi^{1} \in L^{2}(0, T_0 ; V^{*}), \quad \xi^{2} \in C([0, T_0] ; V^{*}), \\
               \sup \limits_{t \in [0, T_0]} \varphi^{1}(u(t))<+\infty,\quad \varphi^{2}(u(\cdot)) \in C([0, T_0]),
            \end{cases}
            \end{align}
        where $C_{w}\left([0, T_0] ; V\right)$ denotes the set of all $V$-valued weakly continuous functions on $[0, T_0]$, and $\xi^{1}$ and $\xi^{2}$ are sections of $\partial \varphi^{1}$ and $\partial \varphi^{2}$ satisfying \eqref{176} for a.e. $t \in [0, T_0]$.
        \end{lemma}

        $\bullet$ {\bf $G$-symmetry induced by the geometry of the domain.}
        To apply the abstract result given in Lemma \ref{l19} in the functional setting towards evolutionary PDEs on a domain $\Omega$ possessing some symmetry characterized by the invariance under the orthogonal transformations, we consider the $G$-symmetry induced by the geometry of the domain and introduce some related notations and definitions.

        Let $G$ be a subgroup of $\mathbf{O}(n)$ (the group of orthogonal transformations of dimension $n$)
        and $\Omega \subset \mathbb{R}^{n}$ with smooth boundary $\partial \Omega$ be invariant under the action of $G$, i.e., $g(\Omega)=\Omega$ for all $g \in G$.
        Suppose that $\Omega$ is compatible with $G$ (see \cite{r58}), that is, for some constant $\rho>0$
            \begin{equation}\label{108}
            m(y, \rho, G) \rightarrow \infty \quad \text { as }|y| \rightarrow \infty \text { with } \operatorname{dist}(y, \Omega) \leq \rho,
            \end{equation}
        where
            $$
            m(y, \rho, G)=\sup \begin{cases}\left.n \in \mathbb{N}: \begin{array}{l}
            \exists g_{1}, g_{2}, \ldots, g_{n} \in G \text { s.t. } \\
            B_{\rho}\left(g_{j} y\right) \cap B_\rho\left(g_{k}y\right)=\emptyset \text { if } j \neq k
            \end{array}\right\}\end{cases}
            $$
        and $B_{\rho}\left(g_{j} y\right)$ denotes the ball centered at $g_{j} y$ with the radius $\rho$.

        Then, we introduce the action of $G$ on functions and the definition of $G$-invariant functions.

            \begin{definition}[Group action on functions and $G$-invariant functions \cite{r55}]\label{d3}
            For $u \in \mathcal{D}=C_{0}^{\infty}(\Omega)$ and $g \in G$, we define $g u$ by
                \begin{equation}\label{175}
                (g u)(x)=u\left(g^{-1} x\right), \quad \forall g \in G, \ x \in \Omega.
                \end{equation}
            For a distribution $U\left(U \in \mathcal{D}^{\prime}, \text{where}\ \mathcal{D}^{\prime} \text{ is the dual space of } \mathcal{D} \right)$ and $g \in G$, we also define $g U$ by
                $$
                \langle g U, \eta\rangle=\left\langle U, g^{-1} \eta\right\rangle, \quad\forall \eta \in C_{0}^{\infty}(\Omega).
                $$
            Especially, for $u \in L_{\text {loc }}^{1}(\Omega)$, then $g u$ is also defined by \eqref{175}.
            Additionally, we define the $G$-invariant function as follows:
                \begin{enumerate}[(i)]
                  \item We say $U \in \mathcal{D}^{\prime}$ is $G$-invariant if $g U=U$ for all $g \in G$.
                        Especially $u \in L_{\text {loc }}^{1}(\Omega)$ is $G$-invariant if and only if $u(g x)=u(x)$ for all $g \in G$ and a.e. $x \in \Omega$.
                  \item A Banach space valued function $f: (0, T) \rightarrow \mathcal{D}^{\prime}$ is said to be
                        $G$-invariant if $f(t)$ is $G$-invariant a.e. $t \in(0, T)$.
                \end{enumerate}
            \end{definition}

         From Definition \ref{d3}, one can deduce that the relation \eqref{175} defines the group
         action on a Banach space $X$ as a transformation group of bounded linear operators in $X$, that is
            $$
              \pi_{X}(g) u = gu, \quad \forall u \in X.
            $$

         Moreover, in this paper, we shall consider the case that $\Omega = \mathcal{A}$ and $G = \mathbf{O}(n)$.
         Obviously, the annulus $\mathcal{A}$ is invariant under the action of $\mathbf{O}(n)$ and are compatible with $\mathbf{O}(n)$ in the sense of \eqref{108}.
         Additionally, if we take $V = W_{0}^{1, p(\cdot)}(\mathcal{A})$ and
         $H = L^2(\mathcal{A})$, then \eqref{177} is satisfied and $\Sigma_{V} = W_{0,\mathrm{rad}}^{1, p(\cdot)}(\mathcal{A})$ (see \eqref{9} for the definition of the functional space $W_{0,\mathrm{rad}}^{1, p(\cdot)}(\mathcal{A})$).

    \subsection{Potential well and stationary problem}\label{sec2.3}
    In this subsection, we introduce a series of structures of the potential well theory utilized to deal with the global well-posedness of the solution of \eqref{1}-\eqref{3}. Firstly, we define the energy functional $J(u)$ and the Nehari functional $I(u)$ as follows, i.e,
        \begin{align}\label{45}
        J(u):=\int_{\mathcal{A}} \frac{|\nabla u(x)|^{p(x)}}{p(x)} \mathrm{d} x-\int_{\mathcal{A}} \frac{|u(x)|^{q(x)}}{q(x)} \mathrm{d} x
        \end{align}
    and
        \begin{align}\label{46}
          I(u):=\int_{\mathcal{A}}|\nabla u(x)|^{p(x)} \mathrm{d} x-\int_{\mathcal{A}}|u(x)|^{q(x)} \mathrm{d} x.
        \end{align}
    We need to point out that, $J(u)$ and $I(u)$ are not well defined on the whole space $W_{0}^{1, p(\cdot)}(\mathcal{A})$ because of lack of the integrability of the last term of $J(u)$ (resp. $I(u)$), i.e., $\int_{\mathcal{A}} \frac{|u(x)|^{q(x)}}{q(x)} \mathrm{d} x$ (resp. $\int_{\mathcal{A}}|u(x)|^{q(x)} \mathrm{d} x$) with given $u \in W_{0}^{1, p(\cdot)}(\mathcal{A})$. Specifically, the multiple-regime phenomenon within $\mathcal{A}$, i.e., the hypothesis $(\mathrm{H_{D}})$ causes the failure of the embedding $W_{0}^{1, p(\cdot)}(\mathcal{A}) \hookrightarrow L^{q(\cdot)}(\mathcal{A})$ for the subcritical regime case and the phase space $W_{0}^{1, p(\cdot)}(\mathcal{A})$ can not provide the $L^{q(\cdot)}$-integrability of $u$ in such situation.

    Fortunately, with the embedding given by Lemma \ref{l5} and the radial symmetry of the annulus $\mathcal{A}$, $J(u)$ and $I(u)$ are well-defined on $W_{0,\mathrm{rad}}^{1, p(\cdot)}(\mathcal{A})$ and of class $C^{1}(W_{0,\mathrm{rad}}^{1, p(\cdot)}(\mathcal{A}),\mathbb{R})$.
    Then the corresponding Nehari manifold for $I(u)$ is given by
        \begin{align}\label{18}
          \mathcal{N}:=\left\{u \in W_{0,\mathrm{rad}}^{1, p(\cdot)}\left(\mathcal{A}\right)\backslash \{0\}: I(u)=0\right\}.
        \end{align}
     Then the space $W_{\mathrm{rad}}^{1, p(\cdot)}(\mathcal{A})$ is split into the following two disjoint sets by $\mathcal{N}$, i.e.,
        \begin{align}\label{21}
          \mathcal{N}^{-}:=\left\{u \in W_{0,\mathrm{rad}}^{1, p(\cdot)}\left(\mathcal{A}\right): I(u)<0\right\}
        \end{align}
     and
        \begin{align}\label{22}
          \mathcal{N}^{+}:=\left\{u \in W_{0,\mathrm{rad}}^{1, p(\cdot)}\left(\mathcal{A}\right): I(u)>0\right\}
          \cup\{0\}.
        \end{align}
    Moreover, the depth of the potential well is defined as
        \begin{align}\label{53}
          d := \inf _{u \in W_{0,\mathrm{rad}}^{1, p(\cdot)}\left(\mathcal{A}\right) \backslash\{0\}} \sup _{\lambda \geq 0} J(\lambda u).
        \end{align}

    Next, we give several relevant properties of the structures of the potential well theory introduced above with detailed discussion, including:
    \begin{enumerate}[(i)]
    \item
    scaling properties of the mapping $j(\lambda):= J(\lambda u)$;

    \item
    properties about the potential well depth and the Nehari functional $I(u)$;

    \item
    variational structures to characterize the supercritical energy level;

    \item
    the stationary problem related to the evolution problem \eqref{1}-\eqref{3}.
  \end{enumerate}

  It's worth mentioning that although constant version of these properties are well established in constant exponent case, to extend them into the multiple-regime case of the variable exponents, it is necessary to check carefully the effects of the variable exponents $p(x),q(x)$ upon these properties.

  Besides, there seem to be few results discussing the related properties of the mapping $j(\lambda):= J(\lambda u)$ as well as the equivalent definition of the potential well depth in the variable exponent setting, which are fundamental and vital for the application of the potential well theory towards the evolutionary PDEs with variable exponents, even for the subcritical regime case.\\

    $\bullet$ {\bf Properties of the mapping $j(\lambda)$.}
    We introduce the mapping $j(\lambda):= J(\lambda u)$ for $\lambda \ge 0$, i.e.,
        \begin{align}\label{12}
          j(\lambda) =
          \int_{\mathcal{A}} \frac{\lambda^{p(x)}}{p(x)}|\nabla u(x)|^{p(x)} \mathrm{d} x
          - \int_{\mathcal{A}} \frac{\lambda^{q(x)}}{q(x)} |u(x)|^{q(x)} \mathrm{d} x,
        \end{align}
    whose properties are used to prove the asymptotic behavior and the finite-time blow-up of the solution with the subcritical initial energy, see the proofs of Theorem \ref{t3} and Theorem \ref{t4}.

    We have to emphasize that, due to the lack of homogeneity of the modular $\rho(u)$ caused by anisotropy of the variable exponent Lebesgue spaces, i.e., $\int_{\mathcal{A}}|\lambda u|^{p(x)} \mathrm{d} x \not\equiv \lambda^{p(x)} \int_{\mathcal{A}}|u|^{p(x)} \mathrm{d} x$ with a constant $\lambda >0$, different from the constant exponent case (i.e., $p(x)\equiv p, q(x)\equiv q$), the continuity and the differentiability of the mapping $j(\lambda)$ are not immediate.
    Specifically, if we ``extract'' the scalar $\lambda$ from $j(\lambda)$, i.e.,
        \begin{align*}
          \check{j}(\lambda) \le
            j(\lambda)
            \le \hat{j}(\lambda)
        \end{align*}
    with $\check{j}(\lambda)$ and $\hat{j}(\lambda)$ separately given by
        \begin{align*}
           \hat{j}(\lambda)
           :=& \max \left\{\lambda^{p^{+}}, \lambda^{p^{-}}\right\} \int_{\mathcal{A}} \frac{|\nabla u(x)|^{p(x)}}{p(x)} \mathrm{d} x
          -\min \left\{\lambda^{q^{+}}, \lambda^{q^{-}}\right\} \int_{\mathcal{A}} \frac{|u(x)|^{q(x)}}{q(x)} \mathrm{d} x
        \end{align*}
    and
        \begin{align*}
              \check{j}(\lambda)
          :=& \min \left\{\lambda^{p^{+}}, \lambda^{p^{-}}\right\} \int_{\mathcal{A}} \frac{|\nabla u(x)|^{p(x)}}{p(x)} \mathrm{d} x
          - \max \left\{\lambda^{q^{+}}, \lambda^{q^{-}}\right\} \int_{\mathcal{A}} \frac{|u(x)|^{q(x)}}{q(x)} \mathrm{d} x,
        \end{align*}
    it is easy to check that although $\check{j}(\lambda)$ and $ \hat{j}(\lambda)$ are continuous for $\lambda >0$,
    but no longer differentiable (specifically at the point $\lambda = 1$).
    Therefore, we give the following lemma to discuss these two properties of $j(\lambda)$.

        \begin{lemma}[Continuity and differentiability of $j(\lambda)$]\label{l3}
        Assume that $(\mathrm{H_{p,q}})$, $(\mathrm{H^1_{p,q}})$ and $(\mathrm{H_{D}})$ hold, then for any fixed $u\in  W_{0,\mathrm{rad}}^{1, p(\cdot)}\left(\mathcal{A}\right) \backslash \{0\}$, the mapping $j(\lambda)$ is continuous for $\lambda \in [0,\infty)$ and differentiable for $\lambda \in (0,\infty)$. Additionally,
                  \begin{align}\label{14}
                  \frac{\rm d}{{\rm d} \lambda} j(\lambda)
                  = \frac{1}{\lambda} I(\lambda u),
                  \quad \lambda \in (0,\infty).
                \end{align}
        \end{lemma}
        \begin{proof}
          We first claim that $j(\lambda)$ is continuous for $\lambda \in [0,\infty)$. Note that the mapping $\lambda \rightarrow \lambda^{p(x)}$ is continuous in $\lambda$ for every fixed $x\in\mathcal{A}$, which means for any fixed $x$ and sufficiently small $\varepsilon >0$, there exists a $\delta(\varepsilon)>0$ such that
            \begin{align*}
              \left|\lambda^{p(x)} - \lambda_0^{p(x)}\right| < \varepsilon
            \end{align*}
          for $|\lambda - \lambda_0| < \delta(\varepsilon)$, and the same fact also happens for the mapping $\lambda \rightarrow \lambda^{q(x)}$.
          Therefore, from the definition of $J(u)$, we get
            \begin{align*}
            \begin{aligned}
                  \left|j(\lambda)-j\left(\lambda_0\right)\right|
                  \le &
                  \int_{\mathcal{A}} \left(
                    \left|\lambda^{p(x)}
                  - \lambda_0^{p(x)}\right|
                    \frac{|\nabla u|^{p(x)}}{p(x)}
                  +  \left|\lambda^
                    {q(x)} - \lambda_0^{q(x)}\right| \frac{|u|^{q(x)}}{q(x)} \right)
                    {\rm d} x\\
                  \leq &
                  \varepsilon \int_{\mathcal{A}}
                  \left(
                    \frac{|\nabla u|^{p(x)}}{p(x)}
                    + \frac{|u|^{q(x)}}{q(x)}
                  \right)
                    {\rm d} x \le C\varepsilon
            \end{aligned}
            \end{align*}
          for $|\lambda - \lambda_0| < \delta(\varepsilon)$ and given $u\in  W_{0,\mathrm{rad}}^{1, p(\cdot)}\left(\mathcal{A}\right) \backslash \{0\}$, which claims the continuity of $j(\lambda)$ with respect to $\lambda$.

          Then it remains to check the differentiability of $j(\lambda)$ with respect to $\lambda$ over $(0,\infty)$. Let $\lambda_0 \in (0,\infty)$ be fixed, and $\{\lambda_i\}$ is an arbitrary sequence satisfying $\lambda_i \rightarrow \lambda_0$, then
          we define
                \begin{align*}
                  f_i(u):=\frac{\left(\lambda_i^{p(x)}
                    -\lambda_0^{p(x)}\right)|\nabla u|^{p(x)}}{\left(\lambda_i-\lambda_0\right) p(x)},
                  \quad
                  g_i(u):=\frac{\left(\lambda_i^{q(x)}
                  -\lambda_0^{q(x)}\right)|u|^{q(x)}}{\left(\lambda_i-\lambda_0\right) q(x)},
                \end{align*}
          and give the following difference quotient
                \begin{align*}
                  \frac{j\left(\lambda_i\right)
                    - j\left(\lambda_0\right)}{\lambda_i-\lambda_0}
                  = \int_{\mathcal{A}} \left(f_i(u) - g_i(u) \right) {\rm{d}} x.
                \end{align*}
          It is obvious that for a.e. $x\in \mathcal{A}$,
              $$f_i(u) \rightarrow \lambda_0^{p(x)-1} |\nabla u|^{p(x)} \quad \text{and } \quad
              g_i(u) \rightarrow \lambda_0^{q(x)-1} |u|^{q(x)}$$
          as $i \rightarrow \infty$ and the corresponding limits are both integrable. According to the boundedness of $p(x),q(x)$ on $\mathcal{A}$, and the embedding
          $W_{\mathrm{rad}}^{1, p(\cdot)}\left(\mathcal{A}\right) \hookrightarrow L^{q(\cdot)}\left(\mathcal{A}\right)$ given by Lemma \ref{l5}, the Lebesgue dominated convergence theorem claims that
                \begin{align*}
                  \lim \limits_{i \rightarrow \infty} \frac{j\left(\lambda_i\right)
                        - j\left(\lambda_0\right)}{\lambda_i-\lambda_0}
                    =
                        \int_{\mathcal{A}} \lambda_0^{p(x)-1} |\nabla u|^{p(x)} {\rm d}x
                        - \int_{\mathcal{A}} \lambda_0^{q(x)-1} |u|^{q(x)} {\rm d}x ,
                \end{align*}
          then \eqref{14} follows from the arbitrariness of the sequence $\{\lambda_i\}$, i.e.,
                \begin{align*}
                  \frac{\rm d}{{\rm d} \lambda} j(\lambda)
                  = \int_{\mathcal{A}} \lambda^{p(x)-1} |\nabla u|^{p(x)} {\rm d}x
                        - \int_{\mathcal{A}} \lambda^{q(x)-1} |u|^{q(x)} {\rm d}x = \frac{1}{\lambda} I(\lambda u).
                \end{align*}
        \end{proof}

        \begin{remark}[On Lemma \ref{l3}]
According to the proof of Lemma \ref{l3}, the continuity and differentiability of $j(\lambda)$ only rely on the boundedness of the variable exponents $p(x)$ and $q(x)$, rather than their smoothness (in other words, the logarithmic continuity \eqref{4} is relatively strong),
          for which the assertion of Lemma \ref{l3} are valid for both the subcritical regime $\inf \limits_{x \in \mathcal{A}}\left(p^{*}(x)-q(x)\right)>0$ and the multiple-regime case studied in this paper.
          Additionally, one can easily check that for $p(x), q(x) \in (1,2)$ with any $x \in \mathcal{A}$, the conclusion of Lemma \ref{l3} still holds true.
        \end{remark}

        \begin{lemma}[Qualitative properties of $j(\lambda)$]\label{l8}
            Assume that $(\mathrm{H_{p,q}})$, $(\mathrm{H^1_{p,q}})$-$(\mathrm{H^2_{p,q}})$ and $(\mathrm{H_{D}})$ hold, then for every $u \in W_{0,\mathrm{rad}}^{1, p(\cdot)}\left(\mathcal{A}\right) \backslash \{0\}$
            \begin{align*}
              \lim \limits_{\lambda \rightarrow 0^{+}}
                            j(\lambda)=0
              \quad\text{and}\quad \lim \limits_{\lambda \rightarrow \infty} j(\lambda) = -\infty.
            \end{align*}
        \end{lemma}
        \begin{proof}
            Firstly, $\lim \limits_{\lambda \rightarrow 0^{+}} j(\lambda)=0$ follows from the continuity of $j(\lambda)$ with respect to $\lambda$ given by Lemma \ref{l3}, then it remains to check the latter assertion. From ($\mathrm{H^{2}_{p,q}}$), we have
                \begin{align}\label{59}
                \begin{aligned}
                  j(\lambda)
                  \leq& \int_{\mathcal{A}} \frac{\lambda^{p(x)}}{p(x)}|\nabla u(x)|^{p(x)} \mathrm{d} x
                    - \int_{\mathcal{A}} \frac{\lambda^{q(x)}}{q(x)} |u(x)|^{q(x)} \mathrm{d} x\\
                  \leq& \max \left\{\lambda^{p^{+}}, \lambda^{p^{-}}\right\} \int_{\mathcal{A}} \frac{|\nabla u(x)|^{p(x)}}{p(x)} \mathrm{d} x
                    -\min \left\{\lambda^{ q^+}, \lambda^{ q^-}\right\} \int_{\mathcal{A}} \frac{|u(x)|^{q(x)}}{q(x)} \mathrm{d} x,
                \end{aligned}
                \end{align}
            then ($\mathrm{H^{2}_{p,q}}$) implies that for sufficiently large $\lambda$,
                $$\max \left\{\lambda^{p^{+}}, \lambda^{p^{-}}\right\}
                = \lambda^{p^{+}} < \lambda^{ q^-} = \min \left\{\lambda^{ q^+}, \lambda^{ q^-}\right\}.$$
            Therefore,
            $j(\lambda) \rightarrow -\infty$ as $\lambda \rightarrow \infty$.
        \end{proof}

        \begin{remark}[On Lemma \ref{l8}]
        We see that Lemma \ref{l8} is also
        valid for both the subcritical case $\inf \limits_{x \in \mathcal{A}}\left(p^{*}(x)-q(x)\right)>0$ (see Lemma 4.1 in \cite{r2})
        and the multiple-regime case, and the condition $(\mathrm{H^2_{p,q}})$ matters in terms of the qualitative shape of $j(\lambda)$, rather than the upper bound of $q(x)$.
        \end{remark}

        \begin{lemma}[Scaling properties of $j(\lambda)$]\label{l12}
           Assume that $(\mathrm{H_{p,q}})$, $(\mathrm{H^1_{p,q}})$-$(\mathrm{H^2_{p,q}})$ and $(\mathrm{H_{D}})$ hold and $u \in W_{0,\mathrm{rad}}^{1, p(\cdot)}\left(\mathcal{A}\right) \backslash \{0\}$, then we possess
                \begin{enumerate}[(i)]	
	                \item   for every
                            $u \in W_{0,\mathrm{rad}}^{1, p(\cdot)}\left(\mathcal{A}\right) \backslash \{0\}$, there exists a unique $\lambda^{*} :=
                            \lambda^{*}(u) > 0$
                            such that $\frac{\rm d}{{\rm d} \lambda} \left. j(\lambda) \right|_{\lambda = \lambda^{*} }=0$ and
                            $$\max \limits_{\lambda\ge0}j(\lambda) = j(\lambda^{*}),$$
                        with
                            \begin{align*}
                                \min \left\{ \Theta^{\frac{1}{q^+ - p^-}}(u),
                                            \Theta^{\frac{1}{q^- - p^+} }(u) \right\}
                                \leq \lambda^*(u)
                                \leq \max \left\{ \Theta^{\frac{1}{q^+ - p^-}}(u),
                                            \Theta^{\frac{1}{q^- - p^+} }(u) \right\},
                            \end{align*}
                        where
                            \begin{align}\label{15}
                                \Theta(u) :=  \frac{\int_{\mathcal{A}} |\nabla u(x)|^{p(x)} \mathrm{d} x}{\int_{\mathcal{A}} |u(x)|^{q(x)} \mathrm{d} x};
                            \end{align}
                    \item in addition, we have that $0<\lambda^{*}<1, \lambda^{*}=1$ and $\lambda^{*}>1$ provided $I(u)<0, I(u)=0$ and $I(u)>0$, respectively.
            \end{enumerate}
        \end{lemma}
        \begin{proof}
        (i) From \eqref{14} and the definition of $I(u)$, we have
                \begin{align}\label{35}
                  \frac{\rm d}{{\rm d} \lambda} j(\lambda) = \frac{1}{\lambda} I(\lambda u) = \int_{\mathcal{A}} \lambda^{p(x)-1}\left(|\nabla u|^{p(x)}-\lambda^{q(x)-p(x)}|u|^{q(x)}\right) \mathrm{d} x.
                \end{align}
            Then we define $i(\lambda)$ as follows
                \begin{align}\label{36}
                  i(\lambda)
                  := \int_{\mathcal{A}} |\nabla u|^{p(x)} \mathrm{d} x
                  - \int_{\mathcal{A}} \lambda^{q(x)-p(x)}|u|^{q(x)} \mathrm{d} x.
                \end{align}
            Obviously, $i(\lambda)$ shares the same sign with
            $\frac{\rm d}{{\rm d} \lambda} j(\lambda)$ according to \eqref{35} and \eqref{36}. Next, we claim the uniqueness of $\lambda^{*}$.
            Similar to the proof of Lemma \ref{l3}, the derivative of $i(\lambda)$ with respect to $\lambda$ can be computed as
                \begin{align}\label{17}
                  \frac{\rm d}{{\rm d} \lambda} i(\lambda) =
                -\int_{\mathcal{A}} (q(x)-p(x))\lambda^{q(x)-p(x)-1}|u|^{q(x)} \mathrm{d} x.
                \end{align}
            Obviously, $\frac{\rm d}{{\rm d} \lambda} i(\lambda) <0$ for $\lambda \in (0,\infty)$ according to $(\mathrm{H^2_{p,q}})$.
            Moreover, since
            $i(0) = \int_{\mathcal{A}} |\nabla u|^{p(x)} \mathrm{d} x >0$
            (the inequality holds strictly by recalling that $u \in W_{0,\mathrm{rad}}^{1, p(\cdot)}\left(\mathcal{A}\right) \backslash \{0\}$), there holds for sufficiently large $\large \lambda$
                \begin{align*}
                  i(\lambda)
                  \le  \int_{\mathcal{A}} |\nabla u|^{p(x)} \mathrm{d} x
                  - \min\{\lambda^{q^--p^+}, \lambda^{q^+-p^-}\} \int_{\mathcal{A}} |u|^{q(x)} \mathrm{d} x <0,
                \end{align*}
            which together with the monotonicity of $i(\lambda)$ for $\lambda \in (0,\infty)$ claims
            the existence and uniqueness of $\lambda^{*}$ such that $i(\lambda^{*}) = 0$, i.e., $\frac{\rm d}{{\rm d} \lambda} \left. j(\lambda) \right|_{\lambda = \lambda^{*} }=0$.
            Furthermore, in view of \eqref{17}, $i(\lambda)$ is strictly decreasing on $(0,\infty)$, which together with the fact that $i(\lambda)$ shares the same sign with $\frac{\rm d}{{\rm d}\lambda}j(\lambda)$ (see \eqref{35} and \eqref{36}) implies that $\frac{\rm d}{{\rm d}\lambda}j(\lambda)$ changes sign from positive to negative exactly once at $\lambda = \lambda^{*}$. Therefore, $\lambda = \lambda^{*}$ is the unique maximum point of $j(\lambda)$ for $\lambda \in (0,\infty)$.
            Moreover, according to \eqref{36}, $i(\lambda^{*}) = 0$ implies that
                \begin{align*}
                  \min \left\{{\lambda^{*}}^{q^{-}-p^{+}},
                        {\lambda^{*}}^{q^{+}-p^{-}}\right\}
                  \leq \frac{\int_{\mathcal{A}}|\nabla u|^{p(x)} \mathrm{d} x}{\int_{\mathcal{A}}|u|^{q(x)} \mathrm{d} x}
                  \leq \max \left\{{\lambda^*}^{{q^{-}-p^{+}}}, {\lambda^*}^{q^{+}-p^{-}}\right\},
                \end{align*}
            which implies the last assertion of (i).

        (ii) On the one hand, according to the proof of (i) and the monotonicity of $i(\lambda)$ with respect to $\lambda$, we have
            \begin{align}\label{255}
            i(\lambda)
            \begin{cases}
              > 0,  \quad  0< \lambda \le\lambda^*, \\
              = 0,  \quad  \lambda = \lambda^*, \\
              < 0,  \quad  \lambda^*< \lambda < \infty.
            \end{cases}
            \end{align}
        On the other hand, from the definition of $i(\lambda)$ (see \eqref{36}), $i(1) = I(u)$, therefore
            \begin{align*}
              i(1)>0\ (=0 \ \text{or} <0)
              \quad \text{  if and only if  } \quad
              I(u)>0\ (=0\ \text{or} <0),
            \end{align*}
        which together with \eqref{255} completes the proof of (ii).
        \end{proof}

        \begin{remark}[On Lemma \ref{l12}]
        The statement (ii) of Lemma \ref{l12} can be also derived from Lemma \ref{l12} (i). From the definition of $\Theta(u)$, i.e., \eqref{15}, there holds
            \begin{align*}
              \Theta(u) > 1\ (=1 \ \text{or} <1)
              \quad \text{ if and only if } \quad
              I(u)>0\ (=0\ \text{or} <0).
            \end{align*}
        \end{remark}

        $\bullet$ {\bf Properties about the potential well depth $d$ and the Nehari functional $I(u)$.}
        We apply the similar arguments as Payne and Sattinger's work \cite{r8} (for the case of the constant exponent) to give an equivalent definition of the potential well depth as follows.

        \begin{lemma}[Equivalent definitions of the potential well depth]\label{l13}
            Assume that $(\mathrm{H_{p,q}})$, $(\mathrm{H^1_{p,q}})$-$(\mathrm{H^2_{p,q}})$ and $(\mathrm{H_{D}})$ hold, then
                \begin{align}\label{37}
                  d = \inf _{u \in W_{0,\mathrm{rad}}^{1, p(\cdot)}\left(\mathcal{A}\right) \backslash \{0\}} \sup _{\lambda \geq 0} J(\lambda u)
                = \inf _{u \in \mathcal{N}} J(u).
                \end{align}
        \end{lemma}
        \begin{proof}
          From Lemma \ref{l12}, for any $v \in W_{0,\mathrm{rad}}^{1, p(\cdot)}\left(\mathcal{A}\right) \backslash \{0\}$, there exists a unique $\lambda^{*}(v) > 0$
          such that $\lambda^{*}(v) v \in \mathcal{N}$ and $J(\lambda v)$ takes its maximum at
          $\lambda = \lambda^{*}(v)$, i.e.,
                \begin{align*}
                  J(\lambda^*(v) v) = \sup _{\lambda \geq 0} J(\lambda v).
                \end{align*}
          By setting $u = \lambda^{*}(v) v$, we can do a normalization such that $\lambda^{*}(u) = 1$ for any $u \in W_{0,\mathrm{rad}}^{1, p(\cdot)}(\mathcal{A}) \backslash\{0\}$, which claims that
                \begin{align}\label{20}
                      \inf _{u \in W_{0,\mathrm{rad}}^{1, p(\cdot)}(\mathcal{A}) \backslash\{0\}} \sup _{\lambda \geq 0} J(\lambda u)
                    = \inf _{u \in W_{0,\mathrm{rad}}^{1, p(\cdot)}(\mathcal{A}) \backslash\{0\}} J(\lambda^*(u) u)
                    = \inf _{u \in \mathcal{N}} J(u).
                \end{align}
          Thus \eqref{37} holds.
        \end{proof}

        \begin{remark}[On Lemma \ref{l13}]
          We point out that, the validity of Lemma \ref{l13} relies on the uniqueness of $\lambda^*$ ensured by Lemma \ref{l12} (i), which requires $(\mathrm{H^2_{p,q}})$. Therefore, we can apply Lemma \ref{l12} (ii) to obtain the last equal sign in \eqref{20}.
          Otherwise, if $\lambda^*(u)$ is not unique for
          $u \in W_{0,\mathrm{rad}}^{1, p(\cdot)}(\mathcal{A})$ without $(\mathrm{H^2_{p,q}})$, then the mapping $u \mapsto \lambda^*(u) u$ is not surjective from
          $W_{0,\mathrm{rad}}^{1, p(\cdot)}(\mathcal{A})$ to $\mathcal{N}$.
        \end{remark}

        Next, in Lemma \ref{l14} below, we give a characterization of the relationship between the sign of $I(u)$ and the norm of $u$ in $W_{0,\mathrm{rad}}^{1, p(\cdot)}\left(\mathcal{A}\right)$ with the radius $r_{{\rm ptw}}$ defined by
                        \begin{align}\label{25}
                          r_{{\rm ptw}} := \min \left\{\delta_{pq}^{\frac{1}{ p^+ - q^-}},
                                \delta_{pq}^{\frac{1}{ p^- - q^+}} \right\},
                        \end{align}
        where the constant $\delta_{pq}$ is given by
            \begin{align}\label{54}
              \delta_{pq} :=
                            \max \left\{ \left( C^{*}_{p(\cdot),q(\cdot)} \right)^{q^-},
                        \left( C^{*}_{p(\cdot),q(\cdot)} \right)^{q^+} \right\},
            \end{align}
        the constant $C^{*}_{p(\cdot),q(\cdot)}$ is the best constant of the Sobolev embedding $W_{0,\mathrm{rad}}^{1, p(\cdot)}\left(\mathcal{A}\right) \hookrightarrow L^{q(\cdot)}\left(\mathcal{A}\right)$,
        and the subscript of the notation $r_{{\rm ptw}}$ stands for the abbreviation of the phrase ``potential well''.

        \begin{lemma}\label{l14}
                 Let $(\mathrm{H_{p,q}})$, $(\mathrm{H^1_{p,q}})$-$(\mathrm{H^2_{p,q}})$ and $(\mathrm{H_{D}})$ be satisfied,
                  then for any
                  $u \in W_{0,\mathrm{rad}}^{1, p(\cdot)}\left(\mathcal{A}\right)$, we have
                    \begin{enumerate}[(i)]
                        \item if $I(u)<0,$ then $\|\nabla u\|_{p (\cdot),\mathcal{A}} >r_{{\rm ptw}}$;
                        \item if $I(u)=0,$ then\ $\|\nabla u\|_{p (\cdot),\mathcal{A}}\ge r_{{\rm ptw}}$ or\ $\|\nabla u\|_{p(\cdot),\mathcal{A}}=0$;
                        \item if $0<\|\nabla u\|_{p (\cdot),\mathcal{A}} < r_{{\rm ptw}}$, then $I(u)>0$;
                        \item the depth of the potential well is bounded below, i.e., $d \ge \tilde{d} >0$,
                        where
                                \begin{align}\label{38}
                                  \tilde{d} := \left( \frac{1}{p^+} - \frac{1}{q^-} \right)
                                  \min \left\{ \left( C^{*}_{p(\cdot),q(\cdot)} \right)^{\frac{p^+ q^-}{ p^+ - q^-}}, \left( C^{*}_{p(\cdot),q(\cdot)} \right)^{\frac{p^- q^+}{ p^- - q^+}} \right\}.
                                \end{align}
                    \end{enumerate}
        \end{lemma}
        \begin{proof}
                    (i) Firstly, Proposition \ref{p1} asserts
                    \begin{align}
                      \int_{\mathcal{A}} |\nabla u|^{p(x)} {\rm d}x
                      \ge \min \left\{\|\nabla u\|_{p (\cdot),\mathcal{A}}^{p^{-}},\|\nabla u\|_{p(
                      \cdot),\mathcal{A}}^{p^{+}}\right\}. \label{10}
                    \end{align}
                Then, to determine the power index in \eqref{10} and continue the proof, we discuss the cases $\|\nabla u\|_{p (\cdot),\mathcal{A}} \le 1$ and $\|\nabla u\|_{p (\cdot),\mathcal{A}} > 1$.
                For $\|\nabla u\|_{p (\cdot),\mathcal{A}} \le 1$, it follows from $I(u)<0$ that
                    \begin{align}\label{5}
                    \begin{aligned}
                    \|\nabla u\|_{p (\cdot),\mathcal{A}}^{p^{+}}
                    &\leq \int_{\mathcal{A}}|\nabla u(x)|^{p(x)} \mathrm{d} x
                    <\int_{\mathcal{A}}|u(x)|^{q(x)} \mathrm{d} x \\
                    & \leq \max \left\{\|u\|_{q(\cdot),\mathcal{A}}^{q^{-}},\|u\|_{q(\cdot),\mathcal{A}}^{q^{+}}\right\}  \\
                    & \leq \max \left\{
                    \left( C^{*}_{p(\cdot),q(\cdot)} \right)^{q^-}
                    \|\nabla u\|_{p (\cdot),\mathcal{A}}^{q^{-}},
                    \left( C^{*}_{p(\cdot),q(\cdot)} \right)^{q^+}
                    \|\nabla u\|_{p (\cdot),\mathcal{A}}^{q^{+}}\right\} \\
                    & \leq \delta_{pq}\|\nabla u\|_{p (\cdot),\mathcal{A}}^{q^{-}}
                    \end{aligned}
                    \end{align}
                 with $\delta_{pq}$ given by \eqref{54}.
                 For $\|\nabla u\|_{p (\cdot),\mathcal{A}} > 1$, we can similarly get
                    \begin{align}\label{6}
                      \|\nabla u\|_{p (\cdot),\mathcal{A}}^{p^{-}}
                      <& \delta_{pq} \|\nabla u\|_{p (\cdot),\mathcal{A}}^{q^{+}}.
                    \end{align}
                 Then combining \eqref{5} and \eqref{6}, we obtain
                    $$
                      \|\nabla u\|_{p (\cdot),\mathcal{A}} >  \min \left\{\delta_{pq}^{\frac{1}{ p^+ - q^-}}, \delta_{pq}^{\frac{1}{ p^- - q^+}} \right\}
                      = r_{{\rm ptw}},
                    $$
                  where $r_{{\rm ptw}}$ is given by \eqref{25}.

                  (ii)
                  Since the case $\|\nabla u\|_{p (\cdot),\mathcal{A}} = 0$ is relatively trivial, we directly consider
                  the case $I(u)=0$ with $\|\nabla u\|_{p (\cdot),\mathcal{A}} \neq 0$. Then, in this situation, we can obtain that $\|\nabla u\|_{p (\cdot),\mathcal{A}} \ge r_{\rm ptw}$ from
                    \begin{align*}
                      \int_{\mathcal{A}}|\nabla u|^{p(x)}  {\rm d}x
                      \le \int_{\mathcal{A}} |u|^{q(x)} {\rm d} x,
                    \end{align*}
                  which is analogous to the proof of (i).

              (iii) Arguing by contradiction, we suppose that $I(u) \le 0$, then the condition $0<\|\nabla u\|_{p (\cdot),\mathcal{A}}<r_{{\rm ptw}}$ contradicts (i) and (ii) of this lemma, which asserts the validity of (iii).

              (iv) Taking any $u \in \mathcal{N}$, then it follows from the definitions of $I(u), J(u)$ and (ii) of this lemma that
                \begin{align}\label{39}
                \begin{aligned}
                  J(u)
                  \geq& \frac{1}{p^{+}} \int_{\mathcal{A}}|\nabla u(x)|^{p(x)} \mathrm{d} x-\frac{1}{q^{-}} \int_{\mathcal{A}}|u(x)|^{q(x)} \mathrm{d} x\\
                  =& \left(\frac{1}{p^{+}}-\frac{1}{q^{-}}\right) \int_{\mathcal{A}}|\nabla u(x)|^{p(x)} \mathrm{d} x + \frac{1}{q^-} I(u)\\
                  \ge& \left(\frac{1}{p^{+}}-\frac{1}{q^{-}}\right) \min \left\{\|\nabla u\|_{p (\cdot),\mathcal{A}}^{p^{-}},\|\nabla u\|_{p(
                      \cdot),\mathcal{A}}^{p^{+}}\right\}\\
                  \ge& \left(\frac{1}{p^{+}}-\frac{1}{q^{-}}\right) \min \left\{r_{{\rm ptw}}^{p^-}, r_{{\rm ptw}}^{p^+} \right\}.
                \end{aligned}
                \end{align}
         Thus \eqref{38} follows from \eqref{39} according to the arbitrariness of $u$.
    \end{proof}

    Now, we give two estimates related with $I(u)$ at the subcritical energy level, in order to deal with the asymptotic behavior and finite-time blow-up of the solution for $0<J(u_0)<d$.

    \begin{lemma}[Refined estimates of $I(u)$]\label{l6}
            Assume that $(\mathrm{H_{p,q}})$, $(\mathrm{H^1_{p,q}})$-$(\mathrm{H^2_{p,q}})$ and $(\mathrm{H_{D}})$ hold and $u \in W_{0,\mathrm{rad}}^{1, p(\cdot)}\left(\mathcal{A}\right) \backslash \{0\}$ with $J(u)< d$, then the following statements hold true:
              \begin{enumerate}[(i)]
                    \item if $I(u)>0$, then
                            \begin{align}\label{33}
                            \begin{aligned}
                               I(u)
                               >&
                               \left(
                               1- \left(\frac{d}{J(u)
                               }\right)
                                ^{\frac{p^{+}-q^{-}}{q^{+}}}\right)
                                \int_{\mathcal{A}}|\nabla u|^{p(x)} \mathrm{d} x;
                            \end{aligned}
                            \end{align}
                    \item if $I(u)<0$, then
                            \begin{align}\label{34}
                              I(u)< q^- \left( J(u) -d \right).
                            \end{align}
              \end{enumerate}
            \end{lemma}
            \begin{proof}
              (i) By the definition of $I(u)$ and Lemma \ref{l12} (ii), we have $\lambda^* >1$ and
                \begin{align}\label{28}
                \begin{aligned}
                    I\left(\lambda^{*} u \right)
                    \le& \left(\lambda^{*}\right)^{p^{+}} \int_{\mathcal{A}}|\nabla u|^{p(x)} \mathrm{d} x
                        - \left(\lambda^{*}\right)^{q^{-}} \int_{\mathcal{A}}|u|^{q(x)} \mathrm{d} x \\
                    =& \left(\lambda^{*}\right)^{q^{-}} I(u)
                        + \left( \left(\lambda^{*}\right)^{p^{+}}
                        - \left(\lambda^{*}\right)^{q^{-}}\right) \int_{\mathcal{A}}|\nabla u|^{p(x)} \mathrm{d} x,
                \end{aligned}
                \end{align}
            which implies that
                \begin{align}\label{32}
                \begin{aligned}
                  I(u)
                  \geq& \left(1-\left(\lambda^{*}\right)
                  ^{p^{+}-q^{-}}\right) \int_{\mathcal{A}}|\nabla u|^{p(x)} \mathrm{d} x.
                \end{aligned}
                \end{align}
            Next, we estimate $\lambda^{*}$. From the definitions of $J(u),I(u)$ and $d$, there holds
                \begin{align}\label{31}
                \begin{aligned}
                d
                \leq& J\left(\lambda^{*} u\right)
                = J\left(\lambda^{*} u\right)-\frac{1}{q^{-}} I\left(\lambda^{*} u\right) \\
                \leq& \left(\lambda^{*}\right)^{p^{+}}
                    \int_{\mathcal{A}}\left(\frac{1}{p(x)}
                    -\frac{1}{q^{-}}\right) |\nabla u|^{p(x)} \mathrm{d} x
                    - \left(\lambda^{*}\right)^{q^{+}}
                      \int_{\mathcal{A}}\left(\frac{1}{q(x)}
                      -\frac{1}{q^{-}}\right)
                |u|^{q(x)} \mathrm{d} x \\
                \leq& \left(\lambda^{*}\right)^{q^{+}}
                \left(J(u)-\frac{1}{q^{-}} I(u)\right),
                \end{aligned}
                \end{align}
            where $\lambda^{*}>1$ according to Lemma \ref{l12} (ii).
            Since $I(u)>0$, from \eqref{31}, we have
                \begin{align}\label{90}
                  J(u) > J(u)-\frac{1}{q^{-}} I(u) > \frac{d}{\left(\lambda^{*}\right)^{q^{+}}},
                \end{align}
            which implies that
            $
            \lambda^{*} \geq\left(\frac{d}{J(u)}\right)^
            {\frac{1}{q^{+}}}
            $.
            Then \eqref{33} follows from (\ref{32}) and \eqref{90}.

            (ii)
            Firstly, since $I(u)<0$, we can derive that $0<\lambda^{*}<1$ according to Lemma \ref{l12} (ii). Then, recalling \eqref{31}, we can similarly get
                \begin{align*}
                \begin{split}
                \begin{aligned}
                d
                \leq& J\left(\lambda^{*} u\right)
                = J\left(\lambda^{*} u\right)-\frac{1}{q^{-}} I\left(\lambda^{*} u\right) \\
                \leq& \left(\lambda^{*}\right)^{p^{-}}
                    \int_{\mathcal{A}}\left(\frac{1}{p(x)}
                    -\frac{1}{q^{-}}\right) |\nabla u|^{p(x)} \mathrm{d} x
                    - \left(\lambda^{*}\right)^{q^{-}}
                      \int_{\mathcal{A}}\left(\frac{1}{q(x)}
                      -\frac{1}{q^{-}}\right)
                |u|^{q(x)} \mathrm{d} x \\
                \leq& \left(\lambda^{*}\right)^{p^{-}}
                \left(J(u)-\frac{1}{q^{-}} I(u)\right)
                < J(u)-\frac{1}{q^{-}} I(u).
                \end{aligned}
                \end{split}
                \end{align*}
            which tells (\ref{34}) immediately.
            \end{proof}

    $\bullet$ {\bf Variational structures for the supercritical initial energy level.}
    In this part, we introduce some structures of the potential well theory especially for the initial data at the supercritical energy level, i.e., $J(u_0)>d$.
    Next, we define the level set of $J(u)$ as follows
        $$
        J^{k}:=\left\{u \in W_{0,\mathrm{rad}}^{1, p(\cdot)}(\mathcal{A}): J(u) < k\right\}.
        $$
    Then, we show that the $W_{0,\mathrm{rad}}^{1, p(\cdot)}\left(\mathcal{A}\right)$ norms of the elements of $\mathcal{N}$ and $\mathcal{N}^+$ are both bounded above within $J^{k}$ in the following lemma.

            \begin{lemma}\label{l4}
               Assume that $(\mathrm{H_{p,q}})$, $(\mathrm{H^1_{p,q}})$-$(\mathrm{H^2_{p,q}})$ and $(\mathrm{H_{D}})$ hold.
              Suppose that $u \in J^{k}$ and $I(u)\ge 0$, then
                $ \|\nabla u\|_{p (\cdot),\mathcal{A}} < r_{{\rm ptw},k}$,
              where $r_{{\rm ptw},k}$ is defined by
                \begin{align}
                  &r_{{\rm ptw},k} :=
                  \max \left\{
                  \left(
                  \frac{p^+ q^- k}{q^- - p^+} \right)^{\frac{1}{p^-}},
                  \left(
                  \frac{p^+ q^- k}{q^- - p^+} \right)^{\frac{1}{p^+}} \right\}. \label{98}
                \end{align}
            \end{lemma}
            \begin{proof}
            From \eqref{39} , we have
              \begin{align*}
                  k > J(u)
                  >& \left( \frac{1}{p^+} - \frac{1}{q^-} \right)
                    \min \left\{
                       \|\nabla u\|_{p (\cdot),\mathcal{A}}^{p^{-}},
                       \|\nabla u\|_{p (\cdot),\mathcal{A}}^{p^{+}}
                       \right\}.
              \end{align*}
            For $\|\nabla u\|_{p (\cdot),\mathcal{A}} \le 1$, we get
                \begin{align*}
                  k > \left( \frac{1}{p^+} - \frac{1}{q^-}
                        \right)
                        \|\nabla u\|_{p (\cdot),\mathcal{A}}^{p^{+}},
                \end{align*}
            which implies that
                \begin{align}\label{19}
                  \|\nabla u\|_{p (\cdot),\mathcal{A}}
                  < \left(
                  \frac{p^+ q^- k}{q^- - p^+} \right)^{\frac{1}{p^+}}.
                \end{align}
            Similarly, for $\|\nabla u\|_{p (\cdot),\mathcal{A}} > 1$, there holds
                \begin{align}\label{27}
                  \|\nabla u\|_{p (\cdot),\mathcal{A}}
                  < \left(
                  \frac{p^+ q^- k}{q^- - p^+} \right)^{\frac{1}{p^-}}.
                \end{align}
            Therefore, (\ref{19}) and (\ref{27}) assert that
            $\|\nabla u\|_{p (\cdot),\mathcal{A}} < r_{{\rm ptw},k}$.
            \end{proof}

            \begin{remark}[On Lemma \ref{l4}]\label{rmk4}
               It is easy to check that, if we take
               $k = \tilde{d}$ (defined by \eqref{38}) in Lemma \ref{l4}, there holds
                $
                  \|\nabla u\|_{p(\cdot),\mathcal{A}}
                < r_{{\rm ptw}}
                $
              for $I(u) > 0$
              (resp. $
                  \|\nabla u\|_{p(\cdot),\mathcal{A}}
                \le r_{{\rm ptw}}
                $ for $I(u) = 0$).
            \end{remark}

    Then, we consider the intersection of $J^k$ and $\mathscr{N}$. Recalling Lemma \ref{l14} (ii), we can derive that
        \begin{align}\label{167}
          \operatorname{dist}(0, \mathcal{N}):=\inf _{u \in \mathcal{N}}\|\nabla u\|_{p (\cdot),\mathcal{A}} \geq r_{{\rm ptw}} >0,
        \end{align}
     where $r_{{\rm ptw}}$ is defined by \eqref{25}. Besides, Lemma \ref{l4} claims that
        \begin{align}\label{168}
           \sup_{u \in \mathcal{N}}\|\nabla u\|_{p (\cdot),\mathcal{A}} < r_{{\rm ptw},k},
        \end{align}
     where $r_{{\rm ptw},k}$ is given by \eqref{98} and is increasing with respect to $k$.
     Moreover, recalling that Lemma \ref{l4} implies $r_{{\rm ptw},k} = r_{{\rm ptw}}$ for $k = \tilde{d}$, see Remark \ref{rmk4}, and Lemma \ref{l14} (iv) tells that $d \ge \tilde{d}$, then for $k > d$ we can deduce that
     $r_{{\rm ptw},k} > r_{{\rm ptw}}$, which together with \eqref{167} and \eqref{168} implies
        $$
        \mathcal{N}_{k}:=\mathcal{N} \cap J^{k} \neq \emptyset.
        $$
    Then we define
        \begin{align*}
          \lambda_{k}:=\inf \left\{\|u\|_{2,\mathcal{A}}: u \in \mathcal{N}_{k}\right\},
          \quad
          \Lambda_{k}:=\sup \left\{\|u\|_{2,\mathcal{A}}: u \in \mathcal{N}_{k}\right\}
        \end{align*}
    and
        $
              \lambda_{\infty}:=\inf \left\{\|u\|_{2,\mathcal{A}}: u \in \mathcal{N}\right\}.
        $
    It is obvious that
        $k \mapsto \lambda_{k}$ is non-increasing, and
        $k \mapsto \Lambda_{k}$ is non-decreasing with respect to the energy level $k$. Then the next lemma show that under the hypotheses $(\mathrm{H_{p,q}})$, $(\mathrm{H^1_{p,q}})$-$(\mathrm{H^2_{p,q}})$ and $(\mathrm{H_{D}})$,
        both $\Lambda_{k}, \lambda_{k}$ are bounded above and below.

    \begin{lemma}\label{l9}
     Assume that $(\mathrm{H_{p,q}})$, $(\mathrm{H^1_{p,q}})$-$(\mathrm{H^2_{p,q}})$ and $(\mathrm{H_{D}})$ hold, then for $k>d$, there holds
                $
                0< \lambda_k \le \Lambda_{k}<\infty.
                $
    \end{lemma}
    \begin{proof}
        Firstly, we prove $\Lambda_{k}<\infty$. This assertion is obvious and follows from the fact that for any $u \in \mathcal{N}$ $\|\nabla u\|_{p (\cdot),\mathcal{A}} < r_{{\rm ptw},k}$ according to Lemma \ref{l4} and the embedding $W_{0,\mathrm{rad}}^{1, p(\cdot)}(\mathcal{A}) \hookrightarrow L^{2}(\mathcal{A})$.

        Next, we prove $\lambda_{k}>0$.
        It follows from $I(u) = 0$ that
            \begin{align}\label{101}
              \int_{\mathcal{A}}|\nabla u(x)|^{p(x)} \mathrm{d} x
              = \int_{\mathcal{A}}|u(x)|^{q(x)} \mathrm{d} x
              \le \|u\|_{q^-,\mathcal{A}}^{q^-}
                + \|u\|_{q^+,\mathcal{A}}^{q^+}.
            \end{align}
       Besides, since it follows from Lemma \ref{l5} (i) that $W_{0,\mathrm{rad}}^{1, p(\cdot)}(\mathcal{A}) \hookrightarrow L^{\infty}(\mathcal{A})$ and the following embeddings hold for $q^- >2$
           $$
           L^{\infty}(\mathcal{A}) \hookrightarrow
           L^{q^+}(\mathcal{A}) \hookrightarrow
           L^{q^-}(\mathcal{A}) \hookrightarrow L^{2}(\mathcal{A}),
           $$
       then the $L^p$-interpolation inequality yields
            \begin{align}\label{102}
              \|u\|_{q^{\pm},\mathcal{A}}
                  \leq C \|u\|_{2,\mathcal{A}}^{\theta_{\pm}} \|u\|_{\infty,\mathcal{A}}^{1-\theta_{\pm}}
                  \le C \|u\|_{2,\mathcal{A}}^{\theta_{\pm}} \|\nabla u\|_{p(\cdot),\mathcal{A}}^{1-\theta_{\pm}} ,
            \end{align}
       where $\theta_{+} := \frac{2}{q^+}$, $\theta_{-} := \frac{2}{q^-}$ and $C$ is some positive constant.
       Additionally, we also have
        \begin{align}\label{103}
          \int_{\mathcal{A}}|\nabla u|^{p(x)} \mathrm{d} x
                \ge \min \left\{
                           \|\nabla u\|_{p (\cdot),\mathcal{A}}^{p^{-}},
                           \|\nabla u\|_{p(\cdot),\mathcal{A}}^{p^{+}}
                           \right\}.
        \end{align}
       Then, it follows from \eqref{101}-\eqref{103} that
        \begin{align}\label{104}
          \|u\|_{2,\mathcal{A}}^{2}  \ge
              \begin{cases}
                C\|\nabla u\|_{p(\cdot),\mathcal{A}}^{p^- - q^+ + 2},
                    &  0 < \|\nabla u\|_{p(\cdot),\mathcal{A}} \le 1,\\
                C\|\nabla u\|_{p(\cdot),\mathcal{A}}^{p^+ - q^- + 2},
                    &  1< \|\nabla u\|_{p(\cdot),\mathcal{A}} < \infty.
              \end{cases}
        \end{align}
       Moreover, since $\|\nabla u\|_{p (\cdot),\mathcal{A}}$ is bounded above and below for given $u \in \mathcal{N}_{k}$ according to Lemma \ref{l14} (ii) and Lemma \ref{l4}, we can deduce that
        $$\min \left\{\|\nabla u\|^{p^- - q^+ + 2}_{p (\cdot),\mathcal{A}},\
            \|\nabla u\|^{p^+ - q^- + 2}_{p (\cdot),\mathcal{A}} \right\}\ge C$$
       with some positive constant $C$ only depending on $p^{\pm}, q^{\pm}$ and the best constant of the
       embedding $W_{0,\mathrm{rad}}^{1, p(\cdot)}(\mathcal{A}) \hookrightarrow L^{q(\cdot)}(\mathcal{A})$. Therefore, it follows from \eqref{104} that $\|u\|_{2,\mathcal{A}} >0 $, which claims $\lambda_{k} \ge C> 0$ with the arbitrariness of $u$.
    \end{proof}

    \begin{remark}[On Lemma \ref{l9}]\label{rmk5}
      The key point of the proof of Lemma \ref{l9} is to estimate the positive and uniform lower bound of $\|u\|_{2,\mathcal{A}}$ for any $ u \in \mathcal{N}_k$. Although the classical Gagliardo-Nirenberg interpolation inequality fails with $(\mathrm{H_{D}})$, we can recover the Gagliardo-Nirenberg type
      interpolation inequality by applying the $L^p$ interpolation inequality for $1<q^- \le q^+ <\infty$
      and the continuous embedding
      $W_{0,\mathrm{rad}}^{1, p(\cdot)}(\mathcal{A}) \hookrightarrow L^{\infty}(\mathcal{A})$ ensured by Lemma \ref{l5} (i), cf. \eqref{102}.
    \end{remark}

    Then, according to the proof of Lemma \ref{l9}, we can obtain the following Corollary immediately.

    \begin{corollary}\label{cor2}
      Assume that $(\mathrm{H_{p,q}})$, $(\mathrm{H^1_{p,q}})$-$(\mathrm{H^2_{p,q}})$ and $(\mathrm{H_{D}})$ hold. Suppose further
        \begin{align}\label{180}
          q^{+} < p^- + 2,
        \end{align}
      then for $k>d$, there holds
        $
        \lambda_{\infty} > 0.
        $
    \end{corollary}

    \begin{remark}[On Corollary \ref{cor2}]
      The constraint \eqref{180} is set to ensure that
        \begin{align*}
          p^+ - q^- + 2 \ge p^- - q^+ + 2 >0.
        \end{align*}
      Moreover, if the spatial dimension $n$ is appropriately low in the following sense
        \begin{align*}
          1 < \frac{n}{p^-} < \frac{p^+}{2} + 1,
        \end{align*}
      we may deduce that $p^+ + 2 < {p^-}^* = \frac{np^-}{n - p^-}$, in which situation the variable exponent $q(x)$ is of the subcritical regime within $\mathcal{A}$.
    \end{remark}

$\bullet$ {\bf The stationary problem corresponding to \eqref{1}-\eqref{3}.}
Note that any critical point of the energy functional $J(u)$ solves the following stationary problem in the weak distribution sense.
    \begin{align}\label{16}
        \begin{cases}
        -\Delta_{p(x)} u=|u|^{q(x)-2}u, & \text { in } \mathcal{A}, \\
          u(x)=0, & \text { on } \partial \mathcal{A},
        \end{cases}
    \end{align}
where $p(x)$ and $q(x)$ satisfy the hypotheses $(\mathrm{H_{p,q}})$, $(\mathrm{H_D})$ and $(\mathrm{H^1_{p,q}})$-$(\mathrm{H^2_{p,q}})$.
Then, the corresponding solution to \eqref{16} lies in the Nehari manifold $\mathcal{N}$ and the existence of a nontrivial solution to \eqref{16} suffices to show that $\mathcal{N}$ is nonempty, which is ensured by the following Lemma.
It is easy to check that $J(u)$ possesses the mountain-pass geometry and satisfies the Palais-Smale condition, for which the existence of \eqref{16} can be proved by the mountain pass theorem and we give the following Lemma \ref{l11} without proof.

        \begin{lemma}\label{l11}
          Let $(\mathrm{H_{p,q}})$, $(\mathrm{H^1_{p,q}})$-$(\mathrm{H^2_{p,q}})$ and $(\mathrm{H_{D}})$ be satisfied, then the stationary problem \eqref{16} admits a nontrivial solution.
        \end{lemma}

\section{Local existence and uniqueness}\label{sec3}
In this section, we employ the subdifferential technique with G-symmetry introduced in Subsection \ref{sec2.4} to investigate the local existence of the solution to \eqref{1}-\eqref{3} and prove its uniqueness by Gronwall's inequality.

Recalling that the annulus $\mathcal{A}$ is radially symmetric, thus it is invariant under the action of $\mathbf{O}(n)$ (the group of orthogonal transformations of dimension $n$) and compatible with $\mathbf{O}(n)$ in the sense of \eqref{108}, which provides the foundation for our application of the subdifferential technique with G-symmetry.
Compared with the Faedo–Galerkin’s method, the argument of the subdifferential technique in the functional setting suits the generalized Principle of Symmetric Criticality better, which was initially proposed for elliptic variational problems \cite{r56}.
Besides, the subdifferential technique in the reflexive Banach space setting consists of that in the Hilbert space setting, which benefits the proof of the energy identities, cf. Lemma \ref{l7}.

    \begin{theorem}[Local existence]\label{l1}
        Assume that $(\mathrm{H_{p,q}})$, $(\mathrm{H^1_{p,q}})$-$(\mathrm{H^2_{p,q}})$ and $(\mathrm{H_{D}})$ hold, $u_{0} \in W_{0,\mathrm{rad}}^{1, p(\cdot)}\left(\mathcal{A}\right)$, then there exists $T_0>0$ such that problem \eqref{1}-\eqref{3} admits a weak solution $u \in L^{\infty}(0, T_0; W_{0,\mathrm{rad}}^{1, p(\cdot)}\left(\mathcal{A}\right))$ with $u_{t} \in L^{2}\left(0, T_0; L^{2}(\mathcal{A})\right)$.
    \end{theorem}
    \begin{proof}
      We tend to apply Lemma \ref{l19} to prove the local existence of the solution to \eqref{1}-\eqref{3}.
      To start this proof, we check first that IBVP \eqref{1}-\eqref{3} can be reduced to the abstract Cauchy problem (CP).
      To this end, we set $V=W_{0}^{1, p(\cdot)}(\mathcal{A})$ (so that $\Sigma_{V}=W_{0,\mathrm{rad}}^{1, p(\cdot)}(\mathcal{A})$, see \eqref{9}) and $H=L^{2}(\mathcal{A})$ endowed with the norms $\|u\|_{V}:=\|\nabla u\|_{p(\cdot),\mathcal{A}}$ and $|u|_{H}:=\|u\|_{2,\mathcal{A}}$, respectively.
      Obviously, there holds $V \subset H \equiv H^{*} \subset V^{*}$.
      Then, considering the components of the energy functional,
      let $\varphi^{1}$ and $\varphi^{2} \in \Phi_{G}(V)$ be defined as
        \begin{align*}
          \varphi^{1}(u):= \int_{\mathcal{A}} \frac{|\nabla u(x)|^{p(x)}}{p(x)} \mathrm{d} x,
          \quad
          \varphi^{2}(u):= \int_{\mathcal{A}} \frac{|u(x)|^{q(x)}}{q(x)}  \mathrm{d} x.
        \end{align*}
      Since $\varphi^{1}(u),\varphi^{2}(u)$ are Fr\'{e}chet differentiable, the subdifferentials of $\varphi^{1}(u)$ and $\varphi^{2}(u)$ coincide with their Fr\'{e}chet derivatives, and are both single valued given by
        \begin{align*}
          \partial \varphi^{1}(u) = -\Delta_{p(x)} u,
          \quad
          \partial \varphi^{2}(u) = |u|^{q(x)-2} u
        \end{align*}
      in the sense of distribution, where $D(\varphi^{1})=D(\partial \varphi^{1})=V$
      and $D(\varphi^{2})=D(\partial \varphi^{2})=V \cap L^{q(\cdot)}(\mathcal{A})$, respectively. Therefore, IBVP \eqref{1}-\eqref{3} can be reduced to the abstract Cauchy problem (CP). In other words, a $G$-symmetric solution of (CP) with $G=\mathbf{O}(n)$, $V=W_{0}^{1, p(\cdot)}(\mathcal{A})$, $\Sigma_{V}=W_{0,\mathrm{rad}}^{1, p(\cdot)}(\mathcal{A})$ and $H=L^{2}(\mathcal{A})$ is exactly the radially symmetric weak solution of \eqref{1}-\eqref{3}.

      Then, we need to check conditions (i)-(iv) of Lemma \ref{l19} one by one.
      At the end of Subsection \ref{sec2.4}, we have obtained the validity of \eqref{177}, i.e., condition (i) of Lemma \ref{l19} holds. Then, it remains to check (ii)-(iv) of Lemma \ref{l19}.
      We check the validity of the condition (ii) of Lemma \ref{l19} first. Note that for any $u\in W_0^{1,p(\cdot)}(\mathcal{A})$, there holds
      $\|u\|_{p(\cdot),\mathcal{A}}^{p^{-}} \le \|u\|_{p(\cdot),\mathcal{A}}^{p^{+}} + 1$, then it follows from Proposition \ref{p1} that
        \begin{align}\label{121}
        \begin{aligned}
          \varphi^{1}(u)
          \ge \frac{1}{p^+} \int_{\mathcal{A}} |\nabla u(x)|^{p(x)} \mathrm{d} x
          \ge \frac{1}{p^+} \min \left\{\|\nabla u\|_{p(\cdot),\mathcal{A}}^{p^{-}},
                \|\nabla u\|_{p(\cdot),\mathcal{A}}^{p^{+}}\right\}
          \ge \frac{1}{p^+} \left( \|\nabla u\|_{p(\cdot),\mathcal{A}}^{p^{-}} - 1 \right),
        \end{aligned}
        \end{align}
      which implies that \eqref{120} is satisfied with $\gamma =p^-, C_1 = 0, C_2 =1$ and $C_3 = p^+$.

      Then, we check condition (iii) of Lemma \ref{l19}. Since $D(\varphi^{1})=V$ and $\Sigma_{V} = W_{0,\mathrm{rad}}^{1, p(\cdot)}(\mathcal{A})$, then it follows from Corollary \ref{cor1} that
      $D\left(\varphi^{1}\right) \cap \Sigma_{V} \subset D\left(\partial \varphi^{2}\right)$.
      Then, let $\{u_{m}\}$ be a $G$-invariant sequence such that
        $$
        \sup _{t \in[0, T]}\left\{\varphi^{1}\left(u_{m}(t)\right)+\left|u_{m}(t)\right|_{H}\right\}<C  \quad \text{and} \quad
        \int_{0}^{T}\left|u_{mt}\right|_{H}^{2} d t<C,
        $$
      which together with \eqref{121} implies that
        \begin{equation}\label{122}
        \sup _{t \in[0, T]}\|u_m(t)\|_{V}
        =\sup _{t \in[0, T]}\|\nabla u_m(t)\|_{p(\cdot),\mathcal{A}} <C,
        \end{equation}
      namely the sequence $\{u_{m}\}$ is uniformly bounded in $V = W_{0}^{1, p(\cdot)}(\mathcal{A})$ for every fixed $t \in [0,T]$.
      Thus Corollary \ref{cor1} yields that $\{u_{m}(t)\}$ is compact in $L^{q(\cdot)}(\mathcal{A})$ for $t \in[0, T]$. Additionally, since $2<q^+<\infty$, we can apply the $L^p$ interpolation inequality, \eqref{122} and Lemma \ref{l5} to obtain that for every fixed $s,t \in (0,T)$
        \begin{align}\label{191}
        \begin{aligned}
        \|u_{m}(t)-u_{m}(s)\|_{q^+,\mathcal{A}}
        & \leq C \|u_{m}(t)-u_{m}(s)\|_{2,\mathcal{A}}^{\theta_+}
            \|u_{m}(t)-u_{m}(s)\|_{\infty,\mathcal{A}}^{1-\theta_+} \\
        & \leq C \|u_{m}(t)-u_{m}(s)\|_{2,\mathcal{A}}^{\theta_+}
            \|\nabla (u_{m}(t)-u_{m}(s))\|_{p(\cdot),\mathcal{A}}^{1-\theta_+} \\
        & \leq C|t-s|^{\frac{\theta_+}{2}}\left(\int_{0}^{T} |u_{mt}(\tau)|_{H}^{2} \mathrm{d} \tau \right)^{ \frac{\theta_+}{2}}
        \end{aligned}
        \end{align}
      for $\theta_+ = \frac{2}{q^+}$, which verifies the equi-continuity of $\{u_m\}$ in
      $C\left([0, T] ; L^{q(\cdot)}(\mathcal{A})\right)$. Therefore, combining the uniform boundedness and the equi-continuity of $\{u_m\}$ given by \eqref{122} and \eqref{191}, respectively,
      the Arzela-Ascoli theorem implies that there exists a limit function $u$ and a convergent subsequence of $\{u_m\}$, i.e.,
      $u_{m_\nu} \in C\left([0, T] ; L^{q(\cdot)}(\mathcal{A})\right)$ ($m=1,2,\cdots$) such that $\left|u_{m_\nu}\right|^{q(x)-2} u_{m_\nu} \rightarrow|u|^{q(x)-2} u$ strongly in $C\left([0, T] ; L^{q^{\prime}(\cdot)}(\mathcal{A})\right)$, where $q'(x) = \frac{q(x)}{q(x)-1}$ is the H\"{o}lder conjugate of $q(x)$ and $\partial \varphi^{2}(u_{m_\nu}) \rightarrow \partial \varphi^{2}(u)$ strongly in $C\left([0, T] ; V^{*}\right)$.

      Finally, we verify the condition (iv) of Lemma \ref{l19}. To this end, we define
        \begin{align}\label{182}
        \begin{aligned}
           \tilde{\varphi}^{2}(u):=
           \begin{cases}
             \varphi^{2}(u), &u \in L^{q(\cdot)}(\mathcal{A}),\\
             +\infty, &u \in H \backslash L^{q(\cdot)}(\mathcal{A}).
           \end{cases}
        \end{aligned}
        \end{align}
      One can easily check that $\varphi^{2}=\tilde{\varphi}^{2}$ on $V$ and $\partial_{H} \tilde{\varphi}^{2}(u)=$ $|u|^{q(x)-2} u$ with $D\left(\partial_{H} \tilde{\varphi}^{2}\right)=L^{2}(\mathcal{A}) \cap L^{2(q(\cdot)-1)}(\mathcal{A})$, where $\partial_{H} \tilde{\varphi}^{2}(u)$ is defined by \eqref{190}.
      Then, for any $u \in D\left(\varphi^{1}\right)=V$, we can select a sequence of
      $u_{l} \in C_{0}^{\infty}(\mathcal{A})$ such that $u_{l} \rightarrow u$ in $V$ as $l \rightarrow \infty$ (such sequence can be found due to the density of $C_{0}^{\infty}(\mathcal{A})$ in $W_0^{1,p(\cdot)}(\mathcal{A})$). Let $\lambda>0$ and $v_{l}=J_{\lambda}^{2} u_{l}$, i.e., $v_{l}+\lambda\left|v_{l}\right|^{q(x)-2} v_{l}=u_{l}$. Moreover, let
      $\beta_x(s):=|s|^{q(x)-2} s$ for every fixed $x \in \mathcal{A}$, then we can deduce that
        \begin{equation}\label{123}
        v_{l}(x)=(1+\lambda \beta_x)^{-1} u_{l}(x), \quad\forall x \in \mathcal{A}.
        \end{equation}
      Then, one can observe from the definition of $\beta_x(s)$ that $(1+\lambda \beta_x)^{-1}(s) \in C^{1}(\mathbb{R})$ for every fixed $x\in \mathcal{A}$ and is non-expansive, i.e., Lipschitz continuous with the corresponding Lipschitz constant $L = 1$. Therefore, it follows from \eqref{123} that $v_{l} \in C^{1}(\mathbb{R})$ and
        $$
        \left|\frac{v_{l}(x+h)-v_{l}(x)}{h}\right| \leq\left|\frac{u_{l}(x+h)-u_{l}(x)}{h}\right| \quad\forall x \in \mathcal{A},\forall h \in \mathbb{R}^{n} \text { with } x+h \in \mathcal{A},
        $$
     which implies that $\left|\nabla v_{l}(x)\right| \leq\left|\nabla u_{l}(x)\right|$ for all $x \in \mathcal{A}$ and
        \begin{align}\label{124}
          \varphi^{1}\left(v_{l}\right) \leq \varphi^{1}\left(u_{l}\right).
        \end{align}
     Since $v_{l} \rightarrow v=J_{\lambda}^{2} u$ in $H$ as $l \rightarrow \infty$, there holds
        \begin{align}\label{125}
          \varphi^{1}\left(J_{\lambda}^{2} u\right) \leq \liminf \limits_{l \rightarrow \infty} \varphi^{1}\left(v_{l}\right).
        \end{align}
    Additionally, it follows from \eqref{124}, \eqref{125} and the fact $\lim \limits_{l\rightarrow \infty} \varphi^{1}\left(u_{l}\right)=\varphi^{1}(u)$ that
        $$\varphi^{1}\left(J_{\lambda}^{2} u\right) \leq \varphi^{1}(u).$$
    Moreover if $u$ is $G$-invariant, then it easily follows from \eqref{123} with $u_{l}, v_{l}$ replaced by $u, v$ that $v$ is also $G$-invariant, namely $J_{\lambda}^{2} u \in \Sigma_{V}$, provided that
    $u \in \Sigma_{V}$. Therefore, (iv) of Lemma \ref{l19} is satisfied with $l_{1} \equiv 1$ and $l_{2} \equiv 0$.
    \end{proof}

A precise continuation criterion for the local weak solution, together with the corresponding blow-up alternative, will be established in Subsection \ref{sec4.1} after deriving the energy identity.

Next, we give the uniqueness of the solution by the Sobelev embedding $W_{0,\mathrm{rad}}^{1, p(\cdot)}(\mathcal{A}) \hookrightarrow L^{\infty}(\mathcal{A})$ given by Lemma \ref{l5} (i).
Although the proof of the uniqueness result is not difficult, we need to stress again that the proof of the following Theorem \ref{t1} benefits from Lemma \ref{l5} (i) with the radially symmetric setting of the present paper, for which we don't have to add the restriction that the spatial dimension is low in sense that $p(x)\ge n$ (in which case the corresponding critical Sobolev exponent $p^*(x) = \infty$, cf. \eqref{165}) or the initial data is of higher integrability that $u_0 \in L^{\infty}(\Omega)$. For the uniqueness results obtained with these mentioned assumptions, we refer the readers to Theorem 1 \cite{r22} for the constant exponent case with $p > n$, and to Theorem 2.3 \cite{r39} for the variable exponent case with
$u_0 \in W_{0}^{1, p(\cdot)}(\Omega) \cap L^{\infty}(\Omega)$.

\begin{theorem}[Uniqueness]\label{t1}
      Let $(\mathrm{H_{p,q}})$, $(\mathrm{H^1_{p,q}})$-$(\mathrm{H^2_{p,q}})$ and $(\mathrm{H_{D}})$ be satisfied and $u_{0} \in W_{0,\mathrm{rad}}^{1, p(\cdot)}\left(\mathcal{A}\right)$, then problem \eqref{1}-\eqref{3} permits at most one weak solution.
    \end{theorem}
    \begin{proof}
    Arguing by contradiction,
    we suppose that there are two solutions $u_1,u_2$ of problem \eqref{1}-\eqref{3} with $u_1(\cdot,0)=u_2(\cdot,0) =u_0$. Let $w := u_2 - u_1$, then we shall check that $u_1 \equiv u_2$, in other words, $w \equiv 0$.
    Substituting $u_2$ and $u_1$ into Eq. \eqref{1}, respectively, and subtracting these two equations.
    Then, by testing the both sides of the subtraction by $w$ and integrating by parts, we get
        \begin{align}\label{119}
        \begin{aligned}
          &\frac{1}{2} \frac{\mathrm{d}}{\mathrm{d} t} \| w \|_{2,\mathcal{A}}^{2}
            + (|\nabla u_2|^{p(x)-2} \nabla u_2 - |\nabla u_1|^{p(x)-2} \nabla u_1,\nabla w)_{\mathcal{A}}\\
          =& \left(|u_2|^{q(x)-2} u_2 - |u_1|^{q(x)-2} u_1, w \right)_{\mathcal{A}}.
        \end{aligned}
        \end{align}
    It follows from the elementary equality (cf. Lemma 2.1, \cite{p11})
        \begin{align*}
           (|\mathbf{a}|^{p-2} \mathbf{a} - |\mathbf{b}|^{p-2} \mathbf{b}) \cdot(\mathbf{a}-\mathbf{b})
           \geq 2^{-p} |\mathbf{a}-\mathbf{b}|^p, \quad \forall \mathbf{a}, \mathbf{b} \in \mathbb{R}^N,\ \forall p \geq 2
        \end{align*}
    and $p^- \ge 2$ that $(|\nabla u_2|^{p(x)-2} \nabla u_2 - |\nabla u_1|^{p(x)-2} \nabla u_1,\nabla w)_{\mathcal{A}} \ge 0$.
    Then we deal with the term on the right hand side of \eqref{119}. By the mean value theorem, Holder's inequality and the embeddings
        $$W_{0,\mathrm{rad}}^{1, p(\cdot)}(\mathcal{A}) \hookrightarrow L^{\infty}(\mathcal{A})
        \hookrightarrow L^{q(\cdot)}(\mathcal{A})$$
    ensured by Lemma \ref{l5} (i), we obtain
        \begin{align}\label{110}
        \begin{aligned}
          &\left(|u_2|^{q(x)-2} u_2 - |u_1|^{q(x)-2} u_1, w \right)_{\mathcal{A}}\\
          \le&  (q^+ - 1) \int_{\mathcal{A}} \left(\int_0^1|s u_2 + (1-s) u_1|^{q(x)-2} \mathrm{~d} s \right) |w|^2 \mathrm{d} x\\
          \le&  (q^+ - 1) \int_{\mathcal{A}} \left(|u_2|^{q(x)-2} + |u_1|^{q(x)-2}  \right) |w|^2 \mathrm{d} x\\
          \le& C \|w\|_{2,\mathcal{A}}^2,
        \end{aligned}
        \end{align}
    with some positive constant $C$. Then gathering \eqref{119}-\eqref{110}, we obtain
        \begin{align*}
          \frac{\mathrm{d}}{\mathrm{d} t} \| w \|_{2,\mathcal{A}}^{2}  \le& C \|w\|_{2,\mathcal{A}}^2
        \end{align*}
    and the uniqueness of the solution follows from a straight application of Gronwall's inequality.
    \end{proof}

In view of Theorems \ref{l1} and \ref{t1}, the local weak solution can be uniquely continued to a maximal interval $[0,T_{\max})$.

\section{Global well-posedness}\label{sec33}
In this section, based on the local well-posedness of the solution obtained in Section \ref{sec3} and the recovered compact embedding, i.e., Corollary \ref{cor1} for radially symmetric functions of $W_{0,\mathrm{rad}}^{1, p(\cdot)}(\mathcal{A})$, we tend to apply the potential well theory established on the annulus $\mathcal{A}$ (introduced in Subsection \ref{sec2.3}) to investigate the effect of the initial data on the global well-posedness of the solution, including the global-in-time existence and finite-time blow-up of the solution.

\subsection{Energy identities and invariant manifolds}\label{sec4.1}
In this subsection, we introduce the energy identities of \eqref{1}-\eqref{3} firstly, and then derive several invariant manifolds of the solution generated by initial data of different energy levels, utilized to prove our main results concerning the global well-posedness of the solution in Subsections \ref{sec4}-\ref{sec5}.

In order to prove the invariant manifolds of the solution, in the following Lemma \ref{l7}
we introduce two energy identities of \eqref{1}-\eqref{3} associated with $J(u)$ and $I(u)$, respectively by applying the subdifferential technique in the Hilbert space setting \cite{r23} and the Lebesgue differentiation theorem.
We emphasize here that the energy identity \eqref{47} given below can not be obtained by taking the test function $\psi = u_t$ in \eqref{23} directly, since higher regularity
        $u_t \in L^{\infty}( 0,T_{\max}; W_{0,\mathrm{rad}}^{1, p(\cdot)}(\mathcal{A}) \cap L^{q(\cdot)}(\mathcal{A}) )$
is required to ensure that the integration make sense, which is unknown and we only have $u_{t} \in L^{2}\left(0, T_{\max}; L^{2}(\mathcal{A})\right)$.

    \begin{lemma}[Energy identities]\label{l7}
    Assume that $(\mathrm{H_{p,q}})$, $(\mathrm{H^1_{p,q}})$-$(\mathrm{H^2_{p,q}})$ and $(\mathrm{H_{D}})$ hold. Let $u$ be a weak solution to problem \eqref{1}-\eqref{3} defined on the interval $(0,T_{\max})$, then both $J(u(t))$ and $I(u(t))$ are absolutely continuous in $t$, and the following two energy identities hold for a.e. $t \in (0, T_{\max})$,
         \begin{align} \label{47}
          \frac{\mathrm{d}}{\mathrm{d} t} J(u(t))
          = -\left\|u_t(t)\right\|_{2,\mathcal{A}}^{2}
        \end{align}
    and
        \begin{align}\label{48}
            \frac{\mathrm{d}}{\mathrm{d} t} \left\|u(t)\right\|_{2,\mathcal{A}}^{2}
          = -2 I(u(t)).
        \end{align}
    \end{lemma}
    \begin{proof}
      Firstly, to prove \eqref{47}, we shall apply the subdifferential theory in the Hilbert space setting \cite{r23}. Specifically, we set the Hilbert space $H = L^2(\Omega)$ as in the proof of Theorem \ref{l2} and define
        \begin{align*}
        \begin{aligned}
           \tilde{\varphi}^{1}(u):=
           \begin{cases}
             \varphi^{1}(u), &u \in W_{0,\mathrm{rad}}^{1, p(\cdot)}(\mathcal{A}),\\
             +\infty, &u \in H \backslash W_{0,\mathrm{rad}}^{1, p(\cdot)}(\mathcal{A}).
           \end{cases}
        \end{aligned}
        \end{align*}
      Recalling $\tilde{\varphi}^{2}(u)$ defined by \eqref{182}, it follows from Theorem 3.3, Lemmas 2.6, 2.7 in \cite{r23} that $u$ solves the following abstract Cauchy problem
        \begin{equation*}
        \begin{cases}
          \frac{\mathrm{d} u}{\mathrm{d} t}(t) + \partial \tilde{\varphi}^{1}(u(t))-\partial \tilde{\varphi}^{2}(u(t)) \ni 0 \ \text { in } H,\\
          u(0)=u_{0}
        \end{cases}
        \end{equation*}
    and
        \begin{align*}
          \int_0^t |u_t(\tau)|^2 \mathrm{d}\tau + \tilde{\varphi}^{1}(u(t)) - \tilde{\varphi}^{2}(u(t))
          = \tilde{\varphi}^{1}(u_0) - \tilde{\varphi}^{2}(u_0) \quad \text {in } H
        \end{align*}
    for $t \in (0,T_{\max})$, i.e.,
            \begin{align}\label{99}
              J(u(t+h)) - J(u(t)) = -\int_{t}^{t+h} \left\|u_{t}(\tau)\right\|_{2,\mathcal{A}}^{2} \mathrm{d} \tau,
            \end{align}
      where $t\in(0,T_{\max})$, $h>0$, and the right hand side of \eqref{99} is absolutely continuous with respect to $t$ according to the absolute continuity of Lebesgue integration. Then by multiplying the both sides of \eqref{99} by $\frac{1}{h}$, the identity \eqref{47} immediately follows from the Lebesgue differentiation theorem by letting $h \rightarrow 0^+$.

      Then we check \eqref{48}. Since $u \in L^{\infty}(0, T; W_{0}^{1, p(\cdot)}(\mathcal{A}))$, by taking
      $\psi = u$ in \eqref{23} and integrating over $\mathcal{A} \times (t, t + h)$ with $h > 0$, we derive that
            \begin{align*}
              &\frac{1}{h} \int_t^{t+h} \int_{\mathcal{A}} u u_t \mathrm{~d} x \mathrm{~d} \tau+\frac{1}{h} \int_t^{t+h} \int_{\mathcal{A}}|\nabla u|^{p(x)} \mathrm{d} x \mathrm{~d} \tau
              = \frac{1}{h} \int_t^{t+h} \int_{\mathcal{A}}|u|^{q(x)} \mathrm{d} x \mathrm{~d} \tau.
            \end{align*}
        Then let $h \rightarrow 0^+$, the identity \eqref{48} straight follows by an application of Lebesgue differentiation theorem and the definition of $I(u)$.
    \end{proof}

    \begin{remark}[On Lemma \ref{l7}]
      It seems that via the subdifferential technique in the reflexive Banach space setting, we can only obtain the energy inequality, see \cite{r53}, which is not adequate for us to obtain the absolute continuity of $J(u(t))$. Moreover, since $u \in C_{w}([0, T_{\max}) ; W_{0,\mathrm{rad}}^{1, p(\cdot)}(\mathcal{A}))$ according to \eqref{174} in Lemma \ref{l19}
      and the energy identity \eqref{99} holds, by the similar argument as the proof of lemma 2.2 \cite{r50}, we may infer better continuity of $u$ in $t$, i.e., $u \in C([0, T_{\max}) ; W_{0,\mathrm{rad}}^{1, p(\cdot)}(\mathcal{A}))$.
    \end{remark}

    \begin{lemma}[Continuation criterion and blow-up alternative]\label{prop:continuation}
    Assume that $(\mathrm{H_{p,q}})$, $(\mathrm{H^1_{p,q}})$-$(\mathrm{H^2_{p,q}})$ and $(\mathrm{H_{D}})$ hold. Let
    \[
    u:[0,T_{\max})\to W_{0,\mathrm{rad}}^{1, p(\cdot)}\left(\mathcal{A}\right)
    \]
    be the maximal weak solution to problem \eqref{1}-\eqref{3} corresponding to $u_0\in W_{0,\mathrm{rad}}^{1, p(\cdot)}\left(\mathcal{A}\right)$. If
    \[
    \sup_{0\le t<T_{\max}}\|\nabla u(t)\|_{p(\cdot),\mathcal{A}}<\infty,
    \]
    then $T_{\max}=\infty$.
    Equivalently, if $T_{\max}<\infty$, then
    \[
    \limsup_{t\to T_{\max}^{-}}\|\nabla u(t)\|_{p(\cdot),\mathcal{A}}=\infty.
    \]
    \end{lemma}
    \begin{proof}
    We argue by contradiction. Suppose that $T_{\max}<\infty$ and
    \begin{equation}
    \sup_{0\le t<T_{\max}}\|\nabla u(t)\|_{p(\cdot),\mathcal{A}}\le M
    \label{eq:continuation-bound}
    \end{equation}
    for some constant $M>0$.

    By Lemma \ref{l5} (i), the radial Sobolev embedding
    $W_{0,\mathrm{rad}}^{1, p(\cdot)}(\mathcal{A}) \hookrightarrow L^{\infty}(\mathcal{A})$
    is continuous. Hence,
    \[
    \sup_{0\le t<T_{\max}}\|u(t)\|_{\infty,\mathcal{A}}\le C M.
    \]
    Consequently,
    \[
    \int_{\mathcal{A}}\frac{|u(x,t)|^{q(x)}}{q(x)}\,\mathrm{d} x\le C_M,
    \qquad 0\le t<T_{\max},
    \]
    where $C_M>0$ depends only on $M$, the exponents and the domain. Therefore,
    \[
    J(u(t))\ge -C_M,
    \qquad 0\le t<T_{\max}.
    \]

    By the energy identity in Lemma \ref{l7},
    \[
    \int_0^t\|u_\tau(\tau)\|_{2,\mathcal{A}}^2\,\mathrm{d}\tau
    +J(u(t))
    =J(u_0),
    \qquad 0<t<T_{\max}.
    \]
    It follows that
    \begin{equation}
    \int_0^{T_{\max}}
    \|u_t(t)\|_{2,\mathcal{A}}^2\,\mathrm{d}t
    \le J(u_0)+C_M<\infty.
    \label{eq:ut-global-bound}
    \end{equation}

    For $0<s<t<T_{\max}$, the Cauchy-Schwarz inequality gives
    \[
    \begin{aligned}
    \|u(t)-u(s)\|_{2,\mathcal{A}}
    &\le
    \int_s^t\|u_\tau(\tau)\|_{2,\mathcal{A}}\,\mathrm{d}\tau\\
    &\le
    (t-s)^{1/2}
    \left(
    \int_s^t\|u_\tau(\tau)\|_{2,\mathcal{A}}^2\,\mathrm{d}\tau
    \right)^{1/2}.
    \end{aligned}
    \]
    Hence $\{u(t)\}_{t<T_{\max}}$ is a Cauchy family in $L^2(\mathcal{A})$
    as $t\to T_{\max}^{-}$. Thus there exists $u_*\in L^2(\mathcal{A})$
    such that
    \begin{equation}
    u(t)\to u_*
    \quad\text{strongly in }L^2(\mathcal{A})
    \quad\text{as }t\to T_{\max}^{-}.
    \label{eq:L2-terminal-limit}
    \end{equation}

    On the other hand, by \eqref{eq:continuation-bound} and the
    reflexivity of $W_{0,\mathrm{rad}}^{1, p(\cdot)}(\mathcal{A})$, for any
    sequence $t_j\uparrow T_{\max}$ there exist a subsequence,
    still denoted by $\{t_j\}$, and some
    $\widehat{u}\in W_{0,\mathrm{rad}}^{1, p(\cdot)}(\mathcal{A})$ such that
    \[
    u(t_j)\rightharpoonup\widehat{u}
    \quad\text{weakly in }
    W_{0,\mathrm{rad}}^{1, p(\cdot)}(\mathcal{A}).
    \]
    Since
    \[
    W_{0,\mathrm{rad}}^{1, p(\cdot)}(\mathcal{A})
    \hookrightarrow L^2(\mathcal{A})
    \]
    continuously, the above convergence also holds weakly in
    $L^2(\mathcal{A})$. In view of \eqref{eq:L2-terminal-limit}, we
    necessarily have $\widehat{u}=u_*$. Therefore,
    \[
    u_*\in W_{0,\mathrm{rad}}^{1, p(\cdot)}(\mathcal{A}).
    \]

    We can now apply Theorem \ref{l1} with $u_*$ as the initial datum.
    There exist $\delta>0$ and a local weak solution $v$ on
    $[0,\delta)$ satisfying $v(0)=u_*$. Define
    \[
    \widetilde{u}(t):=
    \begin{cases}
    u(t), & 0\le t<T_{\max},\\
    v(t-T_{\max}),&
    T_{\max}\le t<T_{\max}+\delta.
    \end{cases}
    \]
    Since $u(t)\to u_*=v(0)$ strongly in $L^2(\mathcal{A})$ as
    $t\to T_{\max}^{-}$, together with
    \[
    u_t\in L^2(0,T_{\max};L^2(\mathcal{A})),
    \qquad
    v_t\in L^2(0,\delta;L^2(\mathcal{A})),
    \]
    the function $\widetilde{u}$ belongs to
    \[
    L^\infty
    \bigl(0,T_{\max}+\delta;
    W_{0,\mathrm{rad}}^{1, p(\cdot)}(\mathcal{A})\bigr)
    \cap
    W^{1,2}
    \bigl(0,T_{\max}+\delta;L^2(\mathcal{A})\bigr).
    \]
    Moreover, the weak formulation is satisfied on both sides of
    $t=T_{\max}$, and the $L^2$-continuity at the joining time
    shows that no jump term appears. Hence $\widetilde{u}$ is a
    weak solution of \eqref{1}-\eqref{3} on
    $[0,T_{\max}+\delta)$.

    This contradicts the maximality of $T_{\max}$. Therefore
    $T_{\max}=\infty$.

    The equivalent blow-up alternative follows immediately:
    if $T_{\max}<\infty$, then
    \[
    \limsup_{t\to T_{\max}^{-}}
    \|\nabla u(t)\|_{p(\cdot),\mathcal{A}}
    =\infty .
    \]
    \end{proof}

    Next, by virtue of Lemma \ref{l7}, we give several invariant manifolds for the solution of problem \eqref{1}-\eqref{3} to deal with initial data at subcritical, critical and supercritical energy levels,
    i.e., $0<J(u_0)<d, J(u_0)=d$ and $J(u_0)>d$ respectively.

    One can observe that we don't consider the case of non-positive initial energy $0<J(u_0)\le 0$ with $u_{0} \in W_{0,\mathrm{rad}}^{1, p(\cdot)}\left(\mathcal{A}\right) \backslash \{0\}$, since it follows from the definitions of $I(u)$ and $J(u)$ that
        \begin{align}\label{169}
          J(u(t)) \ge \frac{1}{q^-}I(u(t))
                        +\left(\frac{1}{p^{+}}-\frac{1}{q^{-}}\right)
                            \int_{\mathcal{A}}\left|\nabla u(t)\right|^{p(x)} \mathrm{d} x
        \end{align}
    for $t \in (0,T_{\max})$, which together with \eqref{47} and $u_{0} \in W_{0,\mathrm{rad}}^{1, p(\cdot)}\left(\mathcal{A}\right) \backslash \{0\}$ implies that $I(u(t))<0$ for $t\in (0,T_{\max})$ and the solution blows up in finite time, see Theorem \ref{t10}.

    Then, we show the stable and unstable invariant manifolds of the solution for $0<J(u_0)\le d$ in the following lemma first.

    \begin{lemma}[Stable and unstable invariant manifolds for $0<J(u_0)\le d$]\label{l17}
    Assume that $(\mathrm{H_{p,q}})$, $(\mathrm{H^1_{p,q}})$-$(\mathrm{H^2_{p,q}})$ and $(\mathrm{H_{D}})$ hold. Suppose that $0< J(u_0) \le d$, the following statements hold
            \begin{enumerate}[(i)]
                \item if $u_0 \in \mathcal{N}^+$, then
                    $
                      u \in \mathcal{N}^+ \text{ for } t \in [0,T_{\max});
                    $
                \item if $u_0 \in \mathcal{N}^-$, then
                    $
                      u \in \mathcal{N}^- \text{ for } t \in [0,T_{\max})
                    $.
            \end{enumerate}
    \end{lemma}
    \begin{proof}
     We prove (i) first by considering the cases $0<J(u_0) < d$ and $J(u_0) =d$ separately.
     For $0<J(u_0) < d$ and $u_0 \in \mathcal{N}^+$, arguing by contradiction, we suppose that $u$ does not lie in $\mathcal{N}^+$ for all $t \in [0,T_{\max})$, which means that there exists a $t_1 > 0$ such that $u \in \mathcal{N}^+$ for  $t \in [0, t_1)$ and $u(t_1) \in \mathcal{N}$ for the first time according to the continuity of $I(u(t))$ in $t$.
     Then by Lemma \ref{l7}, we have
            \begin{align*}
               J(u(t_1)) \le \int_{0}^{t_1} \left\|u_{t}(\tau)\right\|_{2,\mathcal{A}}^{2} {\rm d}\tau + J(u(t_1)) \le J(u_0) < d,
            \end{align*}
     However, since $u(t_1) \in \mathcal{N}$, we also have $J(u(t_1)) \ge d$, which causes a contradiction.

     For $J(u_0) = d$ and $u_0 \in \mathcal{N}^+$, we still argue by contradiction and suppose that $u$ does not always lie in $\mathcal{N}^+$ for $t \in [0,T_{\max})$. Then, there exists a $t_2 > 0$ such that $u \in \mathcal{N}^+$ for  $t \in [0, t_2)$ and $u(t_2) \in \mathcal{N}$ for the first time. Then by Lemma \ref{l7}, we have
            \begin{align*}
               \int_{0}^{t_2} \left\|u_{t}(\tau)\right\|_{2,\mathcal{A}}^{2} {\rm d}\tau + J(u(t_2)) \le J(u_0) = d,
            \end{align*}
         which together with the definition of $d$ claims that $\int_{0}^{t_2} \left\|u_{t}\right\|_{2}^{2} {\rm d} \tau = 0$ and $J(u(t_2)) = d$. Therefore, we can derive the facts that $u \equiv u_0$ for a.e. $t \in [0,t_2]$ and $I(u(t_2)) = I(u_0) > 0$ also lead to a contradiction.

     Then, the proof of the unstable invariant manifold for $0<J(u_0)\le d$ and $u_0 \in \mathcal{N}^-$ can be similarly done as the above discussion for $u_0 \in \mathcal{N}^+$.
    \end{proof}

    \begin{remark}[On Lemma \ref{l17}]
      Although $J(u(t))$ is decreasing according to \eqref{47}, one can deduce from \eqref{169} and Lemma \ref{l17} (i) that $J(u(t))$ remains positive for $u_0\in \mathcal{N}^+$. Otherwise, if there exists a $t_0>0$ such that $J(u(t_0))=0$, we may deduce from \eqref{169} that
        \begin{align*}
          0 = J(u(t_0)) \ge \frac{1}{q^-}I(u(t_0))
                        +\left(\frac{1}{p^{+}}-\frac{1}{q^{-}}\right)
                            \int_{\mathcal{A}}\left|\nabla u(t_0)\right|^{p(x)} \mathrm{d} x >0,
        \end{align*}
      which leads to a contradiction. Moreover, in Theorem \ref{t3}, we show that $u$ converges to the trivial equilibrium of \eqref{1}-\eqref{3}, i.e., $u_0 \in \mathcal{G}_0$, which implies that
          $$\lim \limits_{t \rightarrow \infty} I(u(t))=\lim \limits_{t \rightarrow \infty} J(u(t))=0.$$
      However, the positivity of $J(u(t))$ can not be preserved by \eqref{169} for $u_0\in \mathcal{N}^-$ according to Lemma \ref{l17}, (ii). Furthermore, since
        $$\int_0^t\left\|u_t(\tau)\right\|_{2,\mathcal{A}}^2 \mathrm{d} \tau \geq \frac{1}{t}\left(\int_0^t\left\|u_t(\tau)\right\|_{2,\mathcal{A}} \mathrm{d} \tau\right)^2 \geq \frac{1}{t}\left(\|u(t)\|_{2,\mathcal{A}}-\left\|u_0\right\|_{2,\mathcal{A}}\right)^2,$$
      we can deduce from \eqref{99} and \eqref{169} that
          \begin{align*}
            \frac{1}{q^-} I(u(t)) \leq J(u(t)) \leq J(u_0)-\frac{1}{t}\left(\|u(t)\|_{2,\mathcal{A}}-\left\|u_0\right\|_{2,\mathcal{A}}\right)^2 .
          \end{align*}
      Provided that the $L^2$ norm of the solution blows up in finite time, i.e.,
      $\|u(t)\|_{2,\mathcal{A}} \rightarrow \infty \text { as } t \rightarrow T_{\max }$, we can deduce that
          $$\lim \limits_{t \rightarrow T_{\max }} I(u(t))=\lim \limits_{t \rightarrow T_{\max }} J(u(t))=-\infty.$$
    \end{remark}

    For $J(u_0)>d$, recalling that Lemma \ref{l9} claims that $\Lambda_{k}<\infty$ and
    $\lambda_k \ge \lambda_{\infty}>0$ for $k>d$ in the multiple-regime setting of the variable exponents, based on which, we give the following stable and unstable invariant manifolds following the same spirit of \cite{r7} without proof, whose proof itself is not affected by the multiple-regime setting and requires more of the energy identities given by Lemma \ref{l7} to conduct a contradiction argument against the definitions of $\lambda_k$ and $\Lambda_{k}$.

    \begin{lemma}[Stable and unstable invariant manifolds for $J(u_0)> d$]\label{l15}
    Let $(\mathrm{H_{p,q}})$, $(\mathrm{H^1_{p,q}})$-$(\mathrm{H^2_{p,q}})$ and $(\mathrm{H_{D}})$ be satisfied. Suppose that $J(u_0) > d$, then we have

            \begin{enumerate}[(i)]
                \item if $\|u_0\|_{2,\mathcal{A}} \le \lambda_{J(u_0)}$ and $u_0 \in \mathcal{N}^+$, then
                        $
                          u \in \mathcal{N}^+ \text{ for } t \in [0,T_{\max});
                        $
                \item if $\|u_0\|_{2,\mathcal{A}} \ge \Lambda_{J(u_0)}$ and $u_0 \in \mathcal{N}^-$, then
                        $
                          u \in \mathcal{N}^- \text{ for } t \in [0,T_{\max}).
                        $
            \end{enumerate}
    \end{lemma}

Then, we characterize the unstable invariant manifold of the solution with arbitrarily positive initial energy $J(u_0) > 0$.

        \begin{lemma}[Unstable invariant manifolds for $J(u_0) > 0$]\label{l10}
            Let $(\mathrm{H_{p,q}})$, $(\mathrm{H^1_{p,q}})$-$(\mathrm{H^2_{p,q}})$ and $(\mathrm{H_{D}})$ be satisfied.
            Assume that $u_{0} \in W_{0,\mathrm{rad}}^{1, p(\cdot)}\left(\mathcal{A}\right)$ satisfies
                \begin{align}\label{93}
                    P(\|u_0\|_{2,\mathcal{A}}) > J(u_{0})>0,
                \end{align}
            then $u \in \mathcal{N}^{-}$ for $t\in [0,T_{\max})$,
            where $P(s): \mathbb{R}_+ \rightarrow \mathbb{R}_+$ is defined by
                \begin{align}\label{100}
                \begin{aligned}
                   P(s) :=&
                    \left(\frac{1}{p^{+}}-\frac{1}{q^{-}}\right)
                    \min \left\{
                        \left(C^{*}_{p(\cdot),2}\right)^{-p^{-}} s^{p^{-} },
                        \left(C^{*}_{p(\cdot),2}\right)^{-p^{+}} s^{p^{+} }
                        \right\}
                \end{aligned}
                \end{align}
            and $C^{*}_{p(\cdot),2}$ is the Sobolev constant of the embedding
            $W_{0}^{1, p(\cdot)}(\mathcal{A}) \hookrightarrow L^{2}(\mathcal{A})$.
            \end{lemma}
            \begin{proof}
              Let $u(t)$ be a weak solution to problem \eqref{1}-\eqref{3} generated by $u_0$.
              Firstly, we tend to prove that $u_0\in \mathcal{N}^-$.
              By the definition of $J(u)$, we get
                  \begin{align}
                     J(u_0)
                     \ge& \frac{1}{q^-}I(u_0)
                        +\left(\frac{1}{p^{+}}-\frac{1}{q^{-}}\right)
                            \int_{\mathcal{A}}\left|\nabla u_{0}\right|^{p(x)} \mathrm{d} x \notag\\
                     \ge& \frac{1}{q^-}I(u_0)
                        + \left( \frac{1}{p^+} - \frac{1}{q^-} \right)
                      \min \left\{
                        \|\nabla u_0\|_{p (\cdot),\mathcal{A}}^{p^{-}},
                        \|\nabla u_0\|_{p(\cdot),\mathcal{A}}^{p^{+}}
                            \right\} \label{94}\\
                     \ge& \frac{1}{q^-}I(u_0)
                        +
                    \left(\frac{1}{p^{+}}-\frac{1}{q^{-}}\right)
                    \min \left\{
                        \left(C^{*}_{p(\cdot),2}\right)^{-p^{-}} \|u_0\|_{2,\mathcal{A}}^{p^{-} },
                        \left(C^{*}_{p(\cdot),2}\right)^{-p^{+}}
                        \|u_0\|_{2,\mathcal{A}}^{p^{+} }
                        \right\} \notag\\
                     =& \frac{1}{q^-}I(u_0)
                        + P(\|u_0\|_{2,\mathcal{A}}), \notag
                  \end{align}
            thus $u_0 \in \mathcal{N}^-$ according to \eqref{93}.

            Next, we claim $u(t) \in \mathcal{N}^{-}$ for all $t \in[0, T_{\max})$.
            Since $u_0 \in \mathcal{N}^-$ and $I(u(t))$ is continuous in $t$,
            there exists an $s_0>0$ such that $u(t) \in \mathcal{N}^{-}$ for $t\in(0,s_0)$.
            Then, arguing by contradiction, we assume that there exists an $s_1 \in(s_0, T_{\max})$ such that $u(t) \in \mathcal{N}^{-}$ for $t \in (0,s_1)$, and $u(s_1) \in \mathcal{N}$ for the first time, which together with \eqref{94} implies that
                \begin{align}\label{95}
                  \begin{aligned}
                     J(u(s_1))
                     \ge \frac{1}{q^-}I(u(s_1)) + P(\|u(s_1)\|_{2,\mathcal{A}})
                     = P(\|u(s_1)\|_{2,\mathcal{A}}).
                  \end{aligned}
                  \end{align}
            Recalling \eqref{48}, we get
                $$
                \frac{\rm d}{{\rm d} t}\|u(t)\|_{2,\mathcal{A}}^{2}=-2 I(u)>0, \quad \text { for a.e. } t \in[0, s_1),
                $$
            which indicates that $\|u(t)\|_{2}^{2}$ is increasing as long as $u(t) \in \mathcal{N}^{-}$ and
                \begin{align*}
                  \|u(s_1)\|_{2,\mathcal{A}}^{2} \ge \|u_{0}\|_{2,\mathcal{A}}^{2}.
                \end{align*}
            Hence, by the monotonicity of $P(s)$ with respect to $s$ (see \eqref{100}), we have
                \begin{align}\label{96}
                  P(\|u(s_1)\|_{2,\mathcal{A}}) \ge P(\|u_0\|_{2,\mathcal{A}}).
                \end{align}
            Moreover, \eqref{47} implies
                \begin{align}\label{97}
                  J(u(s_1)) \le J(u_{0}).
                \end{align}
            Then, gathering \eqref{95}-\eqref{97}, we get
                $
                  J(u_{0}) \ge P(\|u_0\|_{2,\mathcal{A}}),
                $
            which contradicts \eqref{93}. Thus $u(t) \in \mathcal{N}^{-}$ for all $t \in[0, T_{\max})$.
            \end{proof}

\subsection{Subcritical and critical initial energy $J(u_0) \le d$}\label{sec4}
In this subsection, we investigate the global well-posedness of the solution with initial data at subcritical and critical initial energy levels, including the finite-time blow-up of the solution for non-positive initial energy $J(u_0) \le 0$ in Theorem \ref{t10} and the global existence and the finite-time blow-up of the solution for positive initial energy $0 < J(u_0) \le d$, see Theorem \ref{t3} and Theorem \ref{t4}, respectively. Particularly, for the subcritical initial energy $0< J(u_0) < d$, we prove the asymptotic behavior of the global solution in terms of the $L^2$ norm in Theorem \ref{t3}, and we also estimate the upper bound of the blow-up time for the blow-up solution in Theorem \ref{t4}.

Firstly, we give the finite-time blow-up of the solution with initial data of non-positive initial energy $J(u_0) \le 0$ in the following Theorem \ref{t10}.

        \begin{theorem}[Finite-time blow-up]\label{t10}
          Let $(\mathrm{H_{p,q}})$, $(\mathrm{H^1_{p,q}})$-$(\mathrm{H^2_{p,q}})$ and $(\mathrm{H_{D}})$ be satisfied. Provided $J(u_0) \le 0$, then $u_{0} \in \mathcal{B}$ and the maximal existence can be estimated by
            \begin{align}\label{185}
              T_{\max } \leq \frac{2}{\left(q^{-}-2\right) C}\left\|u_0\right\|_{2,\mathcal{A}}^{2-q^{-}},
            \end{align}
          where $C$ is some positive constant depending on $p^+,q^{\pm},R$ and $r$.
        \end{theorem}
        \begin{proof}
        Firstly, it follows from Lemma \ref{l7} and the definitions of $J(u)$ and $I(u)$ that
            \begin{align*}
                0\ge J(u)\ge \frac{1}{p^+} I(u)
                 +   \left( \frac{1}{p^+} - \frac{1}{q^-} \right)
                 \int_{\mathcal{A}} |u(x)|^{q(x)}  \mathrm{d} x,
            \end{align*}
        which together with H\"{o}lder's inequality implies that
            \begin{align}\label{68}
            \begin{aligned}
                -I(u)
                \ge& \left( 1 - \frac{p^+}{q^-} \right) \int_{\mathcal{A}} |u(x)|^{q(x)}  \mathrm{d} x\\
                \ge&  \left( 1 - \frac{p^+}{q^-} \right) \left(\|u\|^{q^-}_{q^-,\Omega_{+}}
                    + \|u\|^{q^+}_{q^+,\Omega_{-}} \right)\\
                \ge&  \left( 1 - \frac{p^+}{q^-} \right)
                    \left(|\Omega_{+}|^{\frac{2-q^{-}}{2}} \|u\|^{q^-}_{2,\Omega_{+}}
                    + |\Omega_{-}|^{\frac{2-q^{+}}{2}} \|u\|^{q^+}_{2,\Omega_{-}} \right),
            \end{aligned}
            \end{align}
        where $\Omega_{+}:=\{x \in \mathcal{A}:|u(x)| \geq 1\}$ and $\Omega_{-}:=\{x \in \mathcal{A}:|u(x)|<1\}$.
        Then, let $\|u(t)\|_2^2:= \sigma(t)$ and combing \eqref{48} and \eqref{68}, we obtain
                \begin{align}\label{184}
                \begin{aligned}
                &(\sigma^{\prime}(t))^{\frac{2}{q^{+}}} + (\sigma^{\prime}(t))^{\frac{2}{q^{-}}}\\
                \geq&
                \underbrace{
                \min\left\{
                2^{\frac{2}{q^{+}}} \left( 1 - \frac{p^+}{q^-} \right)^{\frac{2}{q^{+}}}
                    \left|\mathcal{A} \right|^{\frac{2}{q^{+}}-1},
                2^{\frac{2}{q^{-}}} \left( 1 - \frac{p^+}{q^-} \right)^{\frac{2}{q^{-}}}
                    \left|\mathcal{A} \right|^{\frac{2}{q^{-}}-1} \right\}}_{:=\delta_{\Omega}} \sigma(t),
                \end{aligned}
                \end{align}
        Since \eqref{48} and $I(u(t))<0$ for $t \in (0,T_{\max})$ yield that $\sigma^{\prime}(t) = -2I(u(t)) > 0$ for a.e. $t \in (0,T_{\max})$, the function $\sigma(t)$ is increasing in $t$, i.e., $\sigma(t) \ge \sigma(0)$; then, it follows from \eqref{184} that
                \begin{align*}
                (\sigma^{\prime}(t))^{\frac{2}{q^{+}}} + (\sigma^{\prime}(t))^{\frac{2}{q^{-}}}
                \geq
                \delta_{\Omega} \sigma(0),
                \end{align*}
        which yields
                \begin{align}\label{126}
                  \sigma^{\prime}(t) \ge
                  \min\left\{\left(\frac{\delta_{\Omega} \sigma(0)}{2}\right)^{\frac{q^+}{2}},
                             \left(\frac{\delta_{\Omega} \sigma(0)}{2}\right)^{\frac{q^-}{2}}\right\} := \delta_{\sigma}.
                \end{align}
        Then, substituting \eqref{126} into \eqref{184}, we get
            \begin{align*}
                  (\sigma^{\prime}(t))^{\frac{2}{q^{-}}}
                  \ge \left( 1+ (\sigma^{\prime}(t))^{\frac{2}{q^{+}} - \frac{2}{q^{-}}} \right)^{-1}
                  \delta_{\Omega} \sigma(t)
                  \ge \left( 1+ \delta_{\sigma}^{\frac{2}{q^{+}} - \frac{2}{q^{-}}} \right)^{-1}
                  \delta_{\Omega} \sigma(t),
                \end{align*}
            which implies that
                \begin{align*}
                  \sigma^{\prime}(t)
                  \ge
                  \bar{\delta}_{\sigma}
                  \left(\sigma(t)\right)^{{\frac{q^{-}}{2}}}
                \end{align*}
            with $\bar{\delta}_{\sigma} = \left( 1+ \delta_{\sigma}^{\frac{2}{q^{+}} - \frac{2}{q^{-}}} \right)^{-\frac{q^{-}}{2}}
                  \delta^{{\frac{q^{-}}{2}}}_{\Omega} $, namely
                \begin{align}\label{109}
                   \left( \sigma^{1 - \frac{q^{-}}{2}}(t) \right)^{\prime} \le \left(1 - \frac{q^-}{2}\right)\bar{\delta}_{\sigma}.
                \end{align}
            Then one can integrate the both sides of \eqref{109} over $(0,t)$ to obtain
                \begin{align*}
                  \sigma^{1 - \frac{q^{-}}{2}}(t)
                  \le \sigma^{1 - \frac{q^{-}}{2}}(0) + \left(1 - \frac{q^-}{2}\right)\bar{\delta}_{\sigma}t,
                \end{align*}
            which means also
                \begin{align*}
                  \sigma(t) \ge \left( \sigma^{1 - \frac{q^{-}}{2}}(0) + \left(1 - \frac{q^-}{2}\right)\bar{\delta}_{\sigma}t \right)^{\frac{2}{2 - q^-}},
        \end{align*}
    and this suffices to show the validity of \eqref{185}.
    \end{proof}

Then, we show that the local solutions generated by initial data $0<J(u_0)\le d$ inside the stable invariant manifold $\mathcal{N}^+$ are actually global in time and the asymptotic behavior of the global solution with stronger restriction $0<J(u_0)< d$ can be characterized.

        \begin{theorem}[Global existence and polynomial decay]\label{t3}
              Let $(\mathrm{H_{p,q}})$, $(\mathrm{H^1_{p,q}})$-$(\mathrm{H^2_{p,q}})$ and $(\mathrm{H_{D}})$ be satisfied. Provided $u_0 \in \mathcal{N}^+$ with $0<J(u_0) \le d$, then $u_{0} \in \mathcal{G}$. Moreover, if $J(u_0)<d$, then
              the $L^2$ norm of the solution decays polynomially, i.e.,
                    \begin{align}\label{49}
                        \|u(t)\|_{2,\mathcal{A}}
                        \le \left(
                       \left(\frac{p^+}{2}-1 \right)\widetilde{\kappa} t
                           + \|u_0\|^{2-p^+}_{2,\mathcal{A}}
                            \right)^{\frac{1}{2-p^+}},
                        \quad t\in(0,\infty)
                    \end{align}
              where
                \begin{align}\label{170}
                  \widetilde{\kappa} := 2 \min \left\{1,r_{{\rm ptw}}^{p^- - p^+} \right\} C^{*}_{p(\cdot),2}
                            \left( 1 - \left(\frac{d}{J\left(u\right)}\right)
                                    ^{\frac{p^{+}-q^{-}}{q^{+}}} \right) > 0,
                \end{align}
               the constant $r_{{\rm ptw}}$ is given by \eqref{25} and $C^*_{p(\cdot),2}$ is the Sobolev constant of the embedding
              $W_0^{1,p(\cdot)}(\mathcal{A}) \hookrightarrow L^2(\mathcal{A})$.
        \end{theorem}
        \begin{proof}
\textbf{Step 1: Global existence.}
    Recalling that Theorem \ref{l1} has ensured the local existence of the solution of problem \eqref{1}-\eqref{3} on $(0,T_{\max})$, thus we can obtain its global existence by verifying the uniform-in-$t$ boundedness of the norms $\|\nabla u(t)\|_{p(\cdot),\mathcal{A}}$ and $\int_{0}^{t} \left\|u_{t}(\tau)\right\|_{2,\mathcal{A}}^{2} {\rm d}\tau$. Note that $I(u(t))>0$ for $t \in (0,T_{\max})$ with $I(u_0)>0$ according to Lemma \ref{l7}, then the definitions of $J(u)$ and $I(u)$ claim that for every $t \in (0,T_{\max})$
        \begin{align}\label{89}
        \begin{aligned}
        J(u(t))
          >& \left( \frac{1}{p^+} - \frac{1}{q^-} \right)
                    \min \left\{
                       \|\nabla u(t)\|_{p(\cdot),\mathcal{A}}^{p^{-}},
                       \|\nabla u(t)\|_{p(\cdot),\mathcal{A}}^{p^{+}}
                       \right\},
        \end{aligned}
        \end{align}
    which together with \eqref{47} yields
        \begin{align}\label{91}
        \begin{aligned}
          &\int_{0}^{t} \left\|u_{t}(\tau)\right\|_{2,\mathcal{A}}^{2} {\rm d}\tau
          + \left( \frac{1}{p^+} - \frac{1}{q^-} \right)
                    \min \left\{
                       \|\nabla u(t)\|_{p(\cdot),\mathcal{A}}^{p^{-}},
                       \|\nabla u(t)\|_{p(\cdot),\mathcal{A}}^{p^{+}}
                       \right\} < d.
        \end{aligned}
        \end{align}
    In particular,
    \[
    \sup_{0<t<T_{\max}}\|\nabla u(t)\|_{p(\cdot),\mathcal{A}}<\infty.
    \]
    Hence, by Lemma \ref{prop:continuation}, the maximal existence time cannot be finite. Consequently, $T_{\max}=\infty$, and hence $u_0\in\mathcal{G}$.

    \textbf{Step 2: Asymptotic behavior.}
    Next, we discuss the asymptotic behavior of the global solution obtained above for subcritical initial energy, i.e., $J(u_0) <d$.
    It follows from \eqref{48} and \eqref{33} that
        \begin{align}\label{55}
        \begin{aligned}
          \frac{\rm d}{{\rm d} t} \|u\|^2_{2,\mathcal{A}}
          = -2 I(u)
          \le 2 \left( \left(\frac{d}{J\left(u\right)}\right)
                        ^{\frac{p^{+}-q^{-}}{q^{+}}} - 1\right)
                \int_{\mathcal{A}}|\nabla u|^{p(x)} \mathrm{d} x.
        \end{aligned}
        \end{align}
    Recalling that $J(u_0)< d$, $I(u_0) >0$, then we can deduce from Lemma \ref{l7}, Lemma \ref{l17} (i) and Lemma \ref{l4} that
    $\|\nabla u(t) \|_{p(\cdot),\mathcal{A}} \le r_{{\rm ptw}}$ (i.e.,
    $\|r^{-1}_{{\rm ptw}} \nabla u\|_{p(\cdot),\mathcal{A}} \le 1$) for $t \in (0,\infty)$, which together with Proposition \ref{p1}
    and the embedding $W_0^{1,p(\cdot)}(\mathcal{A}) \hookrightarrow L^2(\mathcal{A})$ implies that
        \begin{align}\label{57}
          \rho(r^{-1}_{{\rm ptw}} \nabla u)
          \ge \|r^{-1}_{{\rm ptw}} \nabla u\|^{p^+}_{p(\cdot),\mathcal{A}}
          \ge r^{-p^+}_{{\rm ptw}} \left(C^{*}_{p(\cdot),2}\right)^{p^+} \|u\|_{2,\mathcal{A}}^{p^+},
        \end{align}
    where $C^*_{p(\cdot),2}$ is the Sobolev constant of the embedding
              $W_0^{1,p(\cdot)}(\mathcal{A}) \hookrightarrow L^2(\mathcal{A})$.
    Thus from Proposition \ref{p1}, we deduce from \eqref{57} that
        \begin{align}\label{56}
          \int_{\mathcal{A}}|\nabla u|^{p(x)} \mathrm{d} x
          \ge \min \left\{1,r_{{\rm ptw}}^{p^- - p^+} \right\} \left(C^{*}_{p(\cdot),2}\right)^{p^+} \|u\|_{2,\mathcal{A}}^{p^+}.
        \end{align}
    Then combining \eqref{55} and \eqref{56}, there holds
        \begin{align}\label{164}
              \frac{\rm d}{{\rm d} t} \|u\|^2_{2,\mathcal{A}}
              \le -\widetilde{\kappa} \|u\|^{p^+}_{2,\mathcal{A}} ,
        \end{align}
    where $\widetilde{\kappa}$ is defined by \eqref{170}.
    By setting $\eta(t) := \|u\|^2_{2,\mathcal{A}}$, from \eqref{164} we obtain
        \begin{align}\label{58}
          \frac{\rm d}{{\rm d} t} \eta^{1 - \frac{p^+}{2}}(t) \ge \left( \frac{p^+}{2} - 1 \right) \widetilde{\kappa}.
        \end{align}
    Then integrating the both sides of \eqref{58} over $(0,t)$, we obtain
        \begin{align*}
          \eta^{1 - \frac{p^+}{2}}(t) \ge \eta^{1 - \frac{p^+}{2}}(0) + \left( \frac{p^+}{2} - 1 \right) \widetilde{\kappa} t,
        \end{align*}
    which suffices to show the validity of \eqref{49}.
        \end{proof}

        Next, we show the finite-time blow-up of the solution for $0<J(u_0)\le d$ inside the unstable manifold $u_0 \in \mathcal{N}^-$. To this end, we give the following lemma first.

        \begin{lemma}[A differential inequality for concave method \cite{r6}]\label{l2}
        Suppose that a positive, twice-differentiable function $\xi(t)$ satisfies the inequality
            $$
            \xi^{\prime \prime}(t) \xi(t)-(1+\theta)\left(\xi^{\prime}(t)\right)^{2} \geq 0, \quad t>0,
            $$
            where $\theta>0$ is some constant. If $\xi(0)>0$ and $\xi^{\prime}(0)>0$, then there exists $0<t_{1} \leq \frac{\xi(0)}{\theta \xi^{\prime}(0)}$ such that $\xi(t)$ tends to infinity as $t \rightarrow t_{1}$.
        \end{lemma}

        \begin{theorem}[Finite-time blow-up and upper bound of the blow-up time]\label{t4}
        Let $(\mathrm{H_{p,q}})$, $(\mathrm{H^1_{p,q}})$-$(\mathrm{H^2_{p,q}})$ and $(\mathrm{H_{D}})$ be satisfied. Suppose that $0<J\left(u_{0}\right)\le d$ and $u_0 \in \mathcal{N}^-$, then $u_{0} \in \mathcal{B}$.
        Additionally, if $J(u_0)<d$, the upper bound of the maximal time can be estimated by
            \begin{align*}
              T_{\max} \leq \frac{4\left\|u_{0}\right\|_{2,\mathcal{A}}^{2}}
              {\left(q^{-}-2\right)^{2} \alpha_1},
            \end{align*}
        where $\alpha_1$ is a constant satisfying $\alpha_1 \in \left(0,
        \frac{q^{-}\left(d-J\left(u_{0}\right)
        \right)}{q^{-}-1}\right)$.
        \end{theorem}
        \begin{proof}
        \textbf{Step 1: Finite-time blow-up.}
        Recalling that Theorem \ref{l1} claims the local existence of the solution.
        Firstly, we deal with the case $J(u)<d$. Arguing by contradiction, we suppose that $T_{\max}=+\infty$ and take $T>0$ sufficiently large to define the following auxiliary function
            \begin{align}\label{69}
                M(t):=\frac{1}{2} \int_{0}^{t}\|u\|_{2,\mathcal{A}}^{2} {\rm d} \tau+\frac{1}{2}(T-t)\left\|u_{0}\right\|_{2,\mathcal{A}}^{2}+\frac{1}{2} \alpha(t+\beta)^{2}
                \quad \text { for } t \in[0, T],
            \end{align}
        where $\alpha,\beta$ are positive constants that will be determined later. Then for any $t \in[0, T]$, it is easy to verify that $M(t)>0$.
        Then a direct calculation yields
            \begin{align}\label{70}
            \begin{aligned}
            M^{\prime}(t) =\frac{1}{2}\|u\|_{2,\mathcal{A}}^{2}-\frac{1}{2}\left\|u_{0}\right\|_{2,\mathcal{A}}^{2}
                + \alpha(t+\beta)
            =\int_{0}^{t}\left(u, u_{\tau}\right)_{\mathcal{A}} {\rm d} \tau
                + \alpha(t+\beta)
            \end{aligned}
            \end{align}
        and
            \begin{align}\label{71}
              M^{\prime \prime}(t) =\left(u, u_{\tau}\right)_{\mathcal{A}} +\alpha.
            \end{align}

        On the one hand, since Lemma \ref{l7} yields $I(u(t))<0$ for $t\in(0,T)$ with $u_0 \in \mathcal{N}^-$. Then, according to \eqref{34} and \eqref{47}, there holds
            \begin{align}\label{72}
            \begin{aligned}
               -I(u)
               \ge
                 -q^- \left(J(u) - d \right)
                \ge  -q^- \left(
                    J(u_0) - \int_{0}^{t} \|u\|^2_{2,\mathcal{A}} {\rm d} \tau \right)
                    + q^- d.
            \end{aligned}
            \end{align}
        Then from \eqref{48}, \eqref{72} and \eqref{71}, we have
            \begin{align}\label{76}
              M''(t) \ge -q^- \left(
                    J(u_0) - \int_{0}^{t} \|u\|^2_{2,\mathcal{A}} {\rm d} \tau \right)
                    + q^- d +\alpha.
            \end{align}

        On the other hand, by the Cauchy-Schwarz inequality and Young's inequality, for any $t \in\left[0, T \right]$, we have
            \begin{align*}
            \begin{aligned}
            &\left(\int_{0}^{t}\|u\|_{2,\mathcal{A}}^{2} {\rm d} \tau
                    +\alpha(t + \beta)^{2}\right)
                \left(\int_{0}^{t}\left\|u_{\tau}\right\|_{2,\mathcal{A}}^{2} {\rm d} \tau+\alpha\right)
                -\left(\int_{0}^{t}\left(u, u_{\tau}\right)_{\mathcal{A}} {\rm d} \tau+\alpha(t+\beta)\right)^{2} \\
            =&\left(\int_{0}^{t}\|u\|_{2,\mathcal{A}}^{2} {\rm d} \tau \int_{0}^{t}\left\|u_{\tau}\right\|_{2,\mathcal{A}}^{2} {\rm d} \tau
                -\left(\int_{0}^{t}\left(u, u_{\tau}\right)_{\mathcal{A}} {\rm d} \tau\right)^{2}\right) \\
            &+\left(\alpha \int_{0}^{t}\|u\|_{2,\mathcal{A}}^{2} {\rm d} \tau
                +\alpha(t+\beta)^{2} \int_{0}^{t}\left\|u_{\tau}\right\|_{2,\mathcal{A}}^{2} {\rm d} \tau-2 \alpha(t+\beta) \int_{0}^{t}\left(u, u_{\tau}\right)_{\mathcal{A}} {\rm d} \tau\right) \\
            \geq & 2 \alpha(t+\beta)\left(\int_{0}^{t}\|u\|_{2,\mathcal{A}}^{2} {\rm d} \tau\right)^{\frac{1}{2}}\left(\int_{0}^{t}\left\|u_{\tau}\right\|_{2,\mathcal{A}}^{2} {\rm d} \tau\right)^{\frac{1}{2}}-2 \alpha(t+\beta) \int_{0}^{t}\left(u, u_{\tau}\right)_{\mathcal{A}} {\rm d} \tau \geq 0,
            \end{aligned}
            \end{align*}
        which implies that
            \begin{align}\label{74}
            \begin{aligned}
              (M'(t))^2
              =&
              \left(\int_{0}^{t}\left(u, u_{\tau}\right)_{\mathcal{A}} {\rm d} \tau
              +\alpha (t+ \beta)\right)^{2} \\
              \le &
              \left(\int_{0}^{t}\|u\|_{2,\mathcal{A}}^{2} {\rm d} \tau
                +\alpha (t + \beta)^{2}\right)
                \left(\int_{0}^{t}\left\|u_{\tau}\right\|_{2,\mathcal{A}}^{2} {\rm d} \tau
                + \alpha \right).
            \end{aligned}
            \end{align}
        Then combing \eqref{76} and \eqref{74}, we get
            \begin{align*}
            \begin{aligned}
            & M(t) M^{\prime \prime}(t)
                - \frac{q^-}{2} \left(M^{\prime}(t)\right)^{2} \\
            \geq & M(t) M^{\prime \prime}(t)-\frac{q^-}{2}
            \left(2 M(t)-(T-t)\left\|u_{0}\right\|_{2,\mathcal{A}}^{2}\right)
                \left(\int_{0}^{t}\left\|u_{\tau}\right\|_{2,\mathcal{A}}^{2}
                {\rm d} \tau + \alpha\right) \\
            \geq & M(t)\left(M^{\prime \prime}(t)
                -q^- \left(\int_{0}^{t}\left\|u_{\tau}\right\|_{2,\mathcal{A}}^{2} {\rm d} \tau+\alpha\right)\right) \\
            \geq & M(t) \left( q^- d
             -q^- J\left(u_{0}\right)-(q^--1) \alpha
             \right) > 0,
            \end{aligned}
            \end{align*}
         where the constant $\alpha$ satisfies $0 < \alpha < \frac{q^- (d - J(u_0))}{q^- - 1}$. From Lemma \ref{l2}, we can deduce that there exists a $T^*>0$ such that
            $$
                \lim _{t \rightarrow {T^{*}}^-} M(t)=+\infty
            $$
        and
            \begin{align}\label{51}
              T_{\max}<T^{*} \leq \frac{2 M(0)}{(q^- - 2) M^{\prime}(0)} < \infty.
            \end{align}

        Next, we deal with the case $J(u_0)=d$. Since $J\left(u_{0}\right)=d>0$ and $I\left(u_{0}\right)<0$,
        then the continuity of $J(u(t))$ and $I(u(t))$ in $t$ yields the existence of a sufficiently small $\varepsilon>0$ such that $J(u(t))>0$ and
        $I(u)= -\left(u, u_{t}\right)_{\mathcal{A}}<0$ for $t \in [0, \varepsilon)$,
        which implies that $u_{t} \not\equiv 0$ on the interval $[0, \varepsilon)$.
        Therefore, Lemma \ref{l8} claims that
            \begin{align*}
             0<J\left(u\left(\varepsilon\right)\right) \leq d-\int_{0}^{\varepsilon}\left\|u_{\tau}\right\|_{2,\mathcal{A}}^{2} {\rm d} \tau <d.
            \end{align*}
        Thus we can deduce that $u \in \mathcal{N}^-$ for
        $t \in (\varepsilon, \infty]$ by Lemma \ref{l7}. Then similar as what we did for the case $J(u_0)<d$ in the proof of Theorem \ref{t4} above, we can easily get
            \begin{align*}
            \begin{split}
               & M(t) M''(t) -\frac{q^-}{2}\left(M'(t)\right)^{2}
               > M(t)
               \left(q^-\left(d-J\left(u\left(\varepsilon\right)\right)\right)
               -(q^--1) \alpha\right)>0.
            \end{split}
            \end{align*}
        Then, we can apply Lemma \ref{l2} again to obtain that the solution doesn't exist globally, which together with Lemma \ref{prop:continuation} leads to the finite-time blow-up of the solution.

        \textbf{Step 2: Upper bound of the blow-up time.}
        Taking \eqref{69} and \eqref{70} into \eqref{51}, for any $\beta>0$ and $0<\alpha<\frac{q^{-}
        \left(d-J\left(u_{0}\right)\right)}{q^{-}-1}$, we have
            \begin{align*}
            \frac{2 M(0)}{(q^{-} -2) M^{\prime}(0)}
            = \frac{\left\|u_{0}\right\|_{2,\mathcal{A}}^{2} T + \alpha \beta^{2}}{\left(q^{-} -2\right) \beta \alpha}
            = \frac{\left\|u_{0}\right\|_{2,\mathcal{A}}^{2}}
            {\left(q^{-}-2\right) \beta \alpha} T
            + \frac{\beta}{q^{-}-2},
            \end{align*}
        which implies that for any $T \in (0,T_{\max})$
            \begin{align}\label{78}
              \left(1-\frac{\left\|u_{0}\right\|_{2,\mathcal{A}}^{2}}
              {\left(q^{-}-2\right) \beta \alpha}\right) T \leq \frac{\beta}{q^{-}-2}.
            \end{align}
        Then, by taking
        $\beta \in\left(\frac{\left\|u_{0}\right\|_{2,\mathcal{A}}^{2}}
        {\left(q^{-}-2\right) \alpha},+\infty\right)$ such that
            $$
            \begin{aligned}
            T & \leq \frac{\beta}{q^{-}-2}\left(1-\frac{\left\|u_{0}\right\|_{2,\mathcal{A}}^{2}}
            {\left(q^{-}-2\right) \beta \alpha}\right)^{-1}
            =\frac{\alpha \beta^{2}}{\left(q^{-}-2\right) \beta \alpha-\left\|u_{0}\right\|_{2,\mathcal{A}}^{2}}
            :=T_{\alpha}(\beta),
            \end{aligned}
            $$
        where $T_{\alpha}(\beta)$ is positive and takes its minimum at $\beta=\frac{2\left\|u_{0}\right\|_{2,\mathcal{A}}^{2}}
        {\left(q^{-}-2\right) \alpha}$, we get that
            \begin{align}\label{50}
             T_{\max} \leq T_{\alpha}\left(\frac{2\left\|u_{0}\right\|_{2,\mathcal{A}}^{2}}{\left(q^{-}-2\right) \alpha_1}\right)=\frac{4\left\|u_{0}\right\|_{2,\mathcal{A}}^{2}}{\left(q^{-}-2\right)^{2} \alpha_1}
            \end{align}
           according to the arbitrariness of $T$ with $T \in (0,T_{\max})$ for $\alpha_1 \in \left(0,\frac{q^- (d - J(u_0))}{q^- - 1}\right)$.
        \end{proof}

\subsection{Supercritical initial energy $J(u_0) > d$}\label{sec5}
In this subsection, we consider the global well-posedness of the solution to problem \eqref{1}-\eqref{3} with supercritical initial energy $J(u_0) > d$. We not only give the abstract criteria of the global existence and finite-time blow-up of the solution in Theorem \ref{t5}, but also give a sufficient condition of the finite-time blow-up of the solution and estimate the upper bound of its maximal existence time in Theorem \ref{t2}.

Next, we claim that any global solution to \eqref{1}-\eqref{3} converges to an equilibrium up to a time sequence as shown in the following Lemma \ref{l16}, where the existence of the nontrivial equilibrium is ensured by Lemma \ref{l11}.
Moreover, since
$W_{0,\mathrm{rad}}^{1, p(\cdot)}\left(\mathcal{A}\right) \hookrightarrow L^{q(\cdot)}\left(\mathcal{A}\right)$ compactly and $L^{q(\cdot)}(\mathcal{A}) \hookrightarrow L^2(\mathcal{A})$ (recalling that $q^- \ge 2$), Lemma \ref{l16} can be similarly proved as the subcritical regime case, cf. Theorem 6.3, \cite{r2}.

        \begin{lemma}[Convergence to an equilibrium]\label{l16}
        Let $(\mathrm{H_{p,q}})$, $(\mathrm{H^1_{p,q}})$-$(\mathrm{H^2_{p,q}})$ and $(\mathrm{H_{D}})$ be satisfied. If $u_0\in W_{0,\mathrm{rad}}^{1, p(\cdot)}\left(\mathcal{A}\right) \cap \mathcal{G}$, then there exists a solution $\phi$ of \eqref{16} such that the solution of \eqref{1}-\eqref{3} driven by $u_0$ converges to the stationary solution in $W_{0,\mathrm{rad}}^{1, p(\cdot)}\left(\mathcal{A}\right)$ up to a time sequence.
        \end{lemma}

By virtue of Lemma \ref{l16}, we give the abstract criteria for the global existence and finite-time blow-up of the solution respectively in the following Theorem \ref{t5}. Due to the invariant manifolds given by Lemma \ref{l15}, Theorem \ref{t5} can be proved by a contradiction argument as Lemma 8 \cite{r7}, which is based on the nonempty of the $\omega$-limit set ensured by Lemma \ref{l16} and
is not affected by the multiple-regime assumption. Hence, we omit the corresponding proof here.

        \begin{theorem}[Global existence and finite-time blow-up]\label{t5}
            Let $(\mathrm{H_{p,q}})$, $(\mathrm{H^1_{p,q}})$-$(\mathrm{H^2_{p,q}})$ and $(\mathrm{H_{D}})$ be satisfied. Suppose that $J(u_0) > d$, then we have

            \begin{enumerate}[(i)]
                \item if $\|u_0\|_{2,\mathcal{A}} \le \lambda_{J(u_0)}$ and $u_0 \in \mathcal{N}^+$, then
                          $u_{0} \in \mathcal{G}_0$;
                \item if $\|u_0\|_{2,\mathcal{A}} \ge \Lambda_{J(u_0)}$ and $u_0 \in \mathcal{N}^-$, then
                          $u_{0} \in \mathcal{B}$.
            \end{enumerate}
        \end{theorem}

        Then, we give a sufficient condition for the finite-time blow-up of the solution with arbitrarily positive initial energy $J(u_0)>0$.

        \begin{theorem}[Finite-time blow-up for arbitrarily positive $J(u_0)$]\label{t2}
        Assume $(\mathrm{H_{p,q}})$, $(\mathrm{H^1_{p,q}})$-$(\mathrm{H^2_{p,q}})$ and $(\mathrm{H_{D}})$.
        Suppose that $u_{0} \in W_{0,\mathrm{rad}}^{1, p(\cdot)}\left(\mathcal{A}\right)$ satisfies \eqref{93},
        then the solution to problem \eqref{1}-\eqref{3} blows up in finite time, and the upper bound of the maximal existence time can be estimated by
             $$
                T_{\max} \leq \frac{4\left\|u_{0}\right\|_{2,\mathcal{A}}^{2}}
                {\left(q^{-}-2\right)^{2}
                \alpha_2},
             $$
        where the constant $\alpha_2$ satisfies $\alpha_2 \in \left(0, P(\|u_0\|_{2,\mathcal{A}}) - J(u_{0})\right)$ and $P(\cdot)$ is defined by \eqref{100}.
        \end{theorem}
        \begin{proof}
          Following the similar idea of the proof of Theorem \ref{t4},
        we start the proof of this theorem with the local existence of the solution ensured by Theorem \ref{l1}. Still arguing by contradiction, we suppose that $T_{\max} = \infty$. Then we take a sufficiently large $T>0$ and employ the same auxiliary function $M(t)$ defined by \eqref{69} in the proof of Theorem \ref{t4}.

        Firstly, we claim that $u(t)$ remains in $\mathcal{N}^-$ as $t$ evolves according to Lemma \ref{l10}. Then similar to \eqref{70}-\eqref{74}, for any $t \in(0, T)$ we get
            \begin{align*}
            \begin{aligned}
             &M(t) M^{\prime \prime}(t)
                - \frac{q^-}{2} \left(M^{\prime}(t)\right)^{2}\\
            \geq & M(t)\left(M^{\prime \prime}(t)
                -q^- \left(\int_{0}^{t}\left\|u_{\tau}\right\|_{2,\mathcal{A}}^{2} {\rm d} \tau+\alpha\right)\right) \\
            \geq & M(t) \left(  \alpha -2 I(u) +  \alpha + q^- J(u)
                -q^- J(u_0) -q^- \alpha \right) \\
            >& q^- M(t) \left(  J(u) - \frac{2}{q} I(u) - J(u_0)
                    - \alpha \right)\\
            >& q^- M(t) \left( P(\|u_0\|_{2,\mathcal{A}}) - J(u_0) -\alpha \right) \\
            >& 0
            \end{aligned}
            \end{align*}
        with
        $\alpha \in
            \left(0, P(\|u_0\|_{2,\mathcal{A}}) - J(u_{0})\right)$.
        By an application of Lemma \ref{l2} again, we deduce that there exists a $T^{**}$ such that
            $$
                \lim _{t \rightarrow {T^{**}}^-} M(t)=+\infty.
            $$

       Next, we estimate the upper bound of the blow-up time.
       Recalling \eqref{78}-\eqref{50} and
        $\alpha_2 \in\left(0, P(\|u_0\|_{2,\mathcal{A}}) - J(u_{0})\right)$,
        we can similarly derive that
            $$
            T_{\max}
            \leq T_{\alpha}
                \left(\frac{2
                \left\|u_{0}\right\|_{2,\mathcal{A}}^{2}}{\left(q^{-}-2\right) \alpha_2}\right)
            =\frac{4\left\|u_{0}\right\|_{2,\mathcal{A}}^{2}}
                {\left(q^{-}-2\right)^{2} \alpha_2}.
            $$
        \end{proof}

        \begin{remark}[On Theorem \ref{t2}]
          In Theorem 6.5 of \cite{r2}, the authors proposed a sufficient condition for the finite-time blow-up of the solution with supercritical initial energy, and they set
                $$
                \left\|u_{0}\right\|_{2,\Omega}^{2} \geq \frac{2 p^{+}}{q^{-}-p^{+}} J\left(u_{0}\right)+\frac{q^{+}-2}{q^{-}}|\Omega|
                $$
          to make sure that $\left\|u_{0}\right\|_{2,\Omega} \ge \Lambda_{J(u_0)}$, which remains true for the multiple-regime case on an annulus studied in the present paper. We also point out that Theorem \ref{t2} is valid for any initial data with arbitrarily positive energy satisfying the constraint \eqref{93}, not only for the supercritical initial energy level.
        \end{remark}

\section{Generalization for locally symmetric domains}\label{sec6}
In this section, we tend to generalize the well-posedness results given in Sections \ref{sec3}-\ref{sec33} (including the local and global existence and the finite-time blow-up) of the solution by reconsidering problem \eqref{1}-\eqref{3} on a more general class of domains $\Omega_{\mathcal{A}}$ containing an annulus centered at the origin, i.e.,
    \begin{align}
          &u_t-\Delta_{p(x)} u=|u|^{q(x)-2} u && \text { in } \Omega_{\mathcal{A}} \times(0,+\infty), \label{186}\\
          &u=0, && \text { on } \partial \Omega_{\mathcal{A}} \times(0,+\infty), \label{187}\\
          &u(x, 0)=u_0(x), && \text { in } \Omega_{\mathcal{A}}, \label{188}
    \end{align}
where $\Omega_{\mathcal{A}} \subset \mathbb{R}^n$ ($n\ge 2$) is a bounded domain with smooth boundary satisfying the hypothesis ($\Omega_{R, r}$) introduced below.\\

$\bullet$ {\bf Hypotheses.}
Firstly, we illustrate the hypotheses on the geometry of the domain and the variable exponents $p(x)$ and $q(x)$.

{\flushleft\textbf{($\Omega_{R, r}$)}}
   \quad
    There exists an open ball $B_R(0)$ centered at the origin with radius $R>0$ such that
    $\bar{B}_{R}(0) \subset \Omega_{\mathcal{A}}$.\\

  Since the ball $B_{R}(0)$ contains the annulus $\mathcal{A}$ for $0< r < R$,
  one can observe that hypothesis ($\Omega_{R, r}$) implies that $\Omega_{\mathcal{A}}$ is locally radially symmetric in the sense that it contains an annulus centered at the origin.
  Besides, we draw the following Figure \ref{f2} with $n=2$ for an intuitive understanding of hypothesis
  ($\Omega_{R, r}$).

    \begin{figure}[h]
    \centering
            \begin{overpic}[width=8.5cm,height=6cm]{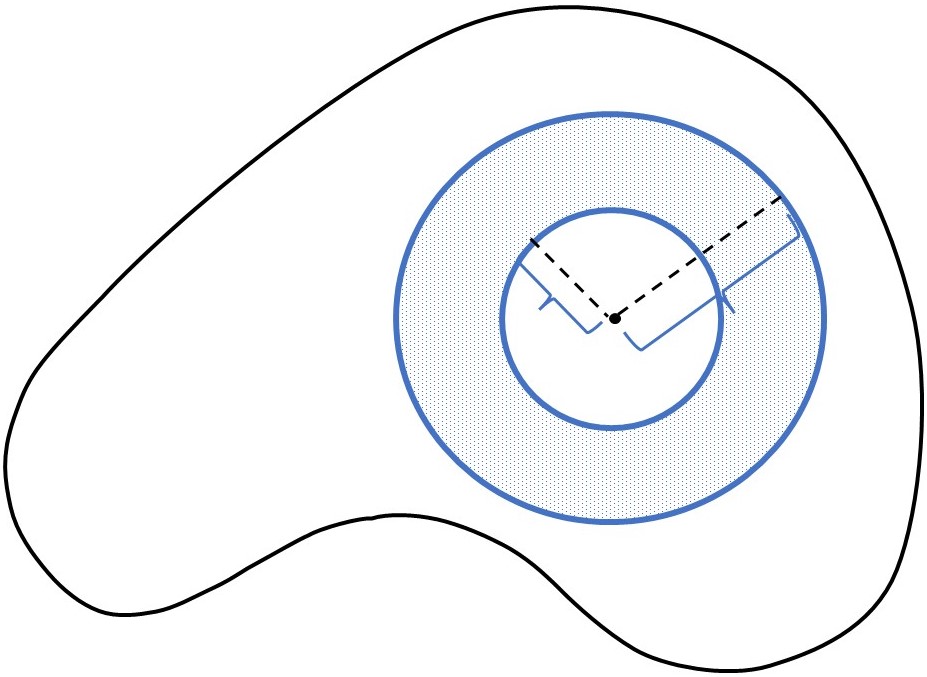}
            \put(157, 80){$0$}
            \put(135, 85){$r$}
            \put(190, 82){$R$}
            \put(115, 70){$\mathcal{A}$}
            \put(145, 68){$B_r(0)$}
            \put(100, 140){$B_R(0)$}
            \put(50, 65){$\Omega_\mathcal{A}$}
            \end{overpic}
    \caption{Locally symmetric $\Omega_{\mathcal{A}}$ satisfying ($\Omega_{R, r}$) for $n=2$}\label{f2}
    \end{figure}

Next, considering the multiple-regime phenomenon of the variable exponents, we set the following hypothesis corresponding to ($\mathrm{H_{D}}$) given in Subsection \ref{sec2.1}.

{\flushleft\textbf{($\mathrm{\bar{H}_{D}}$)}}
    \quad
    Let us define
        \begin{align*}
        \begin{aligned}
          &\Omega_{\mathcal{A},\rm sub}:= \{x\in \Omega_{\mathcal{A}}: q(x)<p^{*}(x)\},\\
          &\Omega_{\mathcal{A},\rm critical}:= \{x\in \Omega_{\mathcal{A}}: q(x)=p^{*}(x)\},\\
          &\Omega_{\mathcal{A},\rm sup}:= \{x\in \Omega_{\mathcal{A}}: q(x)>p^{*}(x)\}.
        \end{aligned}
        \end{align*}
    Moreover, we suppose that $\Omega \backslash \mathcal{A} \subset \Omega_{\mathcal{A},\rm sub}$ and
    $\Omega_{\mathcal{A},\rm critical}, \Omega_{\mathcal{A},\rm sup} \subset \mathcal{A}$ with $\Omega_{\mathcal{A},\rm sub}, \Omega_{\mathcal{A},\rm critical}$ and $\Omega_{\mathcal{A},\rm sup}$ are sets of nonzero measures, where $p^{*}(x)$ is the critical Sobolev exponent given by \eqref{165}.\\

    Hypothesis ($\mathrm{\bar{H}_{D}}$) describes the relationship between the regions corresponding to different regimes
   (i.e., $\Omega_{\mathcal{A},\rm sub}, \Omega_{\mathcal{A},\rm critical}$ and $\Omega_{\mathcal{A},\rm sup}$) and the annulus $\mathcal{A}$, which requires that only the subcritical region appears outside $\mathcal{A}$, and the critical and supercritical regions only appear within $\mathcal{A}$.
   Therefore, compared with the case that the domain is the annulus $\mathcal{A}$ with
   the hypothesis ($\mathrm{H_D}$), the locally symmetric domain $\Omega_{\mathcal{A}}$ satisfying ($\Omega_{R, r}$) with ($\mathrm{\bar{H}_{D}}$) considered here can be regarded as the ``subcritical perturbation'' with respect to the domain.

   Next, we introduce the common boundedness assumption of $p(x),q(x)$ corresponding to $(\mathrm{H_{p,q}})$ as follows.

   {\flushleft\textbf{($\mathrm{\bar{H}_{p,q}}$)}}
   \quad
   The variable exponents $p(x),q(x)$ satisfy
        \begin{align*}
        \begin{cases}
          &2 \le p^{-} \le p(x) \le p^{+} < n,\\
          &2 \le q^{-} \le q(x) \le q^{+} < \infty
        \end{cases}
        \end{align*}
    for all $x\in \Omega_{\mathcal{A}}$, where the constants $p^{\pm},q^{\pm}$ are given by
        \begin{align*}
            & p^{-} := \mathop{\essinf}\limits_{x\in \Omega_{\mathcal{A}}} p(x),\quad
            p^{+} := \mathop{\esssup}\limits_{x\in \Omega_{\mathcal{A}}} p(x), \\
            & q^{-} := \mathop{\essinf}\limits_{x\in \Omega_{\mathcal{A}}} q(x),\quad
            q^{+} := \mathop{\esssup}\limits_{x\in \Omega_{\mathcal{A}}} q(x).
        \end{align*}

   Moreover, we introduce a series of additional hypotheses on the variable exponents $p(x)$ and $q(x)$.

    {\flushleft\textbf{($\mathrm{\bar{H}^1_{p,q}}$)}}
   \quad
        $
          p(x)=p(|x|), \quad q(x)=q(|x|), \quad \forall x \in \bar{\mathcal{A}}.
        $

   {\flushleft\textbf{($\mathrm{\bar{H}^2_{p,q}}$)}}
   \quad
        $
         q(x) \leq 1 + \frac{p^{*}(x)}{2}, \quad \forall x \in \Omega_{\mathcal{A}} \backslash \mathcal{A}.
        $

    {\flushleft\textbf{($\mathrm{\bar{H}^3_{p,q}}$)}}
   \quad
        $
        p^{+} <  q(x), \quad \forall x \in \Omega_{\mathcal{A}}.
        $

   {\flushleft\textbf{($\mathrm{\bar{H}^4_{p,q}}$)}}
   \quad
        $
        q(x) <  {p^-}^*:= \frac{n p^-}{n - p^-}, \quad \forall x \in \Omega_{\mathcal{A}}\backslash \mathcal{A}.
        $\\

    Since the hypotheses $(\mathrm{\bar{H}^1_{p,q}})$-($\mathrm{\bar{H}^3_{p,q}}$) introduced above are relatively complex compared with ($\mathrm{H^1_{p,q}}$)-($\mathrm{H^2_{p,q}}$) given in Subsection \ref{sec2.1},
    to better understand the meaning of these hypotheses, we briefly explain them as follows.
   Hypothesis $(\mathrm{\bar{H}^1_{p,q}})$ requires the variable exponents $p(x),q(x)$ to be symmetric within the annulus $\mathcal{A}$ with respect to $x$.
   Hypotheses $(\mathrm{\bar{H}^2_{p,q}})$ and $(\mathrm{\bar{H}^4_{p,q}})$ assert that $q(x)$ is of the subcritical regime outside $\mathcal{A}$.
   Hypothesis $(\mathrm{\bar{H}^3_{p,q}})$ implies that the diffusion term is dominated by the nonlinear source term on $\Omega_{\mathcal{A}}$, which leads to the mountain-pass geometry of the energy functional.

   \begin{remark}[On hypothesis $(\mathrm{\bar{H}^2_{p,q}})$]\label{rmk1}
      In $(\mathrm{\bar{H}^2_{p,q}})$, we don't employ the same hypothesis (H7) in Subsection 4.1 of \cite{r1}
      due to the fact that $p^- \ge 2$, which means for a.e. $x \in \Omega_{\mathcal{A}}$
            \begin{align*}
              {p}^*(x) \ge {p^-}^* = \frac{np^-}{n - p^-} > p^- \ge 2,
            \end{align*}
      thus there holds
      $\min \left\{p^*(x), 1+\frac{p^*(x)}{2}\right\} = 1+\frac{p^*(x)}{2}$.
    \end{remark}

    \begin{remark}[On hypothesis $(\mathrm{\bar{H}^4_{p,q}})$]\label{rmk2}
    The hypothesis $(\mathrm{\bar{H}^4_{p,q}})$ is equivalent to $q(x) <  {p^-}^*$. Moreover, we emphasize that ($\mathrm{\bar{H}^2_{p,q}}$) can not be obtained from $(\mathrm{\bar{H}^4_{p,q}})$, and $(\mathrm{\bar{H}^4_{p,q}})$ can not be obtained from ($\mathrm{\bar{H}^2_{p,q}}$).
    \end{remark}

    Moreover, considering the local radial symmetry of $\Omega_{\mathcal{A}}$ given by ($\Omega_{R, r}$), we define
        \begin{align}\label{183}
          W_{0,\mathrm{rad}}^{1, p(\cdot)}(\Omega_{\mathcal{A}}):=
          \left\{u \in W_{0}^{1, p(\cdot)}(\Omega_{\mathcal{A}}): u(x)=u(|x|) \text { a.e. } x \in \mathcal{A}\right\},
        \end{align}
    as the working space of the solution.
    And we consider the Sobolev embedding $W_{0,\mathrm{rad}}^{1, p(\cdot)}(\Omega_{\mathcal{A}}) \hookrightarrow L^{q(\cdot)}(\Omega_{\mathcal{A}})$. Since $q(x)$ is of the subcritical regime outside $\mathcal{A}$,
    we can directly deduce from Corollary \ref{cor3} the validity of the following compact embedding.

    \begin{corollary}\label{cor3}
        Let $(\mathrm{\bar{H}_{p,q}})$, $(\mathrm{\bar{H}^1_{p,q}})$, $(\mathrm{\bar{H}^2_{p,q}})$ and ($\mathrm{\bar{H}_{D}}$) be satisfied, then
        $W_{0,\mathrm{rad}}^{1, p(\cdot)}(\Omega_{\mathcal{A}}) \hookrightarrow L^{q(\cdot)}(\Omega)$ compactly.
    \end{corollary}

    Moreover, for the convenience of the statement, in the framework of the potential well theory,
    we shall employ the same notations in Subsection \ref{sec2.3} in the following content with $\mathcal{A}$ replaced by $\Omega_{\mathcal{A}}$ without causing any confusion. For example, the definitions of $I(u)$ and $J(u)$ are now given by
        \begin{align*}
          J(u):=\int_{\Omega_{\mathcal{A}}} \frac{|\nabla u(x)|^{p(x)}}{p(x)} \mathrm{d} x
            -\int_{\Omega_{\mathcal{A}}} \frac{|u(x)|^{q(x)}}{q(x)} \mathrm{d} x
        \end{align*}
    and
        \begin{align*}
          I(u):=\int_{\Omega_{\mathcal{A}}}|\nabla u(x)|^{p(x)} \mathrm{d} x
            -\int_{\Omega_{\mathcal{A}}}|u(x)|^{q(x)} \mathrm{d} x
        \end{align*}
    throughout this section.\\

    $\bullet$ {\bf Local existence.}
    In Section \ref{3}, the Principle of Symmetric Criticality was applied to verify the weak formulation of the equation when the domain is the annulus $\mathcal{A}$, namely the identity \eqref{12} in Definition \ref{d1} for any test function belonging to $W_0^{1,p(x)}(\mathcal{A})$, which requires the radial symmetry of the domain, the variable exponents and the initial data.
    However, under the hypothesis ($\Omega_{R, r}$), the Principle of Symmetric Criticality doesn't work for the domain $\Omega_{\mathcal{A}}$, since the domain is no longer radially symmetric, not to mention defining symmetric functions on such a domain. For example, there exists a point $x_0 \in \Omega_{\mathcal{A}} \backslash B_R(0)$ that may be thrown out of $\Omega_{\mathcal{A}}$ under some rotation $g$, i.e., $g(x_0) \not \in \Omega_{\mathcal{A}}$, shown in the following Figure \ref{f5}.
    Fortunately, by the subdifferential technique in the Banach space setting and the tricks on the test functions with modification from \cite{r1}, the local existence can be obtained as Theorem 4.1 \cite{r1}, for which we introduce the following local existence result in this situation without proof.

    \begin{figure}[h]
    \centering
            \begin{overpic}[width=8.5cm,height=6cm]{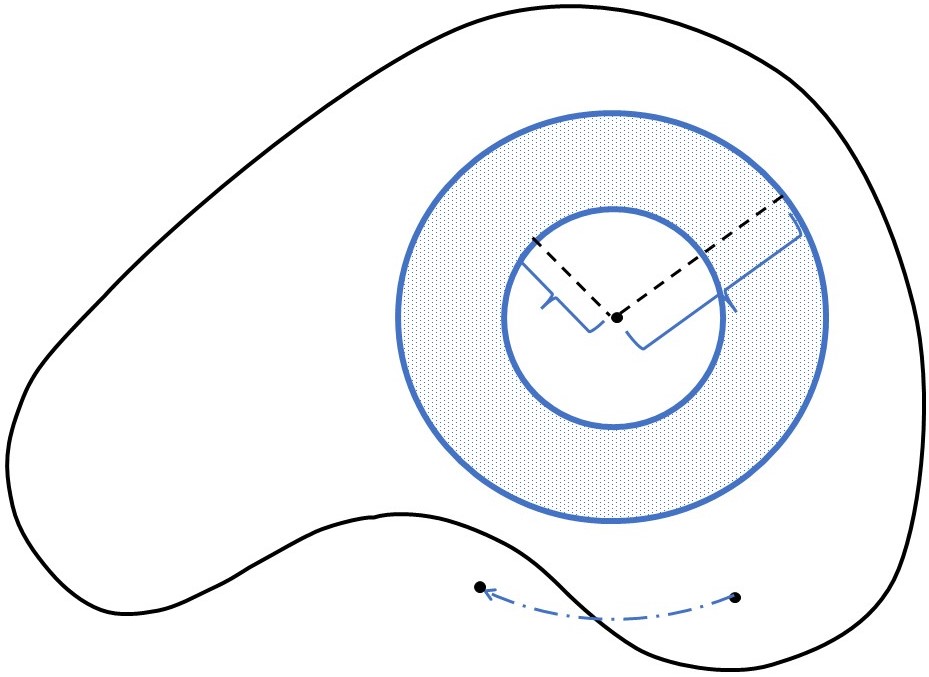}
            \put(157, 80){$0$}
            \put(135, 85){$r$}
            \put(190, 82){$R$}
            \put(115, 70){$\mathcal{A}$}
            \put(145, 68){$B_r(0)$}
            \put(100, 140){$B_R(0)$}
            \put(50, 65){$\Omega_\mathcal{A}$}
            \put(120, 10){$gx_0$}
            \put(190, 10){$x_0$}
            \end{overpic}
    \caption{The point $x_0$ thrown out of $\Omega_{\mathcal{A}}$ under some rotation $g$}\label{f5}
    \end{figure}

    \begin{theorem}[Local existence \cite{r1}]\label{t6}
        Assume that $(\Omega_{R, r})$, $(\mathrm{\bar{H}_{p,q}})$, $(\mathrm{\bar{H}^1_{p,q}})$-$(\mathrm{\bar{H}^3_{p,q}})$ and $(\mathrm{\bar{H}_{D}})$ are satisfied, $u_{0} \in W_{0,\mathrm{rad}}^{1, p(\cdot)}(\Omega_{\mathcal{A}})$, then problem \eqref{186}-\eqref{188} admits a local weak solution $u \in L^{\infty}(0, T_{\max}; W_{0,\mathrm{rad}}^{1, p(\cdot)}(\Omega_{\mathcal{A}}))$ with $u_{t} \in L^{2}\left(0, T_{\max}; L^{2}(\Omega_{\mathcal{A}})\right)$.
    \end{theorem}

    $\bullet$ {\bf Global well-posedness for $J(u_0)\le d$.}
        By the validity of Corollary \ref{cor3} and the local existence of solution ensured by Theorem \ref{t6}, our main results obtained in Subsection \ref{sec4} can be directly extended for
        problem \eqref{186}-\eqref{188} under the hypotheses $(\Omega_{R, r})$, $(\mathrm{\bar{H}_{p,q}})$, $(\mathrm{\bar{H}^1_{p,q}})$-$(\mathrm{\bar{H}^3_{p,q}})$ and ($\mathrm{\bar{H}_{D}}$) in parallel to Theorems \ref{t10}-\ref{t4} for problem \eqref{1}-\eqref{3} without much modification, shown as follows.

        \begin{theorem}
               Assume that $(\Omega_{R, r})$, $(\mathrm{\bar{H}_{p,q}})$, $(\mathrm{\bar{H}^1_{p,q}})$-$(\mathrm{\bar{H}^3_{p,q}})$ and $(\mathrm{\bar{H}_{D}})$ be satisfied. Then the conclusions of Theorems \ref{t10}-\ref{t4} are still valid for the solution to \eqref{186}-\eqref{188}, where the domain $\mathcal{A}$ is replaced by $\Omega_{\mathcal{A}}$ satisfying $(\Omega_{R, r})$.\\
        \end{theorem}

    $\bullet$ {\bf Global well-posedness for $J(u_0)> d$.}
    Firstly, we give a sufficient condition for the finite-time blow-up of the solution to \eqref{186}-\eqref{188} parallel to Theorem \ref{t2} for \eqref{1}-\eqref{3} as follows.

        \begin{theorem}
        Assume that $(\Omega_{R, r})$, $(\mathrm{\bar{H}_{p,q}})$, $(\mathrm{\bar{H}^1_{p,q}})$-$(\mathrm{\bar{H}^3_{p,q}})$ and $(\mathrm{\bar{H}_{D}})$ be satisfied. Then the conclusion of Theorem \ref{t2}
        is still valid for the solution to \eqref{186}-\eqref{188},
        where the domain $\mathcal{A}$ is replaced by $\Omega_{\mathcal{A}}$ satisfying $(\Omega_{R, r})$.
        \end{theorem}

    Then, we consider the abstract criteria for the global existence and finite-time blow-up of the solution to \eqref{186}-\eqref{188} in parallel to Theorem \ref{t5} for \eqref{1}-\eqref{3}.
    To this end, we reconsider the boundedness of the variational numbers $\lambda_k$ and $\Lambda_k$.
    Firstly, one can easily deduce that $\Lambda_k < \infty$ by the virtue of the embedding $W_{0,\mathrm{rad}}^{1, p(\cdot)}(\Omega_{\mathcal{A}}) \hookrightarrow L^2(\Omega_{\mathcal{A}})$, which is analogous to Lemma \ref{l9}, thus we give the following Lemma \ref{l22} without proof.

    \begin{lemma}\label{l22}
      Let $(\Omega_{R, r})$, $(\mathrm{\bar{H}_{p,q}})$, $(\mathrm{\bar{H}^1_{p,q}})$-$(\mathrm{\bar{H}^3_{p,q}})$ and ($\mathrm{\bar{H}_{D}}$) be satisfied, then for $k >d$, there holds $\Lambda_{k}<\infty$.
    \end{lemma}

    Additionally, if $(\mathrm{\bar{H}^4_{p,q}})$ is taken into account, there holds
    $u \in W_0^{1,p^-}(\Omega) \cap L^{q^+}(\Omega)$ and
        \begin{align*}
           W_{0,\mathrm{rad}}^{1, p(\cdot)}(\Omega_{\mathcal{A}})
          \hookrightarrow W_0^{1,p^-}(\Omega)
          \hookrightarrow L^{q^+}(\Omega)
          \hookrightarrow L^{q^-}(\Omega)
        \end{align*}
    according to Corollary \ref{cor3} and Remark \ref{rmk2}, for which we can employ the Gagliardo-Nirenberg interpolation inequality to deduce that $\lambda_{k}$ is bounded below, shown in the following Lemma \ref{l21}.

    \begin{lemma}\label{l21}
     Let $(\Omega_{R, r})$, $(\mathrm{\bar{H}_{p,q}})$, $(\mathrm{\bar{H}^1_{p,q}})$-$(\mathrm{\bar{H}^4_{p,q}})$ and ($\mathrm{\bar{H}_{D}}$) be satisfied, then for $k>d$, there holds
                $
                \lambda_k >0,
                $
     where
        $$
          \lambda_{k}:=\inf \left\{\|u\|_{2,\Omega_\mathcal{A}}: u \in \mathcal{N}_{k}\right\}.
        $$
    \end{lemma}
    \begin{proof}
        Since $I(u) = 0$, there holds
            \begin{align}\label{41}
            \begin{aligned}
              \int_{\Omega_\mathcal{A}}|\nabla u|^{p(x)} \mathrm{d} x
              =& \int_{\Omega_\mathcal{A}}|u|^{q(x)} \mathrm{d} x
              \leq \|u\|_{q^-,\Omega_\mathcal{A}}^{q^-}
                + \|u\|_{q^+,\Omega_\mathcal{A}}^{q^+}.
            \end{aligned}
            \end{align}
        By an application of the Gagliardo-Nirenberg inequality, we get
            \begin{align}\label{42}
              \|u\|_{q^{\pm},\Omega_\mathcal{A}}
              \leq C\|\nabla u\|_{p^{-},\Omega_\mathcal{A}}^{\theta^{\prime}_{\pm}}
              \|u\|_{2,\Omega_\mathcal{A}}^{1-\theta^{\prime}_{\pm}},
            \end{align}
        where $\theta_{-}^{\prime},\theta^{\prime}_{+} \in(0,1)$ are defined by
            \begin{align*}
              \theta'_{-} := \left(\frac{1}{2}-\frac{1}{q^-}\right)
                \left(\frac{1}{n}-\frac{1}{p^{-}}+\frac{1}{2}\right)^{-1},\quad
              \theta'_{+} := \left(\frac{1}{2}-\frac{1}{q^+}\right)
                \left(\frac{1}{n}-\frac{1}{p^{-}}+\frac{1}{2}\right)^{-1}.
            \end{align*}
       Recalling the fact that Proposition \ref{p1} claims
            \begin{align}\label{43}
            \begin{aligned}
                \int_{\Omega_\mathcal{A}}|\nabla u|^{p(x)} \mathrm{d} x
                \ge& \min \left\{
                           \|\nabla u\|_{p (\cdot),\Omega_\mathcal{A}}^{p^{-}},
                           \|\nabla u\|_{p(\cdot),\Omega_\mathcal{A}}^{p^{+}}
                           \right\}\\
                \ge& C \min \left\{
                           \|\nabla u\|_{p^-,\Omega_\mathcal{A} }^{p^{-}},
                           \|\nabla u\|_{p^-,\Omega_\mathcal{A}}^{p^{+}}
                           \right\},
            \end{aligned}
            \end{align}
        where $C$ is a positive constant independent of $u$. Then combining \eqref{41}-\eqref{43}, we obtain
            \begin{align}
            \begin{aligned}
              &\min \left\{
                           \|\nabla u\|_{p^-,\Omega_\mathcal{A}}^{p^{-}},
                           \|\nabla u\|_{p^-,\Omega_\mathcal{A}}^{p^{+}}
                           \right\}\\
              \leq& C \left( \|\nabla u\|_{p^-,\Omega_\mathcal{A}}^{\theta^{\prime}_{-} q^-}
                    \|u\|_{2,\Omega_\mathcal{A}}^{(1 - \theta^{\prime}_{-})q^-}
                + \|\nabla u\|_{p^-,\Omega_\mathcal{A}}^{\theta^{\prime}_{+} q^+}
                    \|u\|_{2,\Omega_\mathcal{A}}^{(1 - \theta^{\prime}_{-})q^+} \right),
            \end{aligned}
            \end{align}
       which implies that
            \begin{align}\label{7}
            \begin{aligned}
              \|u\|_{2,\Omega_\mathcal{A}}^{(1 - \theta^{\prime}_{-})q^-}
                    + \|u\|_{2,\Omega_\mathcal{A}}^{(1 - \theta^{\prime}_{-})q^+}
              \ge&
              \begin{cases}
                C\|\nabla u\|_{p^-,\Omega_\mathcal{A}}^{p^- - \theta^{\prime}_{+} q^+},
                    &  0 < \|\nabla u\|_{p^-} \le 1,\\
                C\|\nabla u\|_{p^-,\Omega_\mathcal{A}}^{p^+ - \theta^{\prime}_{-} q^-},
                    &  1< \|\nabla u\|_{p^-} < \infty.
              \end{cases}
            \end{aligned}
            \end{align}
       Recalling that $\|\nabla u\|_{p (\cdot),\mathcal{A}}$ is bounded both above and below for given $u \in \mathcal{N}_{k}$ according to Lemma \ref{l14} (ii) and Lemma \ref{l4}, thus the right hand side of \eqref{7} is bounded below, which implies that $\|u\|_{2,\Omega_\mathcal{A}} >0 $ with some positive constant $C$. Thus $\lambda_{k} \ge C> 0$ with the arbitrariness of $u$.
    \end{proof}

    Then, by the validity of Lemma \ref{l22} and Lemma \ref{l21}, we can establish the stable and unstable invariant manifolds for the solution generated by initial data with supercritical initial energy similarly as Lemma \ref{l15} and derive the abstract criterion for the global well-posedness of the solution to \eqref{186}-\eqref{188} as what we did for problem \eqref{1}-\eqref{3} in Theorem \ref{t5}, for which we give the following Theorem \ref{t7} and omit its proof.

    \begin{theorem}[Global existence and finite-time blow-up]\label{t7}
            Let $(\Omega_{R, r})$, $(\mathrm{\bar{H}_{p,q}})$, $(\mathrm{\bar{H}^1_{p,q}})$-$(\mathrm{\bar{H}^3_{p,q}})$ and $(\mathrm{\bar{H}_{D}})$ be satisfied. Suppose that $J(u_0) > d$, then we have

            \begin{enumerate}[(i)]
                \item if $\|u_0\|_2 \le \lambda_{J(u_0)}$ and $u_0 \in \mathcal{N}^+$, then
                          $u_{0} \in \mathcal{G}_0$;
                \item if $(\mathrm{\bar{H}^4_{p,q}})$ holds, $\|u_0\|_2 \ge \Lambda_{J(u_0)}$ and $u_0 \in \mathcal{N}^-$, then
                          $u_{0} \in \mathcal{B}$.
            \end{enumerate}
    \end{theorem}

Next, we give several remarks to discuss some possible generalizations, including the generalization for the domain replaced by the ball $B_R(0)$ of dimension two (i.e., $n=2$) or higher dimensions (i.e., $n\ge 3$) and the generalization for the constant exponent case.

\begin{remark}[Generalization for the ball with $n=2$]
  As the another typical bounded radially symmetric domain, one can also consider the generalization of our main results established in Sections \ref{sec3}-\ref{sec6} to the ball $B_{R}(0)$ with $n=2$, centered at the origin with the radius $R>0$.
  In this situation, recalling Remark \ref{rmk6}, similarly as the proof of Lemma \ref{l5}, one can deduce that
    \begin{align*}
      W_{0,\mathrm{rad}}^{1, p(\cdot)}(B_{R}(0)) \hookrightarrow L^{\infty}(B_{R}(0))
      \quad \text{continuously}
    \end{align*}
  and
    \begin{align*}
      W_{0,\mathrm{rad}}^{1, p(\cdot)}(B_{R}(0)) \hookrightarrow L^{q}(B_{R}(0))
      \quad \text{compactly}, \ q \in(1,\infty).
    \end{align*}
  Moreover, since $B_{R}(0)$ is radially symmetric and is invariant under the action of $\mathbf{O}(n)$, all of the main results established in Section \ref{sec3} and Section \ref{sec33} of the present paper
  can be directly generalized for \eqref{1}-\eqref{3} under the hypotheses $(\mathrm{H_{p,q}})$, $(\mathrm{H^1_{p,q}})$-$(\mathrm{H^2_{p,q}})$ and $(\mathrm{H_{D}})$ with the domain $\mathcal{A}$ replaced by $B_{R}(0)$ for $n=2$.
\end{remark}

\begin{remark}[Generalization for the ball with $n\ge 3$]
Still considering the generalization for the ball, the situation for the higher dimensional case $n\ge 3$ is different from the case $n=2$ due to the failure of the recovery of the compact Sobolev embedding with the Strauss inequality. Since \eqref{181} fails at the origin, we have to let $q(x)$ obey the subcritical regime setting in the region $B_r(0)$ ($0<r<R$) to secure the well-posedness of the solution as what we did in Section \ref{sec6} to set the hypothesis ($\mathrm{\bar{H}^2_{p,q}}$).
  Moreover, we point out that the consideration for problem \eqref{1}-\eqref{3}
  with the domain replaced by $B_{R}(0)$ for $n\ge 3$ can be regarded as the special case of $\Omega_{\mathcal{A}}$ satisfying ($\Omega_{R, r}$) that we set in Section \ref{sec6}, where $\Omega=B_{R}(0)$, $\Omega \backslash \mathcal{A} = \bar{B}_{r}(0)$ and $p(x),q(x)$ don't have to be radially symmetric on $B_r(0)$.

    \begin{figure}[h]
    \centering
    \begin{overpic}[width=6cm]{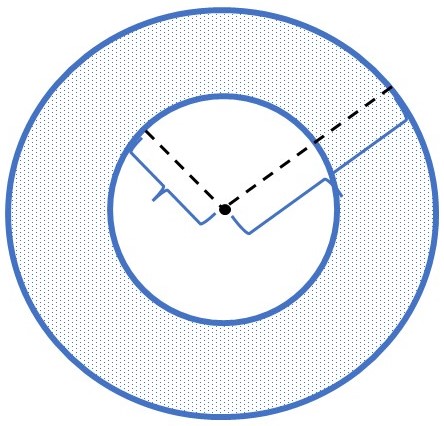}
            \put(83, 73){$0$}
            \put(130, 80){$R$}
            \put(54, 80){$r$}
            \put(20, 70){$\mathcal{A}$}
            \put(75, 50){$B_r(0)$}
            \put(150, 20){$B_R(0)$}
            \end{overpic}
    \caption{Generalization for the ball $B_R(0)$}\label{f4}
    \end{figure}
\end{remark}

\begin{remark}[The constant exponent case]
  The arguments of Section \ref{sec3} - Section \ref{sec33} in the present paper can be applied to the constant exponent case $p(x)\equiv p$ and $q(x)\equiv q$. In other words, the parabolic equations with $p$-Laplacian, i.e.,
    \begin{align}
          &u_t-\Delta_{p} u=|u|^{q-2} u && \text { in } \mathcal{A} \times(0,+\infty), \label{171}\\
          &u=0, && \text { on } \partial \mathcal{A} \times(0,+\infty), \label{172}\\
          &u(x, 0)=u_0(x), && \text { in } \mathcal{A}. \label{173}
    \end{align}
  for the critical regime case $q = p^*$ or the supercritical regime case $q > p^*$ can be regarded as the special case of \eqref{1}-\eqref{3} with $p^-=p^+$ and $q^- = q^+=p^*$ or $q^- = q^+>p^*$.
\end{remark}

\section*{Acknowledgments}
Runzhang Xu was supported by the Hainan Provincial Natural Science Foundation of China (126MS0251), the National Natural Science Foundation of China (12271122), and the Fundamental Research Funds for the Central Universities. In the process of this work (January 2022-October 2026), Runzhang Xu visited the Institute of Mathematical Sciences, The Chinese University of Hong Kong, on several occasions (June 9-21, 2025; August 25-27, 2025; December 21-24, 2025; January 15-27, 2026; February 20-March 1, 2026; April 21-May 3, 2026; and May 20-31, 2026), during which he benefited greatly from valuable discussions with Professor Zhouping Xin.

\section*{Declarations}
{\flushleft
\textbf{Conflicts of interest }
The authors declare no conflict of interest.

{\flushleft
\textbf{Data Availability }}
Data sharing is not applicable to this article as we do not analyse or generate any datasets in this study.


\begin{thebibliography}{99}


\bibitem{pyb4}	
{E. Acerbi, G. Mingione}, {Regularity results for a class of functionals with non-standard growth}, Arch. Ration. Mech. Anal., 156(2) (2001), 121-140.

\bibitem{pyb5}	
{E. Acerbi, G. Mingione}, {Regularity results for electrorheological fluids: the stationary case}, C. R. Math. Acad. Sci. Paris, 334(9) (2002), 817-822.

\bibitem{pyb6}	
{E. Acerbi, G. Mingione}, {Regularity results for stationary electro-rheological fluids}, Arch. Ration. Mech. Anal., 164(3) (2002), 213-259.

\bibitem{pyb7}
{F. Catt\'{e}, P. L. Lions, J. M. Morel, T. Coll}, {Image selective smoothing and edge detection by nonlinear diffusion}, SIAM J. Numer. Anal., 29(1) (1992), 845-866.

\bibitem{pyb8}
{L. \'{A}lvarez, P. L. Lions, J. M. Morel}, {Image selective smoothing and edge detection by nonlinear diffusion. II}, SIAM J. Numer. Anal., 29(3) (1992), 845-866.

\bibitem{pyb9}
{Y. Chen, S. Levine, M. Rao}, {Variable exponent, linear growth functionals in image restoration}, SIAM J. Appl. Math., 66(4) (2006), 1383-1406.

\bibitem{r1}
{C. O. Alves, T. Boudjeriou}, {Existence of solution for a class of heat equation involving the $p(x)$ Laplacian with triple regime}, Z. Angew. Math. Phys., 72 (2021), 1-18.

\bibitem{r2}
{L. C. Nhan, Q. V. Chuong, L. X. Truong}, {Potential well method for $p(x)$-Laplacian equations with variable exponent sources}, Nonlinear Anal. Theory Methods Appl., 56(1) (2020), 103155.

\bibitem{r4}
{L. Diening, P. Harjulehto, P. H\"{a}st\"{o}, M. Ruzicka}, {Lebesgue and Sobolev Spaces with Variable Exponents}, Springer, Berlin, 2011.

\bibitem{r6}
{V. K. Kalantarov, O. A. Ladyzhenskaya}, {The occurrence of collapse for quasilinear equations of parabolic and hyperbolic types}, Soviet Math., 10 (1978), 53-70.

\bibitem{r7}
{F. Gazzola, T. Weth}, {Finite time blow-up and global solutions for semilinear parabolic equations with initial data at high energy level}, Differ. Integral Equ., 18(9) (2005), 961-990.

\bibitem{r8}
{L. E. Payne, D. H. Sattinger}, {Saddle points and instability of nonlinear hyperbolic equations}, Israel J. Math., 22(3-4) (1975), 273-303.

\bibitem{r9}
{R. Z. Xu, J. Su}, {Global existence and finite time blow-up for a class of semilinear pseudo-parabolic equations}, J. Funct. Anal., 264 (2013), 2732-2763.

\bibitem{r10}
{W. A. Strauss}, {Existence of solitary waves in higher dimensions}, Comm. Math. Phys., 55 (1977), 149-162.

\bibitem{r11}
{D. V. Cruz-Uribe, A. Fiorenza}, {Variable Lebesgue Spaces: Foundations and Harmonic Analysis}, Springer Science $\&$ Business Media, Basel, 2013.

\bibitem{r13}
{S. Y. Chung, M. H. Choi}, {A new condition for the concavity method of blow-up solutions to $p$-Laplacian parabolic equations}, J. Differential Equations, 12 (2018), 6384-6399.

\bibitem{r18}
{K. Lee, A. Petrosyan, J. L. Vazquez}, {Large-time geometric properties of solutions of the evolution $p$-Laplacian equation}, J. Differential Equations, 229 (2006), 389-411.

\bibitem{r19}
{T. Kilpel\"{a}inen, P. Lindqvist}, {On the Dirichlet boundary value problem for a degenerate parabolic equation}, SIAM J. Math. Anal., 27 (1996), 661-683.

\bibitem{r22}
{M. Tsutsumi}, {Existence and nonexistence of global solutions for nonlinear parabolic equations}, Publ. Res. Inst. Math. Sci., 8 (1972), 211-229.

\bibitem{r23}
{H. Ishii}, {Asymptotic stability and blowing up of solutions of some nonlinear equations}, J. Differential Equations, 26 (1977), 291-319.

\bibitem{r24}
{G. Akagi}, {Local existence of solutions to some degenerate parabolic equation associated with the $p$-Laplacian}, J. Differential Equations, 241 (2007), 359-385.

\bibitem{r25}
{N. A. Dao, J. I. D\'{i}az}, {The extinction versus the blow-up: Global and non-global existence of solutions of source types of degenerate parabolic equations with a singular absorption}, J. Differential Equations, 263 (2017), 6764-6804.

\bibitem{r26}
{S. Lian, H. Yuan, C. Cao, W. Gao, X. Xu}, {On the Cauchy problem for the evolution $p$-Laplacian equations with gradient term and source}, J. Differential Equations, 235 (2007), 544-585.

\bibitem{r28}
{R. Ferreira, J. D. Rossi}, {Critical exponents for a semilinear parabolic equation with variable reaction}, Proc. Roy. Soc. Edinburgh Sect. A, 142(5) (2012), 1027-1042.

\bibitem{r30}
{J. P. Pinasco}, {Blow-up for parabolic and hyperbolic problems with variable exponents}, Nonlinear Analysis: Theory, Methods $\&$ Applications, 71(3) (2009), 1094-1099.

\bibitem{r31}
{Z. Zheng, S. Zhou}, {Blow-up of p-Laplacian evolution equations with variable source power}, Sci. China Math., 60(3) (2017), 469-490.

\bibitem{r32+}
{J. Arrieta, A. Carvalho}, {Abstract parabolic problems with critical nonlinearities and applications to Navier-Stokes and heat equations}, Trans. Amer. Math. Soc., 352(1) (2000), 285-310.

\bibitem{r32}
{J. M. Arrieta, A. Rodriguez-Bernal}, {Non well posedness of parabolic equations with supercritical nonlinearities}, Commun. Contemp. Math., 6(5) (2004), 733-764.

\bibitem{r33}
{X. Chen, J. S. Guo}, {Boundedness of global solutions of a supercritical parabolic equation}, Nonlinear Analysis: Theory, Methods $\&$ Applications, 68(3) (2008), 621-628.

\bibitem{r34}
{J. W. Cholewa, A. Rodriguez-Bernal}, {Critical and supercritical higher order parabolic problems in $\mathbb{R}^N$}, Nonlinear Analysis: Theory, Methods $\&$ Applications, 104 (2014), 50-74.

\bibitem{r35}
{T. Suzuki}, {Semilinear parabolic equation on bounded domain with critical Sobolev exponent}, Indiana Univ. Math. J., 57(7) (2008), 3365-3396.

\bibitem{r38}
{C. Zhang, L. S. Zhou}, {Renormalized and entropy solutions for nonlinear parabolic equations with variable exponents and $L^1$ data}, J. Differential Equations, 248 (2010), 1376-1400.

\bibitem{r39}
{Z. S. Lian, J. W. Gao, J. H. Yuan, L. C. Cao}, {Existence of solutions to an initial Dirichlet problem of evolutional $p(x)$-Laplace equations}, Ann. Inst. H. Poincar\'{e} Anal. Non Lin\'{e}aire, 29 (2012), 377-399.

\bibitem{r40}
{B. Guo, J. Y. Li, W. J. Gao}, {Singular phenomena of solutions for nonlinear diffusion equations involving $p(x)$-Laplace operator and nonlinear sources}, Z. Angew. Math. Phys., 66 (2015), 989-1005.


\bibitem{r47}
{N. Mizoguchi}, {On the behavior of solutions for a semilinear parabolic equation with supercritical nonlinearity}, Math. Z., 239(2) (2002), 215-229.

\bibitem{r48}
{R. Suzuki}, {Asymptotic behavior of solutions of quasilinear parabolic equations with supercritical nonlinearity}, J. Differential Equations, 190(1) (2003), 150-181.

\bibitem{r50}
E. Vitillaro, Global existence for the heat equation with nonlinear dynamical boundary conditions, Proc. Roy. Soc. Edinburgh Sect. A, 135(1) (2005), 175-207.

\bibitem{r51}
E. Vitillaro, On the Laplace equation with non-linear dynamical boundary conditions, Proc. London Math. Soc., 93(2) (2006), 418-446.


\bibitem{r53}
{G. Akagi, M. \^{O}tani}, {Evolution inclusions governed by the difference of two subdifferentials in reflexive Banach spaces}, J. Differential Equations, 209(2) (2005), 392-415.


\bibitem{r55}
{G. Akagi, M. \^{O}tani}, {Principle of symmetric criticality and evolution equations}, Discrete Contin. Dyn. Syst., 2003(Special) (2003), 1-10.

\bibitem{r56}
{J. Kobayashi, M. \^{O}tani}, {The principle of symmetric criticality for non-differentiable mappings}, J. Funct. Anal., 214(2) (2004), 428-449.

\bibitem{r57}
{R. S. Palais}, {The principle of symmetric criticality}, Commun. Math. Phys., 69(1) (1979), 19-30.

\bibitem{r58}
{M. Willem}, Minimax Theorem, Birk\"{a}user, Boston, 1996.


\bibitem{a1}
{C. Cowan, A. Moameni}, {Supercritical elliptic problems on nonradial domains via a nonsmooth variational approach}, J. Differential Equations, 341 (2022), 292-323.

\bibitem{a2}
{E. Serra}, {Non radial positive solutions for the H\'{e}non equation with critical growth}, Calc. Var. Partial Differential Equations, 23(3) (2005), 301-326.

\bibitem{a3}
{C. Bandle, L. A. Peletier}, {Nonlinear elliptic problems with critical exponent in shrinking annuli}, Math. Ann., 280(1) (1988), 1-19.

\bibitem{a4}
{A. Boscaggin, C. Francesca, N. Benedetta, T. Weth}, {A supercritical elliptic equation in the annulus}, Ann. Inst. H. Poincar\'{e} Anal. Non Lin\'{e}aire, 40(1) (2022), 157-183.

\bibitem{a5}
{D. Castorina, F. Pacella}, {Symmetry of positive solutions of an almost-critical problem in an annulus}, Calc. Var. Partial Differential Equations, 23(2) (2005), 125-138.

\bibitem{a6}
{M. Clapp, F. Pacella}, {Multiple solutions to the pure critical exponent problem in domains with a hole of arbitrary size}, Math. Z., 259(3) (2008), 575-589.

\bibitem{a7}
{D. G. De Figueiredo, O. H. Miyagaki}, {Multiplicity of non-radial solutions of critical elliptic problems in an annulus}, Proc. R. Soc. Edinb. Sect. A Math., 135(1) (2005), 25-37.

\bibitem{a8}
{Z. Q. Wang, M. Willem}, {Existence of many positive solutions of semilinear elliptic equations on an annulus}, Proc. Amer. Math. Soc., 127(6) (1999), 1711-1714.


\bibitem{p1}
Y. Fujishima, K. Ishige, Blow-up set for type I blowing up solutions for a semilinear heat equation, Ann. Inst. H. Poincar\'{e} Anal. Non Lin\'{e}aire, 31(2) (2014), 231-247.

\bibitem{p2}
M. Ishiwata, On the asymptotic behavior of unbounded radial solutions for semilinear parabolic problems involving critical Sobolev exponent, J. Differential Equations, 249(6) (2010), 1466-1482.

\bibitem{p4}
H. Matano, F. Merle, Classification of type I and type II behaviors for a supercritical nonlinear heat equation, J. Funct. Anal., 256 (2009), 992-1064.

\bibitem{p5}
H. Matano, F. Merle, Threshold and generic type I behaviors for a supercritical nonlinear heat equation, J. Funct. Anal., 261 (2011), 716-748.

\bibitem{p6}
P. Pol\'{a}\v{c}ik, Estimates of solutions and asymptotic symmetry for parabolic equations on bounded domains, Arch. Ration. Mech. Anal., 183(1) (2006), 59-91.

\bibitem{p7}
A. Salda\~{n}a, T. Weth, Asymptotic axial symmetry of solutions of parabolic equations in bounded radial domains, J. Evol. Equ., 12(3) (2012), 697-712.

\bibitem{p8}
Z. Tan,  Global solution and blowup of semilinear heat equation with critical sobolev exponent, Comm. Partial Diff. Equations, 26(3–4) (2001), 717–741.

\bibitem{e1}
H. Brezis, L. Nirenberg, Positive solutions of nonlinear elliptic equations involving critical sobolev exponents, Comm. Pure. Appl. Math., 36(4) (1983), 437–477.

\bibitem{p9}
W. M. Ni, P. E. Sacks, J. Tavantzis, On the asymptotic behavior of solutions of certain quasilinear parabolic equations, J. Differential Equations, 54(1) (1984), 97–120.

\bibitem{p10}
V. A. Galaktionov, J. L. Vazquez, Continuation of blowup solutions of nonlinear heat equations in several space dimensions, Comm. Pure Appl. Math., 50(1) (1997), 1–67.

\bibitem{p11}
{L. Damascelli}, {Comparison theorems for some quasilinear degenerate elliptic operators and applications to symmetry and monotonicity
results}, Ann. Inst. H. Poinca\'{r}e Anal. Non Li\'{n}eaire, 15 (1998), 493-516.

\bibitem{ap1}
W. Lian, J. Wang, R. Z. Xu, Global existence and blow up of solutions for pseudo-parabolic equation with singular potential, J. Differential Equations, 269(6) (2020), 4914–4959.

\bibitem{ap3}
R. Z. Xu, W. Lian, Y. Niu, Global well-posedness of coupled parabolic systems, Sci. China Math., 63 (2020), 321–356.

\bibitem{ap4}
Y. T. Wang, R. Z. Xu Xu, C. Yang, Global well-posedness for semilinear heat equations with linear or nonlinear boundary dissipation, Sci. China Math. (2026), https://doi.org/10.1007/s11425-025-2546-3.

\bibitem{aw1}
{Y. Liu, R.~Z. Xu}, {A class of fourth order wave equations with dissipative and nonlinear strain terms}, J. Differential Equations, 244(1) (2008), 200-228.

\bibitem{aw2}
{Y. Luo, R.~Z. Xu, C. Yang}, {Global well-posedness for a class of semilinear hyperbolic equations with singular potentials on manifolds with conical singularities}, Calc. Var. Partial Differential Equations, 61(6) (2022), 210.

\bibitem{aw3}	
{T. Pang, et al.}, {Global existence and non-existence for nonlinear wave equation with variable-exponent nonlinear weak damping term and power-type nonlinearity}, Commun. Math. Sci., 23(7) (2025), 1933-1982.

\bibitem{aw4}	
{Q. Lin, R. Z. Xu}, {Global well-posedness of the variable-order fractional wave equation with variable exponent nonlinearity}, J. Lond. Math. Soc. (2), 111(2) (2025), Paper No. e70091, 59 pp.


\end{thebibliography}
\end{document}